\documentclass[reqno,10pt]{amsart}
\usepackage{amscd,amssymb}
\usepackage{latexsym}
\usepackage[all]{xy}

\def\:{{\colon}}
\def\into{\hookrightarrow}

\def\toisom{\widetilde{\to}}

\def\.{,\dots ,}
\def\wt{\widetilde}
\def\wh{\widehat}
\def\ol{\overline}
\def\wtimes{\wh{\otimes}}
\def\circcirc{{\circ\circ}}

\def\Nr{\calN r}

\def\Spf{{\rm Spf}}

\def\Fr{{\rm Fr}}
\def\Gal{{\rm Gal}}

\def\Spec{{\rm Spec}}

\def\Frac{{\rm Frac}}
\def\bfSpec{{\bf Spec}}

\def\dim{{\rm dim}}

\def\ad{{\rm ad}}
\def\an{{\rm an}}

\def\Int{{\rm Int}}

\def\RZ{{\rm RZ}}

\def\rk{{\rm rk}}
\def\Hom{{\rm Hom}}

\def\sh{{\rm sh}}
\def\tr{{\rm tr}}
\def\deg{{\rm deg}}
\def\inv{{\rm inv}}

\def\bir{{\rm bir}}
\def\gp{{\rm gp}}
\def\tor{{\rm tor}}

\def\cha{{\rm char}}

\def\trdeg{{\rm tr.deg.}}

\def\bfA{{\bf A}}

\def\bfN{{\bf N}}

\def\bfP{{\bf P}}
\def\bfQ{{\bf Q}}
\def\bfR{{\bf R}}

\def\bfZ{{\bf Z}}

\def\gtC{{\mathfrak C}}

\def\gtS{{\mathfrak S}}

\def\gtV{{\mathfrak V}}
\def\gtW{{\mathfrak W}}
\def\gtX{{\mathfrak X}}
\def\gtY{{\mathfrak Y}}
\def\gtZ{{\mathfrak Z}}

\def\gtf{{\mathfrak f}}

\def\gtx{{\mathfrak x}}
\def\gty{{\mathfrak y}}
\def\gtz{{\mathfrak z}}

\def\calA{{\mathcal A}}
\def\calB{{\mathcal B}}
\def\calC{{\mathcal C}}

\def\calF{{\mathcal F}}

\def\calH{{\mathcal H}}

\def\calM{{\mathcal M}}
\def\calN{{\mathcal N}}
\def\calO{{\mathcal O}}

\def\calX{{\mathcal X}}
\def\calY{{\mathcal Y}}
\def\calZ{{\mathcal Z}}

\def\oC{{\ol C}}

\def\oK{{\ol K}}
\def\oL{{\ol L}}
\def\oM{{\ol M}}
\def\oN{{\ol N}}

\def\oP{{\ol P}}

\def\oX{{\ol X}}
\def\oY{{\ol Y}}
\def\oZ{{\ol Z}}

\def\of{{\ol f}}

\def\ok{{\ol k}}
\def\oll{{\ol l}}
\def\om{{\ol m}}

\def\ox{{\ol x}}
\def\oy{{\ol y}}
\def\oz{{\ol z}}

\def\tilB{{\wt B}}

\def\tilE{{\wt E}}
\def\tilF{{\wt F}}

\def\tilK{{\wt K}}
\def\tilL{{\wt L}}

\def\tilT{{\wt T}}

\def\tilY{{\wt Y}}

\def\tilb{{\wt b}}

\def\tilk{{\wt k}}
\def\till{{\wt l}}

\def\tilx{{\wt x}}

\def\hatA{{\wh A}}
\def\hatB{{\wh B}}

\def\hatD{{\wh D}}

\def\hatK{{\wh K}}

\def\hatX{{\wh X}}

\def\hatk{{\wh k}}
\def\hatl{{\wh l}}
\def\hatm{{\wh m}}

\def\hatx{{\wh x}}

\def\ocalO{{\ol\calO}}

\def\ogtC{{\ol\gtC}}

\def\ogtx{{\ol\gtx}}

\def\tilcalA{{\wt\calA}}

\def\tilgtx{{\wt\gtx}}

\def\hatcalO{{\wh\calO}}

\def\kcirc{k^\circ}
\def\lcirc{l^\circ}
\def\mcirc{m^\circ}

\def\Fcirc{F^\circ}

\def\Kcirc{K^\circ}
\def\Lcirc{L^\circ}

\def\tilFcirc{\tilF^\circ}

\def\hatkcirc{\hatk^\circ}
\def\hatlcirc{\hatl^\circ}

\def\calAcirc{\calA^\circ}
\def\calBcirc{\calB^\circ}

\def\okcirc{\ok^\circ}
\def\olcirc{\oll^\circ}
\def\omcirc{\om^\circ}

\def\Kcirccirc{K^{\circ\circ}}

\def\kcirccirc{k^{\circ\circ}}

\def\hatKcirccirc{\hatK^\circcirc}

\def\calAcirccirc{\calA^\circcirc}

\def\oeta{{\ol\eta}}

\def\tilalpha{{\wt\alpha}}

\def\alp{{\alpha}}
\def\lam{{\lambda}}
\def\Lam{{\Lambda}}

\def\oLam{{\ol\Lam}}
\def\oLamcirc{{\oLam^\circ}}
\def\Lamcirc{{\Lam^\circ}}

\def\wHx{{\widetilde{\calH(x)}}}
\def\wHy{{\widetilde{\calH(y)}}}

\def\whka{{\wh{k^a}}}

\def\R+*{{\bf R^*_+}}

\newtheorem{theorsect}{Theorem}[section]
\newtheorem{propsect}[theorsect]{Proposition}

\newtheorem{theor}{Theorem}[subsection]
\newtheorem{prop}[theor]{Proposition}
\newtheorem{lem}[theor]{Lemma}
\newtheorem{cor}[theor]{Corollary}
\newtheorem{conj}[theor]{Conjecture}

\theoremstyle{definition}

\newtheorem{examsect}[theorsect]{Example}

\newtheorem{defin}[theor]{Definition}
\newtheorem{rem}[theor]{Remark}
\newtheorem{exam}[theor]{Example}

\newtheorem{sit}[theor]{Situation}

\begin{document}

\title{Inseparable local uniformization}
\author{Michael Temkin}
\address{\tiny{Einstein Institute of Mathematics, The Hebrew University of Jerusalem, Giv'at Ram, Jerusalem, 91904, Israel}}
\email{\scriptsize{temkin@math.huji.ac.il}}
\thanks{Parts of this work were done when I was visiting Max Planck Institute for Mathematics at Bonn and the Institute for Advanced Study at Princeton, and I am thankful to them for the hospitality. In IAS I was supported by NSF grant DMS-0635607.}

\begin{abstract}
It is known since the works of Zariski in the early 40ies that desingularization of varieties along valuations (called local uniformization of valuations) can be considered as the local part of the desingularization problem. It is still an open problem if local uniformization exists in positive characteristic and dimension larger than three. In this paper, we prove that Zariski local uniformization of algebraic varieties is always possible after a purely inseparable extension of the field of rational functions, and therefore any valuation can be uniformized by a purely inseparable alteration.
\end{abstract}


\maketitle

\section{Introduction}

\subsection{Motivation}
The main goal of this paper is to prove that an integral algebraic variety over a field can be desingularized locally along a valuation by a purely inseparable alteration. In view of analogies with (resp. weak) local uniformization due to Zariski (resp. Gabber) it is natural to call this result inseparable local uniformization of valuations on varieties. An equivalent reformulation of our main result is that any integral algebraic variety $X$ can be {\em covered} by integral regular $X$-schemes $Y_1\. Y_m$ such that each morphism $Y_i\to X$ is dominant, of finite type and the extensions $k(Y_i)/k(X)$ of the fields of rational functions are finite and purely inseparable. As for the definition of the covering, we prefer the following ad hoc definition: $\coprod_{i=1}^m Y_i\to X$ as above is a covering if any valuation on $k(X)$ with center on $X$ lifts to a valuation on some $k(Y_i)$ with center on $Y_i$.

To achieve our main goal, we will study inseparable local uniformization of certain points on Berkovich analytic spaces and of certain valuations on curves over valuation rings. These are secondary goals of the paper, and, in order to bound the length of the paper, we prefer not to explore them beyond what is needed for the proof of the main result. It seems that these questions are worth a deeper study in a separate paper. For example, it is an interesting question if analogous results hold for other classes of analytic points.

Finally, one more secondary goal of the paper is to enrich the classical techniques of desingularization theory with new tools. Probably, the main novelty is the use of Berkovich analytic geometry, which plays a critical role in our proof. In addition, we make heavy use of non-noetherian schemes (related to valuation rings) and the approximation theory (or the theory of projective limits) from \cite[$\rm IV_3$, \S8]{ega}. The former theory has just started to find applications to algebraic geometry (see \cite{temst}, \cite{Ked1} and \cite{Ked2}), while the latter is relatively common in general algebraic geometry but seems to be a new tool in desingularization theory.

\subsection{Known desingularization results: strength versus generality}
To put our result into a general context of desingularization theory we observe that a general aim of desingularizing an integral scheme $X$ is to find a morphism $f{\colon}Y\to X$ such that $Y$ is regular, $f$ is a covering in a natural topology (usually $h$-topology) and $Y$ is as "close" to $X$ as possible. Traditionally, one seeks for a proper and birational $f$ but nowadays other choices are widely used. Let $Y_1\. Y_m$ denote the irreducible components of $Y$ with $K_i=k(Y_i)$. Almost always, one at least requires from $f$ that its restriction on each $Y_i$ is separated, of finite type, dominant and generically finite; in particular, $K_i$ is finite over $K=k(X)$. In this case, we say as earlier that $f$ is a covering if any valuation on $X$ lifts to a valuation on some $Y_i$. Although we will not need that, we remark that the flattening theorem of Raynaud-Gruson implies that the topology of such coverings is nothing else but the topology generated by modifications and flat quasi-finite coverings, and that $f$ is a covering if and only if it is a covering in the $h$-topology of Voevodsky, see \cite{Vo}. Since $f$ should be as "small" as possible, usually one tries to control $m$ and the extensions $K_i/K$, though it is not always possible with concurrent methods. Our result provides a partial control on $K_i/K$, in particular, it implies that $[K_i:K]=p^n$, where $p$ is the characteristic. For the sake of comparison, we briefly describe other known results.

(i) {\it Classical desingularization: $m=1$ and $K_1=K$.} Under these assumptions, $f$ is automatically birational and proper. This case was established by Hironaka for schemes of finite type over local quasi-excellent schemes over $\bfQ$, see \cite{Hir}. Moreover, it is achieved by blowing up regular centers, so one obtains $f$ of a very special form. It was later proved that for varieties one can build $f$ functorially, see, for example, \cite{bm}. The case of general quasi-excellent schemes over $\bfQ$ was deduced in \cite{temdes} and \cite{nonemb}. In positive characteristic, the case of threefolds over a field $k$ with $[k:k^p]<\infty$ was established recently by Cossart and Piltant in \cite{CP1} and \cite{CP2}. The main ingredient of their proof is local uniformization of threefolds. For general quasi-excellent schemes it is only known how to desingularize surfaces.

(ii) {\it Local uniformization: $K=K_1=\dots=K_m$.} The problem  was introduced by Zariski, who named it local uniformization and considered it a local part of the classical desingularization problem. Zariski established in \cite{Zar1} the case of varieties of characteristic zero, and deduced global desingularization for threefolds (it is unknown if one can reduce global desingularization to local uniformization in higher dimensions). In positive characteristic, the only known proof for threefolds is very complicated and has a minor restriction that $[k:k^p]<\infty$ (see \cite{CP1} and \cite{CP2}), and the case of $\dim(X)>3$ is widely open (but see remark (i) on the next page).

(iii) {\it Alterations: $m=1$ and $f$ is proper.} Such $f$ is called an {\em alteration}. This very successful weakening of the classical desingularization problem was introduced by de Jong in \cite{dJ1}. The new problem can be solved with reasonable effort for any scheme of finite type over an excellent surface $S$, but it can replace the classical desingularization in many applications. In addition, de Jong proved that if the ground scheme $S$ is the spectrum of a perfect field then $K_1/K$ can be chosen to be separable. The only other known result on control on $K_1/K$ was recently announced by Gabber: if $S=\Spec(k)$ for a field $k$ then one can choose $f$ so that $[K_1:K]$ is prime to a given prime $l\neq\cha(k)$; see a survey on Gabber's work by Illusie, \cite[Th. 1.3]{Ill}.

(iv) {\it Altered local uniformization of Gabber: no restrictions on $m$ and $K_i$'s.} Gabber proved that weak local uniformization exists for any quasi-excellent scheme (which is a much more general case than usual methods treat). This result plays a key role in Gabber's proof of a fundamental finiteness theorem for \'etale cohomology of general quasi-excellent schemes. Moreover, in order to control $l$-torsion coefficients Gabber proved a prime-to-$l$ strengthening of the weak local uniformization whose precise formulation is given in \cite[Th. 1.1]{Ill}.

(v) {\it Inseparable local uniformization: $K_i/K$ are purely inseparable.} In the case of varieties, this is our Corollary \ref{maincor}.

Our list of known results would not be complete without a discussion on recent works in progress. The author can only express his own expectations that may be completely wrong. Perhaps, one can divide these works to three classes.

(i) Pushing existing techniques to their limit. I expect that the following two problems can be solved in this way: (a) extend the method of \cite{CP1} to desingularize any scheme of dimension $3$ that admits a morphism of finite type to an excellent curve (e.g. to $\Spec(\bfZ)$), (b) extend the methods of this paper to prove simultaneous inseparable log uniformization (see \S\ref{strongsec} below) of quasi-excellent schemes of positive characteristic. Also, extend this to mixed characteristic with inseparable alteration replaced by an alteration of degree $p^n$,
where $p$ is the residue characteristic of the valuation.

(ii) A couple of years ago programs on full resolution of singularities were announced independently by Kawanoue, Hironaka and W{\l}odarczyk. Also, Villamayor and his coauthors develop a new approach to positive characteristic in a couple of papers (without claiming to have a full program). These projects are not completed so far and it seems that nobody can predict how they will develop.

(iii) Recently, T. Urabe claimed a full proof of local uniformization by toric methods, see \cite{Ura}. The preprint has not been accepted for publication so far, and in private communication with the author some experts doubted the proof.

\subsection{The main result}
\begin{conj}\label{insepconj}
Let $X$ be an integral algebraic variety. Then there exists an alteration $f{\colon}Y\to X$ with regular $Y$ and a purely inseparable extension $k(Y)/k(X)$.
\end{conj}

This is conjecture \cite[2.9]{AO}, and it expresses a hope that such control on the extension of fields of rational functions may be substantially easier to achieve than classical desingularization. The author shares this hope despite the fact that the conjecture is widely open so far. Our main result is its local version along a valuation. We formulate this result in Theorem \ref{insepunif} below and call it inseparable local uniformization. Given a finitely generated field extension $K/k$ and a valuation ring $\Kcirc\supset k$ of $K$ (i.e. $K=\Frac(\Kcirc)$), by a {\em $k$-model} of $\Kcirc$ we mean any integral $k$-variety $X$ with generic point $\Spec(K)\to X$ such that $\Kcirc$ is centered on $X$. In particular, an affine model is given by a finitely generated $k$-subalgebra $A\subset\Kcirc$ with $\Frac(A)=K$. As usual, by saying that a model $X'$ {\em refines} $X$ we mean that the isomorphism of their generic points extends to a morphism $f{\colon}X'\to X$.

\begin{theor}\label{insepunif}
Let $K/k$ be a finitely generated field extension, $\Kcirc$ be a valuation ring of $K$ containing $k$ and $X$ be an affine $k$-model of $\Kcirc$. Then there exist finite purely inseparable extensions $l/k$ and $L/lK$ and an affine model $X'$ of $\Kcirc$ such that $X'$ refines $X$ and the unique extension of $\Kcirc$ to a valuation ring of $L$ is centered on a simple $l$-smooth point of the $L$-normalization $\Nr_L(X')$.
\end{theor}

Here $\Nr_L(\Spec(A))$ is the scheme $\Spec(\Nr_L(A))$ where $\Nr_L(A)$ is the integral closure of $A$ in $L$. Recall also that a smooth point $x$ on an $l$-variety is called {\em simple} if $k(x)$ is separable over $l$. By quasi-compactness of the Riemann-Zariski space of valuations centered on an algebraic variety, see \S\ref{ebfibsec}, the theorem implies the following corollary, which is another form of inseparable local uniformization.

\begin{cor}\label{maincor}
Let $X$ be an integral algebraic variety. Then there exists a covering $f{\colon}\coprod_{i=1}^mY_i\to X$ such that each $Y_i$ is integral and regular and the induced extensions $k(Y_i)/k(X)$ are finite and purely inseparable.
\end{cor}

Let us discuss possible reformulations of our result and its relation to the local uniformization.

\begin{rem}\label{projrem}
We use affine models in Theorem \ref{insepunif} because the problem is of local nature, and so our formulation seems to be the most natural one. One easily sees that our formulation implies (and hence is equivalent to) the more traditional version where $X$ is assumed to be proper and one requires $X'$ to be $k$-projective (first refine $X$ so that it becomes projective and then find an affine $X'$ as in the theorem and replace the latter with its $X$-projective compactification). Similarly, one can achieve in addition that $f{\colon}X'\to X$ is a blow up.
\end{rem}

\begin{rem}\label{onemorerem}
(i) Without loss of generality, $X'$ is normal. Then taking $n$ so that $L^{p^n}\subset K$ and using the Frobenius isomorphism $F^n{\colon}X'\toisom \Nr_{K^{1/p^n}}(X')$ we obtain an integral purely inseparable morphism of schemes $h{\colon}X'\to\Nr_L(X')$ which maps the center of $\Kcirc$ to a regular point. (Throughout this paper, {\em integral} morphism always means a morphism of the form $\bfSpec(A)\to X$ where $A$ is an $\calO_X$-algebra which is integral over $\calO_X$.) Moreover, if $[k:k^p]<\infty$ then $h$ is finite.

(ii) The observation from (i) can be sharpened as follows. Assume that $[k:k^p]$ is finite. Then there exists a tower $K=K_m\supset\dots\supset K_0=L^{p^n}$ such that each $K_i=K_{i-1}(a_i^{1/p})$ is purely inseparable of degree $p$ over $K_{i-1}$. Set $\Kcirc_i=\Kcirc\cap K_i$. By (i), $\Kcirc_0$ is locally uniformized by a regular scheme $X_0=\Spec(A_0)$ isomorphic to $\Nr_L(X')$. Multiplying $a_1$ by an appropriate $p$-th power we can achieve that $a_1\in A_0$, and then $\Kcirc_1$ is centered on the model $A_1=A_0[t]/(t^p-a_1)$ of $K_1$. If we know how to uniformize valuations on $\alp_p$-torsors over regular schemes, then we can uniformize $\Kcirc_1$, and proceeding inductively to $\Kcirc_2$, etc., we would uniformize the original $\Kcirc$.

(iii) Thus, Theorem \ref{insepunif} implies that local uniformization would follow from local uniformization of hypersurfaces in $\bfA^{d+1}$ given by equations of the form $t^p=f(x_1\. x_d)$. The latter case is often called the inseparable case, and it was always recognized as an important test case for desingularization methods, where all "bad things" can happen. However, the inseparable case was not viewed as the general case.

(iv) For example, Cossart and Piltant in their proof of local uniformization of threefolds had to study singularities of the form $t^p+g(x_1,x_2,x_3)^{p-1}t+f(x_1,x_2,x_3)=0$, which they call Artin-Shreier case for $g\neq 0$ and inseparable case for $g=0$. Moreover, the proof of the Artin-Shreier case required a little bit more work in \cite{CP2}.
\end{rem}

\subsection{Stronger forms of local uniformization}\label{strongsec}
For inductive purposes we will have to prove in some cases stronger variants of inseparable local uniformization, see Theorems \ref{equivunif} and \ref{Abhth}. So, let us outline what kinds of generalizations we will need. For simplicity, we discuss analogous generalizations of the usual local uniformization of a valued field $K$. By {\em descent local uniformization} of $\Kcirc$ we mean solving the following problem: given a valued subfield with $[K:L]<\infty$ and an affine model $Y$ of $\Lcirc$, find an affine refinement $Y'\to Y$ such that $K$ is centered on a regular point of $\Nr_K(Y')$. Note that if $K/L$
is Galois with $G=\Gal(K/L)$ then this is equivalent to the more standard problem of finding a $G$-equivariant local uniformization of $K$ that refines $\Nr_K(Y)$. It is also natural to ask whether one can achieve in addition that the center of $L$ on $Y'$ is regular. The latter problem is known as (classical) simultaneous local uniformization of $K$ and $L$. More generally, if $K_1\. K_n$ are finite valued extensions of $L$ then by {\em simultaneous local uniformization} of $K_i$'s we mean a refinement $Y'\to Y$ such that each $K_i$ is centered on a regular point of $\Nr_{K_i}(Y')$. A simple toric example of Abhyankar shows that even classical simultaneous local uniformization is impossible in general. However, one can hope that it is always possible to obtain a simultaneous log uniformization, where $K_i$'s are uniformized by log smooth (or toroidal) points. At least, we will prove this for Abhyankar valuations and we will establish in Theorem \ref{equivunif} simultaneous inseparable log uniformization for all valuations of height one.

\subsection{Overview}\label{oversec}
Very roughly speaking, the proof of Theorem \ref{insepunif} runs as follows. Similarly to de Jong's approach, the initial idea is to fiber varieties by curves and prove the theorem by induction on the dimension. We postpone establishing the base of the induction until \S\ref{simulsec}. The induction step is deduced in \S\ref{insepunifsec} from inseparable local uniformization of certain valuations on curves over valuation rings. The latter is proved in \S\ref{decomplsec} via a decompletion procedure, and the main ingredient of its proof is inseparable local uniformization of certain points (called terminal) on non-Archimedean analytic curves (see \S\ref{termsec}). Now, let us describe our method, the above intermediate results, and the organization of the paper in more details.

In \S\ref{chapone} we prove some results on Riemann-Zariski spaces and schemes over valuation rings with their analytifications. This section is very technical because we have to work with non-noetherian schemes and their non-finite normalizations. In order to ease the exposition we prefer to sacrifice generality to some extent. In some cases we prove what we need and possible generalizations are mentioned in remarks. We introduce valued fields in \S\ref{valsec}. Since schemes of finite type over valuation rings may have non-finite normalization, we introduce morphisms of normalized finite type and study their compatibility with projective limits in \S\S\ref{normsec}-\ref{etanormsec}. In \S\S\ref{ebfibsec}-\ref{bircritsec} we use Riemann-Zariski spaces to prove a birational criterion \ref{etaleprop} for a morphism of normal schemes to be \'etale. For schemes of normalized finite type over a valuation ring of height one we define analytic generic fibers in \S\ref{agfsec}, and in the next section we prove the main result of \S\ref{chapone}, Theorem \ref{anfibth}, which gives an analytic criterion for a morphism between such schemes to be strictly \'etale at a point. In a very natural way, the criterion states that $f$ should induce an isomorphism of the corresponding analytic fibers, but the proof is not easy since it is based on many results from \S\S\ref{valsec}-\ref{agfsec}. Finally, in \S\ref{smsec} we apply Theorem \ref{anfibth} to study equivalence of points in the smooth topology. We show that smooth-equivalence descends from projective limits and prove an analytic criterion \ref{smprop} for a point $x$ on a scheme $X$ of normalized finite type over a valuation ring $\kcirc$ to be smooth-equivalent to the closed point of the spectrum of a larger valuation ring $\lcirc$. Note that it is very important to cover the case of non-discrete valuations with a ramified extension $l/k$, and that in this case $X$ is not of finite presentation locally at $x$ because $\lcirc$ is not finitely generated over $\kcirc$ (see also Remark \ref{rem2}). This explains why we have to work in the unusual generality of morphisms of normalized finite type.

The first two sections of \S\ref{chaptwo} are devoted to local uniformization of a $k$-analytic curve $C^\an$ over a perfect analytic field $k$ of positive characteristic. Theorem \ref{terminalth} states that any so-called {\em terminal} point of $C^\an$ (i.e. type $1$ or type $4$ point in Berkovich's classification) lies in an $m$-split disc for a finite extension $m/k$. Note that the proof of this theorem is ultimately based on a difficult Theorem \cite[6.3.1]{temst}, where one-dimensional extensions of perfect analytic fields are studied. Theorem \ref{4th} generalizes \ref{terminalth} to any $k$, but then an $m$-split disc exists only after a preliminary purely inseparable extension $l/k$ of the ground field. This is the inseparable local uniformization of terminal points on Berkovich curves that we mentioned earlier. Finally, we use a decompletion procedure to prove Theorem \ref{dim1unif} stating that certain valuations on a curve $C$ over a valuation ring
$\kcirc$ of height one have uniformizations with centers that are smooth-equivalent to the closed point of $\Spec(\mcirc)$ for a larger valuation ring $\mcirc$. The theorem only applies to valuations with transcendence defect over $k$ (see \S\ref{valsec} for the terminology on valued fields), that is, for valuations corresponding to terminal points on $C^\an$.

\begin{rem}\label{basarem}
I do not know if a similar uniformization result holds for other valuations on $C$. This question seems to be worth an additional study.
\end{rem}

We prove Theorem \ref{insepunif} in the two last sections. First, we establish the induction step in \S\ref{insepunifsec}. We deal in \S\ref{honesec} with the case when $K$ is of height one. As usually is the case with local uniformization, the main difficulty is met already when the height is one, and our case is not an exception. Our proof uses induction on the dimension (i.e. $\trdeg_k(K)$) and is based on Theorem \ref{dim1unif}, so it applies only to transcendentally immediate one-dimensional extensions $K/\ok$ of valued fields of height one. In particular, when dealing in \S\ref{simulsec} with the induction base we should work with a general valued field $K$ which is of height one and Abhyankar over $k$. The main difficulty in the inductive proof of Theorem \ref{insepunif} comes from non-henselianity of the valued field $K$ and it will be discussed in Remark \ref{indrem}.
This difficulty forces us to strengthen the induction assumption when dealing with the height one case. The "minimal" packet that can be proved inductively is a descent version,
but we prefer to establish the full simultaneous inseparable local log uniformization for height one valued fields, see Theorem \ref{equivunif}.

\begin{rem}
It seems certain that simultaneous inseparable local log uniformization holds for valuations of any height, but proving this would involve Berkovich log geometry (or at least working with toroidal analytic spaces). This direction is not developed in the paper, and so we lose the stronger versions of the uniformization when running induction on height in \S\ref{indsec}. So, for general valued fields we only prove Theorem \ref{insepunif} without "bonuses".
\end{rem}

\begin{rem}
(i) As we explained above, the induction argument used in the case of height one valuations is more complicated than a direct induction on the dimension: it runs by induction on the transcendence defect and uses the case of zero transcendence defect as its base, which requires a separate proof. In particular, it is subtler than in de Jong's and Gabber's methods, where defect does not show up. A similar induction scheme was already used in \cite{KK2} to establish a certain form of altered local uniformization.

(ii) Despite its relative novelty (to the best of my knowledge), this induction scheme is very natural because it is well known that "complexity" of the valuation grows with the transcendence defect $D$ and is adequately measured by it. Note also that it appeared in the recent works \cite{Ked1} and \cite{Ked2} of Kedlaya. An interesting common feature of these works and the current paper is that the induction step is done by working with Berkovich analytic discs.
\end{rem}

Finally, in \S\ref{simulsec} we deal with Abhyankar valuations, thereby establishing the induction base in Theorem \ref{insepunif}. Unlike general valuations, Abhyankar ones can be fruitfully studied by the methods of log geometry (or toroidal geometry). In particular, one can even locally uniformize them, as was proved in \cite{KK1}. Unfortunately, this does not cover the descent version of inseparable local uniformization, and it is even unclear if we can use \cite{KK1} as an intermediate step.\footnote{It was communicated to me by F.-V. Kuhlmann, that the first version of \cite{KK1} contained a weak form of simultaneous local uniformization, which was removed due to referee's request.} Therefore, in \S\ref{simulsec} we study Abhyankar valuations "from scratch". It seems that the claim we actually need is not essentially simpler then the full simultaneous local log uniformization of Abhyankar valuations. So we prefer not to restrict the generality at this place, and the latter is our main result on Abhyankar valuations, see Theorem \ref{Abhth}. Note that only basic logarithmic geometry is used in our proof, so we reprove and generalize the main result of \cite{KK1}. The paper contains Appendix \ref{monoidsec}, where we recall some results on monoids which are used in \S\ref{simulsec}, and Appendix \ref{etalesec}, in which we discuss local-\'etale morphisms.

\begin{rem}
It turned out that the results we prove here on Abhyankar valuations are very important for the study of skeletons of analytic spaces and Riemann-Zariski spaces. This direction has nothing to do with desingularization and will be studied in a separate paper.
\end{rem}

We conclude the Introduction with the remark that since Theorem \ref{insepunif} is established, it is very challenging to attack the inseparable desingularization conjecture \ref{insepconj}. It seems very unlikely that our method as it is can be globalized to give an a-la de Jong proof of the conjecture. The problem is that for any specific valuation we have to choose an appropriate sequence of curve fibrations in order not to be stuck with the problem described in Remark \ref{basarem}, so no global fibration suits all valuations simultaneously. The author nevertheless hopes that inseparable local uniformization can be useful in attacking the conjecture.

\subsection{Acknowledgments}
I am indebted to the anonymous referee for a fantastic review job. In addition to pointing out various gaps and inaccuracies and making numerous suggestions on improving the presentation, he suggested simplifications/corrections to the proofs of Lemmas \ref{smetlem}(i), \ref{birfiblem} and Theorem \ref{etaleth}, that are included in the revised version. Appendix \ref{etalesec} is entirely due to the referee and B. Conrad. I am very grateful to D. Rydh for making many valuable comments on section 2 of the paper and to D. Abramovich for outlining the proof of Lemma \ref{descentlem}. In addition, I wish to thank M. Spivakovsky for useful discussions, F. Pop for discussions and encouragement and L. Illusie for interest to this work and informing me about Gabber's work.

\tableofcontents

\section{On schemes over valuation rings}\label{chapone}

\subsection{Valued fields}\label{valsec}
The aim of this section is to recall some facts about valued fields and to fix our terminology. The reader may also wish to consult \cite[\S2.1]{temst}, where a more detailed review is given. By a {\em valued field} $k$ we mean a field provided with a valuation ring which will be denoted $\kcirc$. Alternatively, this information can be given by an equivalence class of valuations (or absolute values) $|\ |{\colon}k^\times\to\Gamma$ with values in an ordered commutative "multiplicative" group. Here and in the sequel, by writing a "multiplicative" group or lattice we mean that they are written in the multiplicative notation $(1,x\mapsto x^{-1},(x,y)\mapsto xy)$. Also, we automatically assume in the sequel that all groups considered in the paper, excluding groups that arise as Galois groups, are commutative.

The ordered group $|k^\times|$ is well defined up to an isomorphism, and the {\em
height} (or rank) of $k$ is the height of $|k^\times|$, that is the number of its non-trivial convex subgroups. It is easy to see that the height of $k$ equals to the Krull-dimension of $\kcirc$. We remark that it is convenient not to fix $\Gamma$ by requiring that $|k^\times|=\Gamma$. For example, $k$ is of height one if and only if $|k^\times|$ admits an ordered embedding into $\bfR_+^\times$, and it is often the
most natural choice to take $\Gamma=\bfR_+^\times$. Let $\kcirccirc$ denote the maximal ideal of $\kcirc$ and let $\tilk=\kcirc/\kcirccirc$ denote the residue field. If $k$ is of height one then we will use the letter $\pi$ to denote a non-zero element from $\kcirccirc$ and we will denote the $(\pi)$-adic completion of $\kcirc$ by $\hatkcirc$ and the completion of $k$ by $\hatk$. Note that $\hatk=\Frac(\hatkcirc)=(\hatkcirc)_\pi$. We say that $k$ is {\em analytic} if it is complete and $\Gamma=\bfR_+^\times$.

By {\em extension} $l/k$ of valued fields we mean an inclusion $k\into l$ which respects the valuations in the sense that $\lcirc\cap k=\kcirc$. If $n=[l:k]$ is finite then it is standard to introduce the numbers $e=e_{l/k}=\#|l^\times|/|k^\times|$ and $f=f_{l/k}=[\till:\tilk]$, and the extension is called {\em immediate} if $ef=1$, i.e. $l$ and $k$ have the same residue fields and value groups. An easy classical result states that $ef\le n$. Moreover, if the valuation of $k$ extends uniquely to $l$ (for example, this is the case when $k$ is analytic) then $ef$ divides $n$ and the number $d=d_{l/k}=n/(ef)$ is called the {\em defect} of the extension. The defect is always a power of $p=\cha(\tilk)$ (this and many other statements in the paper make sense for exponential characteristic, i.e. $p=0$ should be replaced with $p=1$; usually we will not remark when $p=1$ should be used, since this will always be obvious), and if $d=1$ then we say that the extension is {\em defectless}. If, more generally, the valuation of $k$ admits $m$ extensions to $l$ and $e_1\. e_m,f_1\. f_m$ are the corresponding invariants of the extensions of valued fields then $e_1f_1+\dots+e_mf_m\le n$ and the extension is called {\em defectless} when equality holds. A valued field $k$ is called {\em stable} if any finite extension is defectless. For the sake of completeness we discuss briefly how one can define defect numbers in general, though this will not be used in the sequel.

\begin{rem}
\label{defrem} There is a natural one-to-one correspondence between (a) extensions of the valuation on $k$ to $l$, (b) maximal ideals of the integral closure of $\kcirc$ in $l$, and (c) the valued fields $l_i$ over the henselization $k^h$ of $k$ (i.e. $k^h$ is the fraction field of the henselization of $\kcirc$) such that $k^h\otimes_k l=\prod_{i=1}^m l_i$. So, one can define $n_i=[l_i:k^h]$ and $d_i=n_i/(e_if_i)$. Obviously, $e_1f_1d_1+\dots+e_mf_md_m=n$ and it is not difficult to prove that $d_i\in p^\bfN$. Note that a similar definition of henselian degrees $n_i$ is used in \S\ref{bircritsec}, where we study the more general class of unibranch local rings.
\end{rem}

Since we will have to work with infinite extensions of valued fields, it seems natural to also introduce the following invariants: for any extension $l/k$ of valued fields set $E=E_{l/k}=\dim_\bfQ(|l^\times|/|k^\times|\otimes_\bfZ\bfQ)$ and $F=F_{l/k}=\trdeg_\tilk(\till)$. Sometimes these cardinals are called the rational rank and the dimension, respectively. We say that the extension is {\em transcendentally immediate} if $E=F=0$, i.e. $\till/\tilk$ is algebraic and $|l^\times|/|k^\times|$ is torsion. For a general $l/k$, let $B\subset l^\times$ be a subset such that the following condition is satisfied: (*) $B=B_E\coprod B_F$, $|b|=1$ for any $b\in B_F$ and the reduction maps $B_F$ bijectively onto a transcendence basis $\tilB_F$ of $\till$ over $\tilk$, and the projection of $l^\times$ onto the "multiplicative" $\bfQ$-vector space $(|l^\times|/|k^\times|)\otimes_\bfZ\bfQ$ maps $B_E$ bijectively onto a $\bfQ$-basis. We omit a rather straightforward check that the elements of $B$ are algebraically independent over $k$ (see, for example, \cite[4.8]{ctdescent}, where it is proved that the graded reduction of $B$ is a transcendence basis of the graded reduction of $l$ over that of $k$). It follows, in particular, that $E+F$ cannot exceed $N=\trdeg_k(l)$, and when $N$ is finite we define the {\em transcendence defect} $D=D_{l/k}=N-E-F$. If $l/k$ admits a transcendence basis $B$ that satisfies (*) then we say that $l/k$ is {\em Abhyankar} (or {\em transcendentally defectless}) and $B$ is an {\em Abhyankar transcendence basis}. Note that for a finite $N$ the extension is Abhyankar if and only if $D_{l/k}=0$, and then any $B$ satisfying (*) is an Abhyankar transcendence basis.

\begin{rem}\label{Abhrem}
Choose any $B=B_E\coprod B_F$ that satisfies (*). Then the extension $l/k$ splits to a tower $l/k(B)/k$ with Abhyankar bottom level and transcendentally immediate top level. In particular, one can define $D_{l/k}$ for a general extension $l/k$ as $\trdeg_{k(B)}(l)$, and this agrees with the above definition when $\trdeg_k(l)<\infty$.
\end{rem}

\begin{rem}\label{Abhrem2}
Let $l/k$ be a finitely generated Abhyankar extension. Then one easily sees that $\till$ is a finitely generated extension of $\tilk$ of transcendence degree $F$ and $|l^\times|/|k^\times|$ is a finitely generated group whose torsion is contained in $(|k^\times|\otimes_\bfZ\bfQ)/|k^\times|$. In particular, if $|k^\times|$ is divisible (for example, trivial) then $|l^\times|/|k^\times|$ is a lattice of rank $E$. We will also need the following difficult result called the (generalized) stability theorem: if $k$ is stable then $l$ is stable. We refer to a very recent paper \cite{Kuh} for a proof; it seems that although this fact was known to experts, no proof was published earlier.

Let us also indicate how the stability theorem can be deduced from the results of \cite[\S6]{temst}, where an analytic analog is proved (i.e. one deals with topologically finitely generated extensions of analytic fields). The reduction consists of many easy steps: (i) one can assume that $l=k(x)$ is of transcendence degree $1$; (ii) by the same easy argument as used in the proof of \cite[6.3.6]{temst}, it suffices to consider the case when $k$ is algebraically closed; (iii) by a limit argument we can assume that $k$ is of finite transcendence degree over a prime field, in particular, the height of $k$ is finite; (iv) a valuation of finite height $h>1$ is stable if and only if it is composed of stable valuations of smaller height, hence everything follows from the case of height $1$; (v) a valued field $l$ of height one is stable if and only if $\hatl$ is stable and $\hatl/l$ is separable, but $\hatl$ is stable by \cite[6.3.6]{temst} (in the case of $F_{l/k}=1$, which is the more difficult one, this is, actually, the stability theorem of Grauert-Remmert \cite[5.3.2/1]{BGR}); (vi) one checks straightforwardly that $\hatl/l$ is separable in our case (the $p$-rank of $l=k(x)$ is one, i.e. $l$ has unique inseparable extension of degree $p$, which is easily seen to be not contained in $\hatl$).
\end{rem}

\subsection{Morphisms of normalized finite type}\label{normsec}
Since schemes of finite presentation over valuation rings are non-noetherian and often have non-finite normalizations, we should study normalization of reduced schemes and related issues. In applications all schemes will have noetherian underlying topological space, so the reader can have in mind only this particular case throughout \S\ref{normsec}.

By a {\em modification} of a reduced scheme $X$ we mean a proper morphism $\phi{\colon}X'\to X$ with reduced source that restricts to an isomorphism of dense subschemes. Next let us discuss normalization of schemes. For simplicity we will only consider reduced schemes $X$ with finitely many irreducible components. Such schemes will be called {\em admissible} and by {\em admissible morphism} we mean any morphism of admissible schemes that takes generic points to generic points. If $X$ is admissible then we denote the scheme of its generic points by $\eta(X)$ and set $k(X)=\prod_{x\in\eta(X)}k(x)$. In particular, $\eta(X)=\Spec(k(X))$ and if $X=\Spec(A)$ then $k(X)=\Frac(A)$ is the total ring of fractions of $A$. Obviously, $X\mapsto\eta(X)$ is a functor on the category of admissible schemes and morphisms.

Recall that the {\em normalization} $\Nr(X)$ of an admissible scheme $X$ is the disjoint union of normalizations of its irreducible components. If $X=\Spec(A)$ then $\Nr(X)=\Spec(B)$, where $B=\Nr_{\Frac(A)}(A)$ is the integral closure of $A$ in its total ring of fractions. Since normalization is compatible with localizations, $\Nr(X)$ in general can be glued from $\Nr(X_i)$ where $\{X_i\}$ is an open affine covering of $X$. This construction can also be described globally as follows. Let $i{\colon}\eta\to X$ be the embedding and let $\calM_X=i_*(\calO_\eta)$ be the sheaf of meromorphic functions. Then $\Nr(X)=\bfSpec(\Nr_{\calM_X}(\calO_X))$, where $\Nr_{\calM_X}(\calO_X)$ is the integral closure of $\calO_X$ in $\calM_X$. By a {\em partial normalization} of an admissible scheme $X$ we mean any scheme $X'=\bfSpec(\calF)$ for an $\calO_X$-subalgebra $\calF\subset\Nr(\calO_X)$. Note that $X'$ is integral over $X$ and $\Nr(X')\toisom\Nr(X)$. An admissible morphism $X'\to X$ is a partial normalization if and only if it is integral and $\eta(X')\toisom\eta(X)$.

In the sequel, {\em qcqs} stands for "quasi-compact and quasi-separated". The following lemma is a consequence of \cite[6.9.15]{egaI}.

\begin{lem}\label{partlem}
Any partial normalization $X'$ of an admissible qcqs scheme $X$ is $X$-isomorphic to the projective limit of all finite modifications of $X$ dominated by $X'$.
\end{lem}

\begin{defin}\label{nftdef}
An admissible morphism $f{\colon}Y\to X$ between qcqs schemes is called of {\em normalized finite type} if it splits into a composition of a partial normalization $Y\to Y_0$ and an admissible morphism $Y_0\to X$ of finite type. For shortness, we will often abbreviate "normalized finite type" as {\em nft}.
\end{defin}

\begin{rem}\label{rem1}
(i) It would be more pedantic to say partially normalized (or subnormalized) finite type, but this sounds too messy.

(ii) One can define morphisms of normalized finite presentation similarly, but it is not clear if one obtains a meaningful class of morphisms. For example, perhaps such morphisms are not stable under compositions.

(iii) Without the admissibility assumption, Definition \ref{nftdef} would lead to a class of morphisms not closed under compositions. Indeed, there exists an integral scheme $X$ with a point $x$ such that the fiber $Y_x$ of $Y=\Nr(X)$ over $x$ is not finite. Then $Y_x\to X$ is not a composition of a partial normalization $Y_x\to Z$ with a finite type morphism $Z\to X$.
\end{rem}

In order to study nft morphisms it will be convenient to consider a broader class of morphisms as follows.

\begin{defin}\label{iftdef}
A morphism $f\:Y\to X$ between qcqs schemes is {\em ift} if it can be factored into a composition of an integral morphism $Y\to Y_0$ and a finite type morphism $Y_0\to X$.
\end{defin}

\begin{prop}\label{iftprop}
Let $g{\colon}Z\to Y$ and $f{\colon}Y\to X$ be morphisms of qcqs schemes and let $h=f\circ g$.

(i) If $f$ and $g$ are ift then $h$ is ift.

(ii) If $h$ is ift then $g$ is ift.
\end{prop}
\begin{proof}
To prove (i) it suffices to show that if $f$ is integral and $g$ is of finite type then $h$ is ift. By \cite[Th. 4.3]{Con} $g$ is a composition of a closed immersion $Z\into T$ and a finitely presented morphism $T\to Y$. Using \cite[6.9.15]{egaI} we can represent $Y$ as the projective limit of finite $X$-schemes $X_\alp$. By \cite[$\rm IV_3$, 8.8.2(ii)]{ega} $T\to Y$ is the base change of a finitely presented morphism $T_\alp\to X_\alp$ for large enough $\alp$. In particular, being a base change of $Y\to X_\alp$, the morphism $T\to T_\alp$ is integral. Hence $h$ factors into the composition of an integral morphism $Z\to T\to T_\alp$ with a finite type morphism $T_\alp\to X_\alp\to X$.

Let us prove (ii). Let $h\:Z\to Z_0\to X$ with the first morphism integral and the second one of finite type. Then $g$ splits as $Z\into Z\times_XY\to Z_0\times_XY\to Y$. The first morphism is a locally closed immersion, the second one is integral and the third one is of finite type. Thus, all three are ift and hence $g$ is ift by part (i).
\end{proof}

Let us mention two other basic facts that will not be used.

\begin{rem}
(i) A morphism $\Spec(B)\to\Spec(A)$ is ift if and only if $B$ is integral over a finitely generated $A$-subalgebra.

(ii) Using technique from the proof of \cite[Th. 1.1.2]{temrz}, one can show that the property of being ift is local on the source.

(iii) The above property can be used to give a better definition that applies to all schemes: a morphism $f\:Y\to X$ is ift if it is quasi-compact and locally on $Y$ factors into a composition as in Definition \ref{iftdef}. (We preferred to use a more ad hoc definition to minimize our work.)
\end{rem}

\begin{lem}\label{iftnft}
Let $f\:Y\to X$ be an admissible morphism. Then $f$ is nft if and only if $f$ is ift and for any $y\in\eta(Y)$ the field extension $k(y)/k(f(x))$ is finitely generated.
\end{lem}
\begin{proof}
Only the inverse implication needs a proof. So, assume that $Y\to Z$ is integral and $Z\to X$ is of finite type. As earlier, represent $Y$ as a projective limit of finite $Z$-schemes $Z_\alp$. The morphisms $g_\alp\:Y\to Z_\alp$ are integral, and replacing $Z_\alp$'s with the schematic images of $Y$ we can make these morphisms admissible. For each $y\in\eta(Y)$
the field $k(y)$ is the union of the $k(f(x))$-subfields $k(g_\alp(y))$. Hence for large enough $\alp$ we have that $\eta(Y)\toisom\eta(Z_\alp)$ and we obtain that $f$ is nft.
\end{proof}

\begin{cor}\label{nftprop}
Let $g{\colon}Z\to Y$ and $f{\colon}Y\to X$ be admissible morphisms of qcqs schemes and $h=f\circ g$.

(i) If $f$ and $g$ are nft then $h$ is nft.

(ii) If $h$ is nft then $g$ is nft.
\end{cor}
\begin{proof}
Combine the above lemma with Proposition \ref{iftprop}.
\end{proof}

\subsection{$\eta$-normalization and $\eta$-nft morphisms}\label{etanormsec}
In addition to the absolute notions of normalization and modification of $X$, we will also need their relative analogs with respect to a morphism $f{\colon}Y\to X$. Although we repeat here a general definition from \cite[\S3.3]{temst}, where one only assumes that $X$ and $Y$ are qcqs, the reader can have in mind only the cases described in Example \ref{normexam} below, in which $f$ is either a point (i.e. $Y$ is the spectrum of a field) or the embedding $X_\eta\to X$ of the generic fiber $X_\eta$ of a morphism $X\to S$ with an integral target.

By a {\em $Y$-modification} of $X$ we mean a factorization of $f$ into a composition of a schematically dominant morphism $f'{\colon}Y\to X'$ with a proper morphism $g{\colon}X'\to X$. A $Y$-modification is {\em finite} if $g$ is finite. Note that the family of all (resp. finite) $Y$-modifications is filtered and has a final object $X_0$ which is the schematic image of $Y$ in $X$ (i.e., $X_0$ is the minimal closed subscheme of $X$ such that $Y$ factors through $X_0$), and so $X_0=\bfSpec(\calF_0)$, where $\calF_0$ is the image of $\calO_X$ in $f_*\calO_Y$.

If $Y=\Spec(B)$ and $X=\Spec(A)$ then we define $\Nr_B(A)$ to be the integral closure of the image of $A$ in $B$, and set $\Nr_Y(X)=\Spec(\Nr_B(A))$. In general, let $\Nr_Y(\calO_X)$ be the integral closure of the image of $\calO_X$ in the quasi-coherent sheaf of rings $f_*\calO_Y$. Then the $Y$-normalization of $X$ is defined as $\Nr_Y(X)=\bfSpec(\Nr_Y(\calO_X))$ and for any $\calO_X$-subalgebra $\calF\into\Nr_Y(\calO_X)$ the scheme $\bfSpec(\calF)$ is called a {\em partial $Y$-normalization} of $X$. The following analog of Lemma \ref{partlem} is also a consequence of \cite[6.9.15]{egaI}.

\begin{lem}\label{Ypartlem}
If $Y\to X$ is a morphism of qcqs schemes then any partial $Y$-normalization $X'$ of $X$ is $X$-isomorphic to the projective limit of all finite $Y$-modifications of $X$ which are dominated by $X'$.
\end{lem}

\begin{rem}
If $X$ is admissible then $\Nr(X)=\Nr_{\eta(X)}(X)$, thus expressing absolute normalization in terms of $Y$-normalization. Since $\calM_X=(i_\eta)_*\calO_{\eta(X)}$ in the absolute case, it is natural to view the sheaf $f_*(\calO_Y)$ as the sheaf of meromorphic functions on $X$ with respect to $Y$.
\end{rem}

We will use $Y$-normalizations in two particular cases described below.

\begin{exam}\label{normexam}
(i) If $Y=\Spec(K)$ for a field $K$ then we will usually say $K$-normalization, $K$-modification, etc., instead of $Y$-normalization, $Y$-modification, etc., and write $\Nr_K(X)$ instead of $\Nr_Y(X)$. If $X$ is covered by open affine subschemes $X_i=\Spec(A_i)$ then $\Nr_K(X)$ is pasted from the schemes $\Spec(A'_i)$, where $A'_i=\Nr_K(A_i)$ if the image of $Y$ is in $X_i$ and $A'_i=\Nr_0(A_i)=0$ otherwise.

(ii) Let $S$ be an integral scheme with $k=k(S)$ and generic point $\eta=\Spec(k)$ ($S$ will be the spectrum of a valuation ring in applications). For an $S$-scheme $X$ we will usually say $\eta$-normalization, $\eta$-modification, etc., instead of $X_\eta$-normalization, $X_\eta$-modification, etc., and write $\Nr_\eta(X)$ instead of $\Nr_{X_\eta}(X)$. The $\eta$-normalization of $X$ is pasted from $\eta$-normalizations of affine subschemes, and for an affine $X=\Spec(B)$ sitting over an affine subscheme $\Spec(A)\into S$ we have that $X_\eta=\Spec(B_\eta)$ for $B_\eta=B\otimes_A k$, and $\Nr_\eta(X)$ is the spectrum of the integral closure of the image of $B$ in $B_\eta$.
\end{exam}

Note that for an integral scheme $S$ with generic point $\eta$, $\Nr_\eta$ is a functor on the category of $S$-schemes. Indeed, it suffices to prove that if $S=\Spec(A)$, $X=\Spec(B)$ and $Y=\Spec(C)$ then any $S$-morphism $Y\to X$ lifts uniquely to a morphism $\Nr_\eta(Y)\to\Nr_\eta(X)$. But if $B'$ and $C'$ are the integral closures of the images of $B$ and $C$ in $B_\eta=B\otimes_A K$ and $C_\eta=C\otimes_A K$, respectively, then the $A$-homomorphism $B\to C$ lifts uniquely to an $A$-homomorphism $B'\to C'$. In particular, if $Y$ is $\eta$-normal then any $S$-morphism $Y\to X$ factors uniquely through $\Nr_\eta(X)$. Note also that analogous statements hold for $K$-normalizations of $K$-pointed schemes (where all morphisms are compatible with the $K$-points). In the sequel, it will often be convenient to work with normal or $\eta$-normal schemes, but,
unfortunately, the $\eta$-normalization morphism need not be finite. This forces us to give the following definition.

\begin{defin}\label{Ynfpdef}
Assume that $S$ is an integral scheme with $\eta=\eta(S)$ and $f{\colon}Y\to X$ is an $S$-morphism between qcqs schemes. Then we say that $f$ is of {\em $\eta$-normalized finite type/presentation} if it is the composition of a partial $\eta$-normalization $Y\to Y_0$ and a morphism $Y_0\to X$ of finite type/presentation. We will abbreviate these as {\em $\eta$-nft} and {\em $\eta$-nfp}.
\end{defin}

\begin{rem}\label{Yrem1}
(i) Note that if $Y\to X$ is $\eta$-nft then $\calO_Y$ has no non-trivial $\calO_S$-torsion because any such torsion is killed by any partial $\eta$-normalization.

(ii) The following fact was observed by D. Rydh. Although it will not be used later, we include it for the sake of completeness. If $S$ is reduced and with finitely many generic points then the notions of $\eta$-normalized finite type and presentation for $X\to S$ are equivalent. The proof can be easily obtained from \cite[3.4.6]{RG} and the fact that any $S$-scheme $\oX$ of finite type can be embedded into a finitely presented $S$-scheme $\oY$ such that $\oX\to\oY$ is an isomorphism over a dense open subscheme of $S$.
\end{rem}

\begin{defin} Let $f{\colon}S'\to S$ be a dominant morphism of integral schemes with generic points $\eta'$ and $\eta$. Then the $\eta$-normalized base change functor $\calF_f$ from the category of $S$-schemes to the category of $S'$-schemes is defined as the composition of the base change with $\eta'$-normalization, i.e. for an $S$-morphism $g{\colon}Y\to X$, $\calF_f(g)=\Nr_{\eta'}(g\times_S S')$.
\end{defin}

Note that for an $\eta'$-normal $S'$-scheme $Y$ and an $S$-scheme $X$, any $S$-morphism $Y\to X$ factors through $\calF_f(X)=\Nr_{\eta'}(X\times_S S')$ uniquely. Also, if $g{\colon}S''\to S'$ is another dominant morphism with an integral source then $\calF_{f\circ g}=\calF_g\circ\calF_f$. Now, we are going to study $\eta$-normalized filtered projective limits analogously to \cite[$\rm IV_3$, \S8]{ega}. In applications, we will have a valuation ring $\calO$ approximated by local rings $\calO_\alp$ of varieties in the sense that $\calO$ is a filtered union of $\calO_\alp$. Then $S=\Spec(\calO)$ is isomorphic to the filtered projective limit of $S_\alp=\Spec(\calO_\alp)$ and we will approximate $S$-schemes with $S_\alp$-schemes.

\begin{sit}\label{projsit}
Let $\{S_\alp\}_{\alp\in A}$ be a filtered projective family of integral qcqs schemes with dominant affine transition morphisms and an initial scheme $S_0$. The scheme $S=\projlim S_\alp$ exists by \cite[$\rm IV_3$, 8.2.3]{ega} and is integral by \cite[$\rm IV_2$, 5.13.3, $\rm IV_3$, 8.4.1]{ega}. Set $k_\alp=k(S_\alp)$
and $\eta_\alp=\Spec(k_\alp)$. Let, furthermore, $X_0$ and $Y_0$ be the $\eta_0$-normalizations of $S_0$-schemes $\oX_0$ and $\oY_0$ of finite presentation, and let $f_0{\colon}Y_0\to X_0$ be an $S_0$-morphism. We define $X_\alp,Y_\alp$
and $f_\alp$ (resp. $X,Y$ and $f$) to be the $\eta$-normalized base changes of $X_0,Y_0$ and $f_0$ with respect to the morphism $S_\alp\to S_0$ (resp. $S\to S_0$).
\end{sit}

\begin{prop}\label{normlimprop}
(i) The schemes $X$ and $Y$ are $S$-isomorphic to $\projlim_\alp X_\alp$ and $\projlim_\alp Y_\alp$, respectively.

(ii) There is a natural bijection
$$\mu{\colon}\injlim_{\alp\in A}\Hom_{S_\alp}(Y_\alp,X_\alp)\toisom\Hom_S(Y,X)$$

(iii) If an $\eta$-normal scheme $Z$ is $\eta$-nfp over $S$ then there exists $\alp\in A$ such that $Z$ is $S$-isomorphic to the $\eta$-normalized base change of an $S_\alp$-scheme $\oZ_\alp$ of finite presentation.

(iv) The morphism $f$ is \'etale (resp. smooth) if and only if there exists $\alp_0\in A$ such that for each $\alp\ge\alp_0$ the morphism $f_\alp$ is \'etale (resp. smooth).
\end{prop}
\begin{proof}
We deduce the proposition from its analog in \cite[$\rm IV_3$, \S8]{ega}. Let us prove that $X\toisom\projlim_{\alp\in A} X_\alp$. The question is local on $X_0$, so we can assume that it is affine, and then the schemes $X'=X_0\times_{S_0}S$ and $X'_\alp=X_0\times_{S_0}S_\alp$ and their generic fibers over $S$ and $S_\alp$, respectively, are also affine, say, $X'_\alp=\Spec(A_\alp)$, $X'=\Spec(A)$,
$X_\eta=X'_\eta=\Spec(A_\eta)$ and $X_{\alp,\eta}=X'_{\alp,\eta}=\Spec(A_{\alp,\eta})$ (where to simplify notation we write $X_{\alp,\eta}$ instead of $X_{\alp,\eta_\alp}$). By \cite[$\rm IV_3$, \S8.2]{ega}, $X_\eta=\projlim_\alp X_{\alp,\eta}$ and $X'=\projlim_\alp X'_\alp$, hence $A_\eta$ is the filtered union of its subalgebras $A_{\alp,\eta}$ and $A$ is the filtered inductive limit of the $A_\alp$'s. Therefore, the subring $\Nr_{A_\eta}(A)$ of $A_\eta$ is the filtered union of the subrings $\Nr_{A_{\alp,\eta}}(A_\alp)$, and applying $\Spec$ we obtain that $X$ is the filtered projective limit of the $X_\alp$'s. Applying the same argument to $Y$ we finish the proof of (i).

To prove (iii) we note that $Z$ is the $\eta$-normalization of a scheme $\oZ$ of finite presentation over $S$, and then by \cite[$\rm IV_3$, 8.8.2(ii)]{ega} $\oZ$ is the base change of a scheme $\oZ_\alp$ of finite presentation over some $S_\alp$. So, $\oZ_\alp$ is as claimed. Let us prove (ii). Since $\eta$-normalized base changes induce compatible maps from $\Hom_{S_\alp}(Y_\alp,X_\alp)$ to $\Hom_{S_\beta}(Y_\beta,X_\beta)$ (for $\beta\ge\alp$) and to $\Hom_S(Y,X)$, a map $\mu$ naturally arises. We first treat the case when $X_0$ is separated. Then $X$ and all $X_\alp$'s are separated because they are affine over $X_0$, and so any morphism from the above $\Hom$'s is determined by its restriction to the generic fibers $Y_{\alp,\eta}$ and $Y_\eta$ (which are schematically dense in $Y_\alp$ and $Y$ by $\eta$-normality). Since $\eta=\projlim_\alp\eta_\alp$ and $X_{\alp,\eta}=X_{0,\eta}\times_{\eta_0}\eta_\alp$ we obtain that $X_\eta=X_{0,\eta}\times_{\eta_0}\eta$ and similarly for $Y$'s. The $\eta_0$-schemes $X_{0,\eta}$ and $Y_{0,\eta}$ are of finite type, hence there is a natural isomorphism
$$\mu_\eta{\colon}\injlim_{\alp\in A}\Hom_{\eta_\alp}(Y_{\alp,\eta},X_{\alp,\eta})\toisom\Hom_\eta(Y_\eta,X_\eta)$$
by \cite[$\rm IV_3$, 8.8.2.(i)]{ega}. The injectivity of $\mu$ follows, and to prove the surjectivity we will find a morphism $g_\alp{\colon}Y_\alp\to X_\alp$ which induces a given morphism $g{\colon}Y\to X$. Since $Y$ is the projective limit of the $Y_\alp$'s by (i), \cite[$\rm IV_3$, 8.13.1]{ega} implies that the $S_0$-morphism $Y\to\oX_0$ (which is the composition of $g$ with the projection $X\to X_0\to\oX_0$) is induced from a morphism $g'{\colon}Y_\alp\to\oX_0$. Obviously, $g'$ factors through $X_\alp$, hence we obtain a morphism $g_\alp{\colon}Y_\alp\to X_\alp$ compatible with $g$. In particular, $g_{\alp,\eta}$ is compatible with $g_\eta$, and therefore $g_\eta$ is the base change of $g_{\alp,\eta}$. Then the schematical density of $Y_\eta$ in $Y$ implies that $g$ must coincide with the normalized base change of $g_\alp$, i.e. $\mu(g_\alp)=g$ as required. This establishes the case of a separated $X_0$, and the general case is deduced using an affine atlas for $X_0$. We omit the details, since we will use only the separated case in applications.

\begin{lem}\label{smetlem}
Let $f{\colon}Y\to X$ be a smooth (resp. \'etale) morphism of finite type, and assume in assertions (iii) and (iv) below that $f$ is an $S$-morphism for an integral scheme $S$ with generic point $\eta$.

(i) If $X$ is integral and normal then $Y$ is a finite disjoint union of integral normal schemes.

(ii) If $X$ and $Y$ are integral, $k/k(X)$ is a finite extension and $l=kk(Y)$ is any $k(X)$-field that is generated by subfields $k(X)$-isomorphic to $k$ and $k(Y)$, then the induced morphism $\Nr_l(Y)\to\Nr_k(X)$ is smooth (resp. \'etale). In particular, taking $k=k(X)$ one obtains that $\Nr(f)$ is smooth (resp. \'etale).

(iii) If $X$ is $\eta$-normal then so is $Y$.

(iv) The $\eta$-normalization morphism $\Nr_\eta(f)$ is smooth (resp. \'etale). In particular, if $g{\colon}S'\to S$ is a dominant morphism of integral schemes then the $\eta$-normalized base change of $f$ is smooth (resp. \'etale).
\end{lem}
\begin{proof}
Note that (iii) follows from \cite[16.2.1]{LMB}. (Also, as D. Rydh pointed out \cite[2.2.1]{Laz} implies that (iii) holds more generally for any flat $f$ with geometrically reduced fibers.) I am grateful to the referee for the following
argument that shortened the proof of (i). If $\eta$ denotes the generic point of $X$ then applying (iii) to $S=X$ we obtain that $Y$ is $Y_\eta$-normal. Clearly, $Y_\eta$ is a smooth $\eta$-variety, hence it is a finite disjoint union of integral normal schemes. By the transitivity of normality, $Y$ is also a finite disjoint union of integral normal schemes.

The assertions of (ii) and (iv) are deduced from (i) and (iii), respectively, in a similar way, so we will prove only (ii). Set $X'=\Nr_k(X)$ and let $f'{\colon}Y'\to X'$ be the base change of $f$. Since $f'$ is smooth and $X'$ is normal, $Y'$ is a disjoint union of integral normal schemes by part (i) of the lemma. Since the extension $k(Y)/k(X)$ is separable by smoothness of $f$, $k\otimes_{k(X)}k(Y)$ is a direct product of fields and $l$ is one of the factors. Let $Y'_l$ be the irreducible component of $Y'$ with the generic point corresponding to $l$, then it suffices to prove that $\Nr_l(Y)\toisom Y'_l$ because obviously $Y'_l$ is smooth (resp. \'etale) over $X'$. The morphism $\Nr_l(Y)\to X$ factors through $X'$, hence we also obtain a morphism from $\Nr_l(Y)$ to $Y'$. It is integral because both $\Nr_l(Y)$ and $Y'$ are integral over $Y$. The generic point of $\Nr_l(Y)$ is
mapped isomorphically onto the generic point of $Y'_l$, hence we obtain a birational integral morphism $\Nr_l(Y)\to Y'_l$, which must be an isomorphism by normality of $Y'_l$.
\end{proof}

Now, let us prove (iv). If $f_\alp$ is smooth (resp. \'etale) then by Lemma \ref{smetlem}(iv) so is its $\eta$-normalized base change $f$. Conversely, assume that $f$ is smooth (resp. \'etale). Since $X=\projlim X_\alp$, $f$ is the base change of a smooth (resp. \'etale) morphism $\of_\alp{\colon}\oY_\alp\to X_\alp$ for some $\alp$. Then $Y\toisom\oY_\alp\times_{X_\alp}X$ is the $\eta$-normalization of $\oY_\alp\times_{S_\alp}S$, hence $f$ is the $\eta$-normalized base change of $\of_\alp$. Thus, $f_\alp$ and $\of_\alp$ are two morphisms of $\eta$-normalized $S_\alp$-schemes whose $\eta$-normalized base changes to $S$ are isomorphic. By part (ii) of the proposition, they become isomorphic already over some $S_\beta$, hence $f_\beta$ is isomorphic to the
$\eta$-normalized base change $\of_\beta$ of $\of_\alp$ for each $\beta$ larger than some $\beta_0$. But $\of_\beta$ is smooth (resp. \'etale) by Lemma \ref{smetlem}(iv), hence $f_\beta$ is smooth (resp. \'etale) for each $\beta\ge\beta_0$.
\end{proof}

\subsection{Birational fibers}\label{ebfibsec}
First we recall some definitions and results from \cite[\S3.2]{temst}. For any field $K$ by $\bfP_K$ we denote the Riemann-Zariski space of $K$. Its points are valuation rings of $K$. If $X$ and $Y$ are two subsets in $K$ and $\gtZ$ is a subset of $\bfP_K$ then by $\gtZ\{X\}\{\{Y\}\}$ we denote the subset of $\gtZ$ which consists of elements $\calO\in\gtZ$ such that $X\subset\calO$ and $Y\subset m_\calO$. In other words, $\gtZ\{X\}\{\{Y\}\}$ is cut off from $\gtZ$ by the inequalities $|x|\le 1$ and $|y|<1$ with $x\in X, y\in Y$. The Zariski topology on $\gtZ$ is defined by non-strict inequalities, and the constructible topology on $\gtZ$ is defined by the inequalities of both types, i.e. the basis of the Zariski topology is formed by the sets $\gtZ\{f_1\. f_n\}$, and the basis of the constructible topology is formed by the sets $\gtZ\{f_1\. f_n\}\{\{g_1\. g_m\}\}$. Zariski topology is the default one, so each time we will use the constructible topology it will be said explicitly. It is well known that the sets $\gtZ=\bfP_K\{X\}\{\{Y\}\}$ are compact in the constructible topology (for example, one can use the arguments from \cite[3.2.1]{temst} or \cite[5.3.6]{ct}), hence they are quasi-compact in the weaker Zariski topology.

Assume, now, that $\gtZ=\bfP_K\{X\}\{\{Y\}\}$ and let us make a few more simple remarks on the constructible topology. (All what we will say holds, more generally, for arbitrary spectral topological spaces.) A subset $S\subset\gtZ $ is called {\em constructible} if
$S=\cup_{i=1}^n\gtZ\{X_i\}\{\{Y_i\}\}$ with finite sets $X_i,Y_i\subset K$. The family of constructible sets is closed under taking finite unions, finite intersections and complements (the latter follows from the observation that $\gtZ=\gtZ\{f\}\coprod\gtZ\{\{f^{-1}\}\}$ for any $f\in K^\times$). Arbitrary intersections (resp. unions) of constructible sets are called {\em pro-constructible} (resp. {\em ind-constructible}). Note that a set is pro-constructible (resp. ind-constructible) if and only if it is compact (resp. open) in the constructible topology, so we will use these
notions to avoid mentioning the constructible topology. Note also that a Zariski open set is quasi-compact if and only if it is constructible, and any pro-constructible set is quasi-compact in the Zariski topology.

Next, let us recall the relation between Riemann-Zariski spaces and schemes. To any integral scheme $X$ provided with a dominant morphism $\eta{\colon}\Spec(K)\to X$ one can associate a Riemann-Zariski space $\RZ_K(X)$ which is defined as the projective limit of the underlying topological spaces of the $K$-modifications of $X$ (actually, the definition makes sense for any scheme $X$ with a point $\eta$). Points of $\RZ_K(X)$ can be naturally interpreted as morphisms $\phi{\colon}\Spec(\calO)\to X$ where $\calO\in\bfP_K$ and the restriction of $\phi$ onto the generic point is $\eta$. The natural projection $\RZ_K(X)\to\bfP_K$, which keeps $\calO$ but forgets $\phi$, is a local homeomorphism in general and a topological embedding when $X$ is separated. Thus, for a separated $X$ we can identify $\RZ_K(X)$ with the subset of $\bfP_K$ consisting of the valuation rings centered on $X$. In particular, $\bfP_K\{A\}$ can be naturally identified with the projective limit of all $K$-modifications of $\Spec(A)$.

For any point $x\in X$ by the {\em birational fiber} $X^\bir_x$ over $x$ we mean the
preimage of $x$ under the projection $\RZ_K(X)\to X$, and we identify $X^\bir_x$ with a subset of $\bfP_K$. So, $X^\bir_x$ is the set of valuations centered on $x$. Note that $X^\bir_x=\bfP_K\{\calO_{X,x}\}\{\{m_x\}\}$ is the set of all valuation rings $\calO\subseteq K$ that {\em dominate} $\calO_{X,x}$, i.e. $m_\calO\cap\calO_{X,x}=m_x$. If $\calO$ is a local domain with field of fractions $K$ and $L/K$ is any extension of fields then by the birational fiber of $\calO$ in $L$ we mean the set $\bfP_L\{\calO\}\{\{m_\calO\}\}$, which is the preimage of the birational fiber of the closed point of $\Spec(\calO)$ under the natural map $\bfP_L\to\bfP_K$.

We finish this section with proving some results that are not covered by \cite{temst} but will be useful in the sequel.

\begin{lem}\label{birfiblem}
Let $A$ and $A'$ be two normal local rings with field of fractions $K$ and birational fibers $X$ and $X'$. Then $X'\subseteq X$ if and only if $A\subseteq A'$ and $A'$ dominates $A$.
\end{lem}
\begin{proof}
If $A'$ dominates $A$ then any valuation ring dominating $A'$ also dominates $A$ and hence $X'\subseteq X$. Conversely, assume that $X'\subseteq X$. Recall that by \cite[Ch.6, \S1, Th. 3]{Bou}, a normal local ring coincides with the intersection of all valuation rings of the fraction field that dominate it. Thus $A$ coincides with the intersection of all valuation rings $\calO\in X$, and similarly for $A'$. Therefore, $A\subseteq A'$. Finally, a valuation ring $\calO\in X'$ dominates both $A'$ and $A$ and hence $A'$ dominates $A$.
\end{proof}

\begin{theor}\label{Zarconth}
Let $A$ be a local domain with $K=\Frac(A)$ and $X=\Spec(A)$, and let $x\in X$ be the closed point. Then the following conditions are equivalent:

(i) the birational fiber $X_x^\bir\subset\bfP_K$ is connected,

(ii) for any modification $X'\to X$ the preimage of $x$ is connected,

(iii) $A$ is unibranch.
\end{theor}
\begin{proof}
First we note that if $A$ is not unibranch then both (i) and (ii) obviously fail. Hence (iii) follows from either of the first two conditions. Until the end of the proof we will therefore assume that $A$ is unibranch, and our aim is to deduce both (i) and (ii). First we prove equivalence of (i) and (ii). For each modification $X'\to X$ let $X'_x$ denote the fiber over $x$. So, $Z:=X_x^\bir$ is the preimage of $X'_x$ under the projection $\bfP_K\{A\}\to X'$. In particular, if $Z$ is connected then each $X'_x$ is so. Conversely, assume that $Z$ is disconnected, say, $Z=U\coprod V$ for open $U$ and $V$. Both $U$ and $V$ are quasi-compact and hence $U=\cup_{i=1}^nZ\{F_i\}$ with finite $F_i\subset K$, and similarly for $V$. Find a modification $X'\to X$ such that each $f\in F_i$ induces a morphism $f{\colon}X'\to\bfP^1_\bfZ$. Then $U$ and $V$ are the full preimages of quasi-compact open subsets $U',V'\subset X'_x$ given by the same formulas involving the $F_i$'s (e.g. $Z\{f\}$ (resp. $Z\{f^{-1}\}$) is the preimage of the open subscheme $\Spec(\bfZ[T])$ (resp. $\Spec(\bfZ[T^{-1}])$) obtained by removing the infinity (resp. zero) section of $\bfP_\bfZ^1$). In particular, $X'_x=U'\coprod V'$ is disconnected.

By the definition of unibranch local rings, the normalization $A'$ of $A$ is a local ring, hence the birational fibers of $A$ and $A'$ coincide. So, it suffices to prove (i) for normal local domains, and since (i) and (ii) are equivalent, we will in the sequel assume that $A$ is normal. Note that (ii) is then the Zariski connectedness theorem for normal schemes. The theorem is classical for noetherian rings, see \cite[$\rm III_1$, 4.3.1]{ega}. Its generalization to the general case was proved by M. Artin using such a complicated tool as proper base change theorem for \'etale cohomology, see \cite[Exp. XII, Cor. 5.7]{sga4}. For the sake of comparison, we will show how one can complete the proof without using the latter result. Thus, our starting point is that (ii) and, therefore, (i) hold when $A$ is normal noetherian, and our strategy will be to deduce that (i) holds for any normal $A$.

Assume, conversely, that $A$ is normal but $Z$ is disconnected. We proved earlier that $Z=U\coprod V=(\cup_{l=1}^n Z\{F_l\})\coprod(\cup_{j=1}^m Z\{G_j\})$ with non-empty $U$ and $V$ and finite subsets $F_l,G_j\subset K$ (in particular, each set $Z\{F_l,G_j\}$ is empty). Find a filtered family of noetherian normal local rings $\{A_i\}_{i\in I}$ such that $A_i\subset A$, $A$ dominates $A_i$, $K_i:=\Frac(A_i)$ contains the sets $F_l$ and
$G_j$ and $\cup_{i\in I}A_i=A$.  Note that $Z=\cap_{i\in I}Z_i$ where $Z_i\subset\bfP_K$ is the birational fiber of $A_i$ in $K$, and $\{Z_i\}_{i\in I}$ is a filtered family of pro-constructible sets. We claim that for sufficiently large $i$, the sets $Z_i\{F_l\}$ and $Z_i\{G_j\}$ cover $Z_i$, and $Z_i\{F_l,G_j\}=\emptyset$ for any $1\le l\le n$ and $1\le j\le m$. The second is obvious since $\cap_{i\in I}Z_i\{F_l,G_j\}=Z\{F_l,G_j\}=\emptyset$, and to prove the first we note that the open set $W:=(\cup_{l=1}^n \bfP_K\{F_l\})\cup(\cup_{j=1}^m \bfP_K\{G_j\})$ contains $Z=\cap Z_i$. Since each $Z_i$ is compact in the constructible topology, already some $Z_i$ lies in $W$. Choose $i$ as above. Then $Z_i=(\cup_{l=1}^n Z_i\{F_l\})\coprod(\cup_{j=1}^m Z_i\{G_j\})$, and since the $F_l$'s and the $G_j$'s are in $K_i$, the same representation is valid already for the birational fiber $Z'_i\subset\bfP_{K_i}$ of $A_i$. In particular, $Z'_i=U'\coprod V'$ with open $U'$ and $V'$, that must be non-empty because $U$ and $V$ are contained in their preimages in $Z_i$. The latter implies that $Z'_i$ is disconnected, and we obtain a contradiction to the already established noetherian case. This finishes the proof.
\end{proof}

\begin{cor}\label{birconlem}
Let $A$ be a geometrically unibranch local domain and let $L$ be an extension of $\Frac(A)$ of finite degree $n$. Then $\Nr_L(A)$ is a semi-local ring with at most $n$ maximal ideals and the birational fiber of $A$ in $L$ is the disjoint union of the connected components which are the birational fibers of the closed points of $\Nr_L(\Spec(A))$.
\end{cor}
\begin{proof}
By \cite[$\rm IV_4$, 18.10.16(i)]{ega} any finite $L$-modification $X\to Y:=\Spec(A)$ has at most $n$ closed points. Since $\Nr_L(Y)$ is the projective limit of finite $L$-modifications of $Y$, it has at most $n$ closed points too. In particular, $\Nr_L(A)$ is semi-local with at most $n$ maximal ideals, and it is clear that the birational fiber of $A$ is the disjoint union of the birational fibers of the closed points of $\Nr_L(Y)$, which are connected by Theorem \ref{Zarconth}.
\end{proof}

\subsection{Birational criterion of \'etaleness}\label{bircritsec}
We will prove a criterion for a morphism $f{\colon}Y\to X$ between normal integral schemes to be \'etale at a point. Similar results are proved in the thesis of D. Rydh, where he studied, in particular, families of zero cycles (some of these results are available at \cite{Rydh}). The classical criterion \cite[$\rm IV_4$, 18.10.16(ii)]{ega} does not cover our needs because it gives a criterion for $f$ to be finite \'etale, and so it is not local on $Y$. However, one can improve this criterion by working with henselizations or combining it with \cite[$\rm IV_4$, 18.12.1]{ega}.\footnote{The idea to use \cite[$\rm IV_4$, 18.12.1]{ega} is due to the referee, and it simplified and corrected some arguments from the first version of the paper.}

Recall that in \cite[$\rm IV_4$, 18.10.16]{ega}, to each point $y_i$ that is isolated in the fiber over a point $x\in X$ one associates the separable degree $n_i$ of $k(y_i)$ over $k(x)$, and for a separated $f$ the sum of all $n_i$'s equals to $n=[k(Y):k(X)]$ if and only if $f$ is finite \'etale over a neighborhood $U$ of $x$ (i.e. $f\times_XU$ is finite \'etale). If we want to work locally with $y_i$'s then we have to refine the numbers $n_i$ so that the multiplicity of ramification is taken into account.

\begin{defin}\label{degdef}
Assume that $f{\colon}Y\to X$ is a dominant nft morphism between integral schemes and assume that $X$ is unibranch at a point $x$ and $y\in Y$ is isolated in the fiber $f^{-1}(x)$. Note that $[k(Y):k(X)]<\infty$ because locally at $y$ $f$ is a composition of a partial normalization with a quasi-finite morphism. By the {\em henselian degree} $n_{y/x}$ of $f$ at $y$ we mean the minimal possible value of $\sum_{i=1}^m [k(Y'_i):k(X')]$, where $g{\colon}X'\to X$ is a morphism with an integral source and such that $g^{-1}(x)=\{x'\}$ and $g$ is strictly \'etale at $x'$, and $Y'_1\. Y'_m$ are the irreducible components of $Y\times_X X'$ containing the preimage of $y$.
\end{defin}

The following properties of henselian degrees are obvious.

\begin{rem}\label{hensrem}
(i) Instead of using strictly \'etale base changes one can use the henselization $X^h=\Spec(\calO^h_{X,x})$ of $X$ at $x$. Recall that $X^h$ is integral by \cite[$\rm IV_4$, 18.6.12]{ega}. If $Y_1^h\.Y_m^h$ are the irreducible components of $Y^h=Y\times_X X^h$ containing the preimage of $y$ then $n_{y/x}=\sum_{i=1}^m[k(Y_i^h):k(X^h)]$.

(ii) If $Y$ is unibranch at $y$ then $Y^h$ is unibranch at the preimage of $y$, hence $m=1$.

(iii) If $g{\colon}Y'\to Y$ is an integral morphism with $k(Y)\toisom k(Y')$ then $n_{y/x}=\sum_{y'\in g^{-1}(y)} n_{y'/x}$. In particular, the fiber of $g$ over $y$ contains at most $[k(Y):k(X)]$ points.
\end{rem}

In the sequel, we will use the notion of local-\'etaleness which is recalled in Appendix \ref{etalesec}.

\begin{theor}\label{etaleth}
Let $f{\colon}Y\to X$ be a dominant nft morphism between integral schemes, and assume that $X$ is unibranch at $x$. Let $y_1\. y_m$ be all isolated points of the fiber over $x$, $n_i=n_{y_i/x}$ and $n=[k(Y):k(X)]$, then

(i) Assume that $f$ is separated. Then $\sum_{i=1}^m n_i\le n$ and the equality holds if and only if $f$ is integral over a neighborhood of $x$.

(ii) Assume that $X$ is normal at $x$. Then $f$ is local-\'etale at $y_i$ if and only if $n_i$ equals to the separable degree of $k(y_i)$ over $k(x)$.

(iii) Assume that $X$ is normal in a neighborhood of $x$. Then $f$ is \'etale at $y_i$ if and only if $n_i$ equals to the separable degree of $k(y_i)$ over $k(x)$.

(iv) Assume that $X$ is normal at $x$ (resp. in a neighborhood of $x$). Then $f$ is strictly local-\'etale (resp. strictly \'etale) at $y_i$ if and only if $n_i=1$.
\end{theor}
\begin{proof}
Choose an integral $Y_0$ of finite type over $X$ and such that $Y$ is its partial normalization. Then $Y$ is the projective limit of finite modifications of $Y_0$ by Lemma \ref{partlem}, hence there exists a finite modification $Y'\to Y_0$ with points $y'_1\. y'_m\in Y'$ which are discrete in the fiber over $x$ and are the images of $y_1\. y_m$. We claim that it suffices to prove the theorem for $Y'$ and the $y'_i$'s instead of $Y$ and the $y_i$'s. Indeed, $n_{y'_i/x}=n_{y_i/x}$ by Remark \ref{hensrem}(iii) and if $f'{\colon}Y'\to X$ is local-\'etale at $y'_i$ and $X$ is normal at $x$ then $Y'$ is normal at $y'_i$ and therefore the partial normalization $Y\to Y'$ induces isomorphisms $\Spec(\calO_{Y,y_i})\toisom\Spec(\calO_{Y',y'_i})$. Similarly, if $f'$ is \'etale at $y'_i$ and $X$ is normal in a neighborhood of $x$ then $Y'$ is normal in a neighborhood of $y'_i$ and $Y\to Y'$ is a local isomorphism at $y'_i$. Thus, in order to prove all parts of the theorem, we can replace $Y$ with $Y'$ achieving that $f$ is of finite type. In particular, $f$ is integral over a neighborhood of $x$ if and only if it is finite over that neighborhood.

To prove (i) we consider the henselization $X^h=\Spec(\calO^h_{X,x})$ with the closed point $x^h$ and the base change $f^h{\colon}Y^h\to X^h$ of $f$, and note that $f^h$ is finite if and only if $f$ is finite over a neighborhood of $x$. Furthermore, by \cite[$\rm IV_4$, 18.5.11(c)]{ega} any irreducible component of $Y^h$ contains at most one isolated point in the fiber $Y^h_x$ over $x^h$, hence $\sum n_i\le n$ and the equality takes place if and only if any irreducible component of $Y^h$ contains an isolated point from $Y^h_x$. So, it suffices to prove that $f^h$ is finite if and only if any irreducible component of $Y^h$ contains an isolated point from $Y^h_x$, but the latter is an immediate consequence of \cite[$\rm IV_4$, 18.5.11(c)]{ega}.

The direct implications in (ii), (iii) and (iv) are obvious, so let us prove the converse ones. Fix $i$ and $y=y_i$, and assume that $n_i$ equals to the separable degree of $k(y)$ over $k(x)$. According to \cite[$\rm IV_4$, 18.12.1 and 18.12.2]{ega} there is an \'etale morphism $g{\colon}X'\to X$, such that $g^{-1}(x)=\{x'\}$ and $g$ is strictly \'etale at $x'$, and an open neighborhood $V'$ of the unique point $y'\in Y\times_XX'$ above $y$ such that $V'\to X'$ is finite. Furthermore, $X'$ is normal at $x'$ because $X$ is normal at $x$, hence we can replace $X'$ with the irreducible component containing $x'$.

Recall that $k(V')=\prod_{j=1}^m k(V'_j)$ where $V'_j$ are the irreducible components of $V'$. So, $k(V')$ is a finite $k(X')$-algebra and we claim that $n_i=[k(V'):k(X')]$. Indeed, choose a morphism $g{\colon}X''\to X'$ that is strictly \'etale at a point $x''$ above $x'$ and computes $n_i$, see Definition \ref{degdef}. Clearly, we can replace $X''$ with $X''\times_XX'$ achieving that $X''\to X$ factors through $X'\to X$. Then $V''=V'\times_{X'}X''$ has unique point $y''$ above $x''$, and by finiteness of $V''\to X''$ any irreducible component of $V''$ contains $y''$. Therefore, $n_i=[k(V''):k(X'')]=[k(V'):k(X')]$ and by \cite[$\rm IV_4$, 18.10.16(ii)]{ega} we obtain that $V'\to X'$ is \'etale at $y'$. By \'etale descent, $f$ is \'etale at $y$ and we have proved both (ii) and (iii). Finally, (iv) follows from (ii) and (iii).
\end{proof}

Parts (ii)--(iv) of the theorem provide a local criterion of \'etaleness, but may look rather tautological because a direct computation of the degrees $n_{y/x}$ involves \'etale localization, so at first glance we say that a morphism is \'etale if it so \'etale-locally. However, the situation is subtler since one can gain some control on the degrees by other methods, and part (i) of the theorem gives such an example. We will see that one can test the degree by restricting the computation to a single valuation ring.

\begin{lem}\label{lem1}
Let $f{\colon}Y\to X$ and $g{\colon}X'\to X$ be dominant morphisms between integral schemes, and assume that $f$ is nft and $g$ induces an isomorphism of the generic points. Let $x'\in X'$ be a point such that $X'$ and $X$ are unibranch at $x'$ and $x=g(x')$, respectively. If $y\in Y$ is an isolated point of the fiber over $x$, and $y'_1\. y'_m$ are all points of $Y'=\Nr_{k(Y)}(Y\times_X X')$ sitting over $y$ and $x'$, then $n_{y/x}=\sum_{i=1}^m n_{y'_i/x'}$.
\end{lem}
\begin{proof}
Note that if $y$ is the only preimage of $x$ and $f$ is integral over a neighborhood of $x$ then $Y'\to X'$ is integral over a neighborhood of $x'$ and the lemma follows from Theorem \ref{etaleth}(i) because $n_{y/x}=[k(Y):k(X)]=[k(Y'):k(X')]=\sum_{i=1}^m n_{y'_i/x'}$. We will reduce the general case to the above one by performing an \'etale base change.

Since the morphism $Y'\to Y$ factors through $\Nr(Y)$, it follows from Remark \ref{hensrem}(iii) that it suffices to prove the lemma for $\Nr(Y)$ and all preimages of $y$ instead of $Y$ and $y$. Thus, we can assume that $Y$ is normal. We claim that there exist an \'etale morphism $h{\colon}\oX\to X$ such that $\oX$ is integral, $h^{-1}(x)=\{\ox\}$, and $h$ is strictly \'etale at $\ox$, and a neighborhood $\oY$ of the preimage $\oy\in Y\times_X\oX$ of $y$ such that the morphism
$\oY\to\oX$ is integral. Indeed, $Y$ can be realized as the limit of $X$-schemes $Y_\alp$ of finite type so that $Y$ is integral over each $Y_\alp$. Since $y$ is isolated in the fiber over $x$, the same is true for its image $z$ in $Z=Y_\alp$ for a large enough $\alp$. Then by \cite[$\rm IV_4$, 18.12.1 and 18.12.2]{ega} there exists a strictly \'etale (over $x$) morphism $h\:\oX\to X$ and a neighborhood $\oZ$ of the preimage $\oz\in Z\times_X\oX$ of $z$ such that $\oZ\to Z$ is finite. So we can take this $h$ and set $\oY=\oZ\times_ZY$.

Note that $\oY$ is irreducible by Remark \ref{hensrem}(ii), hence it is integral and $n_{y/x}=n_{\oy/\ox}=[k(\oY):k(\oX)]$. Set $\oX'=\oX\times_X X'$ and let $\ox'$ be the preimage of $x'$, then it suffices to prove the lemma for the morphisms $\oY\to\oX$ and $\oX'\to\oX$ with points $\ox,\ox'$ and $\oy$ instead of the original data because the projections $\oX\to X$, $\oX'\to X'$ and $\oY\to Y$ are strictly \'etale at $\ox,\ox'$ and $\oy$, and hence $\oY'=\Nr_{k(\oY)}(\oY\times_\oX\oX')$ is strictly \'etale over $Y'$ at the preimage $\oy'_i$ of $y'_i$, and the matching henselian degrees are equal: $n_{y'_i/x'}=n_{\oy'_i/\ox'}$. It remains to recall that as we noted in the beginning of the proof, the case of $\oY,\oX$ and $\oX'$ follows from Theorem \ref{etaleth}(i).
\end{proof}

The lemma will be used to show that Theorem \ref{etaleth} admits the following refinement where the degrees do not appear.

\begin{theor}\label{etaleprop}
Let $f{\colon}Y\to X$ be a dominant nft morphism between integral schemes.
Let $y\in Y$ and $x=f(y)$ and assume that $y$ is isolated in its fiber. Consider the following:

(i) $f$ is strictly \'etale at $y$.

(ii) $f$ is strictly local-\'etale at $y$.

(iii) $f_y^\bir{\colon}Y_y^\bir\to X_x^\bir$ is bijective and $\calO\to\calO'$ is strictly local-\'etale for any $\calO\in X_x^\bir$ with preimage $\calO'\in Y_y^\bir$.

(iv) There exists $\calO\in X_x^\bir$ such that $(f_y^\bir)^{-1}(\calO)=\{\calO'\}$ and
$\calO\to\calO'$ is strictly local-\'etale.

Then (i)$\implies$(ii)$\implies$(iii)$\implies$(iv). If $X$ is unibranch at $x$ then (iv)$\implies$(iii). Furthermore, if $X$ is normal at $x$ then (iv)$\implies$(ii), and if $X$ is normal in a neighborhood of $x$ then (iv)$\implies$(i).
\end{theor}

Let us make a side remark on this result before going to the proof.

\begin{rem}
(i) Implication (ii)$\implies$(iii) is simple and intuitive. It would be natural to expect that the converse is true under mild restrictions (e.g. $X$ is unibranch at $x$). However, the real point of the theorem is the implication (iv)$\implies$(iii) showing that instead of checking the whole birational fiber $f_y^\bir$, it suffices to test its single (!) element. Slightly more generally, we will see in the proof that if $y$ is discrete in $f^{-1}(x)$ and $X_x^\bir$ is connected (i.e. $X$ is unibranch at $x$) then for any $\calO\in X_x^\bir$ the sum of (naturally defined) henselian degrees $n_{\calO_i/\calO}$ over $\calO_i\in(f_y^\bir)^{-1}(\calO)$ is constant (i.e. does not depend on the choice of $\calO$ in $X_x^\bir$).

(ii) Another subtle point of the theorem is that we do not make any finite presentation assumption. In general, a local-\'etale morphism of finite type does not have to be \'etale, see appendix \ref{etalesec}. However, such implication does hold whenever the target is integral, see Proposition \ref{localetale}. This allows us to obtain the implication (iv)$\implies$(i) for nft morphisms.
\end{rem}

\begin{proof}
The implication (i)$\implies$(ii) and (iii)$\implies$(iv) are obvious. To prove that (ii)$\implies$(iii) we assume that $f$ is strictly local-\'etale at $y$. Shrinking $X$ we can also assume that $f^{-1}(x)=\{y\}$. Fix an element $\calO\in X_x^\bir$ and set $X'=\Spec(\calO)$. Let $x'\in X'$ be the closed point and let $f'{\colon}Y'=Y\times_XX'\to X'$ be the base change morphism. Then $f'^{-1}(x')=\{y'\}$ and $f'$ is strictly local-\'etale at $y'$. As $\calO'=\calO_{Y',y'}$ is local-\'etale over the valuation ring $\calO$, it is itself a valuation ring and we obtain that $\calO'\in(f_y^\bir)^{-1}(\calO)$. It remains to show that any other valuation ring $\calO''\in(f_y^\bir)^{-1}(\calO)$ coincides with $\calO'$. The morphism $\Spec(\calO'')\to Y\times_XX'$ sends the closed point to $y'$ and hence factors through $\Spec(\calO')$. Thus, $\calO''$ contains $\calO'$ and hence is a localization of $\calO'$. But $\Spec(\calO')$ has the unique point $y'$ above $x'$ and this implies that $\calO''=\calO'$.

Assume, now, that $X$ is unibranch at $x$. Fix $\calO\in X_x^\bir$ and apply Lemma \ref{lem1} with $X'=\Spec(\calO)$ to compute $n_{y/x}$. We obtain that $n_{y/x}=\sum_{i=1}^m n_{y'_i/x'}$ where $y'_1\. y'_m$ are the points of $Y':=\Nr_{k(Y)}(Y\times_X X')$ that sit above $y$ and the closed point $x'\in X'$. We claim that $\calO'_i:=\calO_{Y',y'_i}$ are valuation rings and $\{\calO'_1\.\calO'_m\}=(f_y^\bir)^{-1}(\calO)$.

Since $y$ is discrete in the fiber, $k(Y)/k(X)$ is finite and therefore $C=\Nr_{k(Y)}(\calO)$ is a semi-local ring whose localizations are the valuation rings of $k(Y)$ that contain $\calO$. In particular, $C$ is a Pr\"ufer ring, see \cite[Ch. VII, \S2, Exercise 12]{Bou}. Any $C$-subring of $k(Y)$ is a localization of $C$ (possibly infinite), hence the ring $B=C\calO_{Y,y}$ generated by $C$ and $\calO_{Y,y}$ is a localization of $C$. In particular, $B$ is integrally closed and hence coincides with $\Nr_{k(Y)}(\calO\calO_{Y,y})$, and therefore $$\Spec(B)=\Nr_{k(Y)}(\Spec(\calO_{Y,y})\times_XX')=\Spec(\calO_{Y,y})\times_YY'$$ is a localization of $Y'$ that contains all the $y'_i$'s. It remains to recall that the local rings of the preimages of $x'$ in $\Spec(C)$ are exactly the valuation rings of $k(Y)$ that extend $\calO$, and hence any preimage $y'\in\Spec(B)$ of $x'$ corresponds to a valuation ring of $k(Y)$ that extends $\calO$ and contains $\calO_{Y,y}$, i.e. to an element of $(f_y^\bir)^{-1}(\calO)$.

Now, we can sum over the elements of $(f_y^\bir)^{-1}(\calO)$ to compute $n_{y/x}$. If (iv) is satisfied then the fiber consists of a single element $y'_1$ and $n_{y'_1/x'}=1$ because $\calO'\subseteq\calO^h$. Therefore $n_{y/x}=1$ and the same argument shows that any other fiber $(f_y^\bir)^{-1}(A)$ for $A\in X_x^\bir$ is of the form $\{A'\}$ with strictly local-\'etale $A'/A$ (note that by \cite[2.1.6,2.1.7]{temst}, $\calO'/\calO$ is strictly \'etale when the height of $\calO$ is finite, but this does not have to be true in general). This proves the implication (iv)$\implies$(iii).

If $X$ is normal at $x$ then, as we saw, $n_{y/x}=1$ and by Theorem \ref{etaleth} (iv), $f$ is strictly local-\'etale at $y$. Finally, if $X$ is normal in a neighborhood of $x$ then
Theorem \ref{etaleth} (iv) implies that $f$ is strictly \'etale at $x$.
\end{proof}

\subsection{Analytic generic fiber}\label{agfsec}
Until the end of \S\ref{anfibsec} we assume that $k$ is a valued field of height one with a non-zero element $\pi\in\kcirccirc$ and completion $\hatk$. We set also $\eta=\Spec(k)\into S=\Spec(\kcirc)$, $s=\Spec(\tilk)=S\setminus\eta$ and $\gtS=\Spf(\hatkcirc)$. For any $S$-scheme $X$ its {\em generic fiber} is defined as $X_\eta=X\times_S\eta$, and by the {\em closed fiber} $X_s$ we mean the preimage of $s$ with the reduced scheme structure. Caution: $X_s$ is not the schematic fiber over $s$ but its reduction. The $(\pi)$-adic formal completion of $X$ will be denoted $\gtX$; it is a formal $\gtS$-scheme with the closed fiber $\gtX_s\toisom X_s$. If $X$ is of finite type/presentation over $S$ then so is $\gtX$ over $\gtS$.

In this section and in \S\ref{chaptwo} we will work with $\hatk$-analytic spaces introduced by Berkovich. Almost all our results hold for general analytic spaces as introduced in \cite{berihes}, but to make the reading of the paper simpler we will mainly work with good analytic spaces introduced in \cite{berbook}. These are analytic spaces in which each point possesses an affinoid neighborhood. If not said to the contrary, it will be automatically
assumed that the spaces are good and strictly analytic. In particular, these analytic spaces correspond to rigid analytic spaces. We will make a heavy use of non-rigid points however. Sometimes, we will remark that our results hold more generally without goodness and/or strict analyticity assumption, but (up to one explicitly mentioned exception) these notes will not be used later and can be ignored by the reader.

Let us recall some terminology. For a $\hatk$-analytic space $Y$ with a point $y$ by $\calO_{Y,y}$ we denote the local ring of $Y$ at $y$ (it behaves reasonably well because $Y$ is good), by $\kappa(y)$ we denote the residue field $\Frac(\calO_{Y,y}/m_y)$ and by $\calH(y)$ we denote the completed residue field $\wh{\kappa(y)}$. For any formal scheme $\gtX$ of finite presentation over $\gtS$, Berkovich defined in \cite[\S1]{berform} its generic fiber $\gtX_\eta$ as a compact (not necessarily good) strictly $\hatk$-analytic space (note that $\eta$ is only a formal part of notation here). In particular, if $\gtX=\Spf(A)$ is affine then $\gtX_\eta=\calM(\calA)$ is affinoid with $\calA=A_\pi\toisom A\otimes_{\hatkcirc}\hatk$. Also, Berkovich defined an anti-continuous reduction map $\pi_\gtX{\colon}\gtX_\eta\to\gtX_s$ in the sense that preimages of open sets are closed and vice versa (recall that affinoid domains are closed in analytic geometry). In particular, to any $X$ of finite presentation over $S$ we can functorially associate its analytic generic fiber $\gtX_\eta$ with the reduction map $\pi_X{\colon}\gtX_\eta\to\gtX_s\toisom X_s$: complete $X$ and take the generic fiber of $\gtX=\hatX$. Moreover, this construction works for any $X$ of finite type because $\calA=\hatA_\pi=(\hatA/I)_\pi$, where $I$ is the $\pi$-torsion ideal, but $\hatA/I$ is of topologically finite presentation over $\hatkcirc$ by \cite[1.1(c)]{BL}, and so $\calA$ is a $k$-affinoid algebra.

\begin{rem}\label{genfibaffrem}
(i) One can give a more explicit description of the analytic generic fiber as follows. If $A$ is a finitely presented $\kcirc$-algebra then $A$ is of the form $\kcirc[T_1\. T_m]/(f_1\. f_n)$, hence we have that $\hatA=\hatkcirc\{T_1\. T_m\}/(f_1\. f_n)$ and $\hatA_\pi= \hatk\{T_1\. T_m\}/(f_1\. f_n)$. In particular, for $X=\Spec(A)$ the analytic generic fiber $\gtX_\eta=\calM(\hatA_\pi)$ is the affinoid (perhaps empty) domain given by the conditions $|T_i|\le 1$ in the analytification of the $\hatk$-scheme $X_\eta\otimes_k\hatk =\Spec(\hatk[T_1\. T_m]/(f_1\. f_n))$. We refer the reader to \cite[\S3.4]{berbook} for the definition of this analytification $(X_\eta\otimes_k\hatk)^\an$.

(ii) Using the above description we can describe the kernel $I=\cap_{n=0}^\infty\pi^nA$ of the completion homomorphism $A\to\hatA$ when $A$ is reduced and $\kcirc$-flat. We claim that $I$ consists of the functions vanishing on all irreducible components of $X=\Spec(A)$ with non-empty closed fiber, so $X'=\Spec(A/I)$ is obtained by removing from $X$ all irreducible components with empty closed fiber. The claim easily reduces to the following: if $A$ is integral and $\kcirc$-flat and $X_s$ is not empty then $I=0$. Note that $\gtX_\eta$ is not empty because $\gtX$ is $\hatkcirc$-flat and non-empty. Since $\hatA=\wh{A/I}$, we obtain that $\gtX_\eta$ is a non-empty affinoid domain in the analytifications of both $X_\eta\otimes_k\hatk$ and $X'_\eta\otimes_k\hatk$. Hence the latter are of equal dimensions, and it follows that $I=0$.
\end{rem}

In the sequel, we will mainly be interested in the case of an $\eta$-normal $X$. Since $\eta$-normalization can take us outside of the category of $S$-schemes of finite type, we have to extend the construction of the analytic generic fiber to $\eta$-nft $S$-schemes. The following result implies, in particular, that when working with reduced flat $S$-schemes we do not have to distinguish between nft, $\eta$-nft and $\eta$-nfp $S$-schemes.

\begin{lem}\label{fintypelem}
Let $X$ be an $S$-scheme then

(i) If $X$ is $\eta$-nft then it is $S$-flat and $\eta$-nfp.

(ii) Assume that $X$ is reduced. Then $X$ is nft if and only if it is $\eta$-nft.
\end{lem}
\begin{proof}
Assume that $X$ is $\eta$-nft. Then $X$ is flat over $S$ because any partial $\eta$-normalization kills $\pi$-torsion. Therefore, $X$ is a partial $\eta$-normalization of a flat $S$-scheme $\oX$ of finite type. But $\oX$ is automatically of finite presentation over $S$ by \cite[3.4.7]{RG}, hence $X$ is $\eta$-nfp. (Actually, this is a particular case of Remark \ref{Yrem1}.)

Next, let us prove (ii). If $X$ is nft then it is a partial normalization of a reduced flat scheme $\oX$ of finite type over $S$, hence $X$ is the projective limit of finite modifications $X_\alp$ of $\oX$. But $X_\eta$ is of finite type over $k$, hence already some $X_{\alp,\eta}$ is isomorphic to $X_\eta$, and then $X$ is a partial $\eta$-normalization of $X_\alp$. Conversely, if $X$ is reduced and $\eta$-nft then it is a partial $\eta$-normalization of a finite type $S$-scheme $Y$ with reduced $Y_\eta$. It follows that $X$ is a partial normalization of the schematic closure of $Y_\eta$ in $Y$, and hence $X$ is nft.
\end{proof}

Recall that for a $\hatk$-affinoid algebra $\calA$, it is standard to denote the subring of power-bounded elements of $\calA$, the ideal of power-nilpotent elements and the reduction as $\calAcirc$, $\calAcirccirc$ and $\tilcalA=\calAcirc/\calAcirccirc$, respectively.

\begin{lem}\label{normlem}
Assume that $A$ is a flat $\kcirc$-algebra of finite type such that $A\otimes_{\kcirc}\hatk$ is reduced (for example, this automatically happens when $A_\pi$ is geometrically reduced over $k$), and let $A'$ be the integral closure of the image of $A$ in $A_\pi$. Then $\hatA'\toisom\calAcirc$ for the $\hatk$-affinoid algebra
$\calA=\hatA_\pi$.
\end{lem}
\begin{proof}
Note that the kernel $I=\cap_{n=0}^\infty\pi^nA$ of the completion homomorphism $A\to\hatA$ is also an ideal in $A_\pi$, and hence an ideal in $A'$. Since $\Spec(A/I)$ is obtained from $X=\Spec(A)$ by removing all irreducible components with empty closed fiber, one easily sees that $A'/I$ is the integral closure of $A/I$ in $(A/I)_\pi=A_\pi/I$. Also, $A/I$ is $\kcirc$-flat because it has no $\pi$-torsion. In particular, $A/I$ satisfies the assumption of the lemma, and since $\hatA\toisom\wh{A/I}$ and $\hatA'\toisom\wh{A'/I}$, it
suffices to prove the lemma for $A/I$ instead of $A$. Thus, we can assume that $A\into\hatA$.

Choose any surjective homomorphism $\hatkcirc\{T_1\. T_n\}\to\hatA$, then inverting $\pi$ we obtain a surjective homomorphism of affinoid algebras $\phi{\colon}\hatk\{T_1\. T_n\}\to\calA$. Note that $\calA$ is reduced because $\calM(\calA)$ is an affinoid domain in the analytification of the reduced $\hatk$-scheme $\Spec(A\otimes_{\kcirc}\hatk)$, and analytification preserves reducedness by a GAGA-type theorem \cite[3.4.3]{berbook}. Since $\calA$ is reduced, \cite[6.2.4/1]{BGR} asserts that $\phi$ induces a norm on $\calA$ which is equivalent to the spectral norm. So, $\phi(\hatkcirc\{T_1\. T_n\})$ contains an ideal
$\omega\calAcirc$ for a non-zero element $\omega\in\kcirccirc$, in particular, $\omega\calAcirc\into\hatA$. Since $A\into\hatA$, we have the inclusion $A'\into A_\omega\into(\hatA)_\omega=\calA$. One easily sees that $\calAcirc$ is integrally closed in $\calA$, hence $A'\into\calAcirc$ and we obtain the embedding $\omega A'\into\omega\calAcirc\into\hatA$. It follows that $\omega A'\into A$ because $\hatA\cap A_\omega=A$ in $\calA$ (the latter is obvious since $\hatA$ is the $(\omega)$-adic completion of $A$). Since $\omega A'$ is an open ideal in $A'$, we obtain that $A$ is an open subring of $A'$, and then $\hatA$ is an open subring of $\hatA'$ containing open ideals $\omega\hatA'\subset\omega\calAcirc$. Note that $\hatA'$ has no $\omega$-torsion because $A'$ has no $\omega$-torsion, and hence the embedding $A'\into\calAcirc$ factors through the embedding $i{\colon}\hatA'\into\calAcirc$.

We have to establish the surjectivity of $i$. Note that for any $\omega\in\kcirc$ we have that $\omega\hatA'\cap A'=\omega A'$. Assume, now, that $\hatA'\subsetneq\calAcirc$. Then there exist elements $a\in\hatA'$ and $\omega\in\kcirc$ such that $a/\omega$ does not belong to $\hatA'$ but is integral over it, in particular, we can find $m\in\bfN$ and $b_j\in\hatA'$ such that $x=a^m+b_1a^{m-1}\omega+\dots +b_{m-1}a\omega^{m-1}\in\omega^m\hatA'$. The inclusion survives when we move $a$ and $b_j$'s slightly, hence we can achieve, in addition, that $a$ and $b_j$'s are in $A'$, but $a/\omega\notin\hatA'$. Then $x\in\omega^m\hatA'\cap A'=\omega^mA'$, and so $a/\omega\in A'_\omega$ is integral over $A'$. But the latter contradicts our assumption that $A'$ is integrally closed.
\end{proof}

\begin{cor}\label{bijlem}
If $X=\Spec(A)$ is an affine flat $S$-scheme of finite type and with reduced $X_\eta\otimes_k\hatk$, and $X'=\Spec(A')$ is a partial $\eta$-normalization of $X$, then $\hatA'$ is an open $\hatA$-subalgebra of $\calAcirc$, where $\calA=\hatA_\pi$. In particular, $\hatA'_\pi\toisom\calA$ is $k$-affinoid.
\end{cor}
\begin{proof}
Let $A''$ be the integral closure of $A$ in $A_\pi$. We proved above that $\omega A''\subset A$ for a non-zero $\omega\in\kcirc$, hence $A$ contains an open ideal $\omega A'$ and therefore $\hatA\subset\hatA'\subset\hatA''=\calAcirc$.
\end{proof}

Using Lemma \ref{normlem} we can extend the construction of analytic generic fibers and reduction maps to affine $\eta$-nft $S$-schemes $X$ such that $X_\eta\otimes_k\hatk$ is reduced (for example, $X_\eta$ is geometrically reduced): to each such scheme $X=\Spec(A)$ we associate the affinoid space $\gtX_\eta=\calM(\hatA_\pi)$. We define the {\em reduction map} $\pi_X{\colon}\gtX_\eta\to X_s$ as follows: if $X=\Spec(A)$, $\gtX_\eta=\calM(\calA)$ and $x\in\gtX_\eta$ is a point then the character $\calA\to\calH(x)$ induces a character $A\to\hatA\to\calAcirc\to\calH(x)^\circ\to\wHx$, which defines a point on $X_s$. If $X'$ denotes the $\eta$-normalization of $X$ then $\pi_X$ is the composition of the reduction map $\gtX_\eta\to\gtX'_s$, which is surjective and anti-continuous by \cite[2.4.1]{berbook}, and the projection $\gtX'_s\toisom X'_s\to X_s$. Hence, $\pi_X$ is surjective and anti-continuous.

\begin{lem}\label{bijlem2}
Let $X$ be an affine $\eta$-nft $S$-scheme with reduced $X_\eta\otimes_k\hatk$, and let $X'$ be any partial $\eta$-normalization of $X$, then

(i) the morphism $\gtX'_\eta\to\gtX_\eta$ of analytic generic fibers is an isomorphism,

(ii) the closed fiber $X_s$ is of finite type over $\tilk$,

(iii) there exists an affine reduced flat $S$-scheme $\oX$ of finite presentation such that $X$ is a partial $\eta$-normalization of $\oX$ and the projection $X\to\oX$ induces an isomorphism $X_s\to\oX_s$ on the closed fibers. In particular, $X\to\oX$ is bijective.
\end{lem}
\begin{proof}
Let $X''$ be the $\eta$-normalization of $X$. By Corollary \ref{bijlem}, $\gtX_\eta$ and $\gtX'_\eta$ are isomorphic to $\gtX''_\eta$, so we obtain (i). Furthermore, $\gtX''_s=\Spec(\tilcalA)$, hence it is of finite type over $\tilk$ by \cite[6.3.4/3]{BGR}. Choose any affine reduced flat $S$-scheme $\oX$ of finite presentation such that $X$ is a partial $\eta$-normalization of $\oX$ (we use Lemma \ref{fintypelem}(i)). Then the
morphisms $X''\to X\to\oX$ induce surjective integral morphisms on closed fibers $X''_s\to X_s\to\oX_s$. Since $X_s$ is reduced and $X''_s$ is of finite type over $\tilk$, $X_s$ is of finite type over $\tilk$. This proves (ii), and it is clear now that replacing $\oX$ with a sufficiently large finite $\eta$-modification dominated by $X$ we achieve that $X_s\toisom\oX_s$.
\end{proof}

In the sequel, we will use only the second part of the following remark. Actually, the latter will only be used in the proof of Lemma \ref{closlem}.

\begin{rem}\label{genfibrem}
(i) The definitions of the analytic generic fiber $\gtX_\eta=\calM(\hatA_\pi)$ and the reduction map $\pi_X{\colon}\gtX_\eta\to X_s$ make sense for any affine $S$-scheme $X$ of finite type or of $\eta$-normalized finite type. However, if $\calX=X_\eta\otimes_k\hatk$ is not reduced then $\gtX_\eta$ can be an affinoid domain in a closed subspace of $\calX^\an$ obtained by killing some nilpotent functions -- certain nilpotent elements of $A_1=A\otimes_{\kcirc}\hatkcirc$ can be infinitely $\pi$-divisible and then they are killed by passing to the separated completion $\hatA_1=\hatA$. Also, if $X$ is not $S$-flat (in the finite type case) then $\pi_X$ does not have to be surjective.

(ii) The above constructions commute with localizations, hence to any $S$-scheme $X$ of finite type or of $\eta$-normalized finite type one can functorially associate a strictly analytic generic fiber $\gtX_\eta$ with an anti-continuous reduction map $\pi_X{\colon}\gtX_\eta\to X_s$. However, $\gtX_\eta$ is not good already when $X=\bfA^2_S\setminus S$ (the relative $\bfA^2$ with punched origin).
\end{rem}

\begin{lem}\label{extcomlem}
Let $X$ be an affine $\eta$-nft scheme over $S$ such that its generic fiber is geometrically reduced, and let $l/k$ be a finite extension of valued fields with $S_l=\Spec(\lcirc)$. Then $\gtX_{l,\eta}:=\gtX_\eta\otimes_\hatk\hatl$ is the analytic generic fiber of $X_l:=X\times_S S_l$ (considered either as an $S_l$-scheme or $S$-scheme). In particular, if $X_{l,\eta}$ is normal and integral with generic point $\Spec(L)$ then the analytic generic fiber of $\Nr_L(X_l)$ is isomorphic to $\gtX_\eta\otimes_\hatk\hatl$.
\end{lem}
Note that the affineness assumption can be removed due to Remark \ref{genfibrem}(ii).
\begin{proof}
Assume that $X=\Spec(A)$, so that $X_l=\Spec(A_l)$ for $A_l=A\otimes_{\kcirc}\lcirc$. The analytic generic fiber of $X_l$ is defined as $\calM((\wh{A_l})_\pi)$ where $\pi\in\kcirccirc\setminus\{0\}$ and the completion is $(\pi)$-adic. In particular, it is not important for the construction of $\gtX_{l,\eta}$ whether we view $X_l$ as an $S_l$-scheme or $S$-scheme. Now we use that $\wh{A_l}\toisom\wh{A\otimes_{\kcirc}\lcirc}\toisom\hatA\wtimes_{\hatkcirc}\hatlcirc$, hence
$(\wh{A_l})_\pi\toisom\hatA_\pi\wtimes_\hatk\hatl\toisom\hatA_\pi\otimes_\hatk\hatl$, and applying the functor $\calM$ we obtain that $\gtX_{l,\eta}\toisom\gtX\otimes_\hatk\hatl$. The last claim follows from Lemma \ref{normlem} because $\Nr_L(X_l)=\Spec(A'_l)$ where $A'_l$ is the integral closure of $A_l$ in $(A_l)_\pi$, and hence $\Nr_L(X_l)$ and $X_l$ have isomorphic analytic generic fibers.
\end{proof}

We conclude this section with one more definition. For any point $x\in X_s$, by the {\em analytic fiber over $x$} we mean the preimage $X^\an_x=\pi^{-1}_X(x)$. If $x$ is closed then $X^\an_x$ is open, so we regard it as an open analytic subspace in $\gtX_\eta$. (We do not need this, but one can show that in general $X^\an_x$ can be provided with a structure of an analytic $k$-space, i.e. an analytic space over a larger analytic field $K$, though the choice of $K$ is not canonical.) Note also that $X^\an_x$ is an analytic domain in the larger space $\calX^\an$, where $\calX=X\otimes_k\hatk$. It follows from \cite[\S2.5]{berbook} that the space $X^\an_x$ has no boundary. Hence the analytic domain embedding $X^\an_x\into\calX^\an$ has no boundary, and therefore $X^\an_x$ is open also in $\calX^\an$.

\subsection{Analytic criterion of \'etaleness}\label{anfibsec}
Throughout this section $k$, $S$ and $\gtS$ are as in \S\ref{agfsec}. Our main result is the following theorem.

\begin{theor}\label{anfibth}
Let $f{\colon}Y\to X$ be a morphism of integral affine flat $\eta$-nft $S$-schemes of such that $Y_\eta\otimes_k\hatk$ and $X_\eta\otimes_k\hatk$ are reduced. Let also $y\in Y_s$ be a closed point with $x=f(y)$ and let $Y^\an_y$, $X^\an_x$ be the corresponding analytic fibers. Assume, finally, that $X$ is normal in a neighborhood of $x$. Then $f$ is strictly \'etale at $y$ if and only if the natural map $Y^\an_y\to X^\an_x$ is an isomorphism.
\end{theor}
Probably, it is enough to take $X$ to be $\eta$-normal in the assumptions of the theorem, but we cannot attack this case with our methods (due to assumptions in Theorem \ref{etaleprop}).
\begin{proof}
The direct implication is easier to prove and it holds even without the normality assumption on $X$. Assume that $f$ is strictly \'etale at $y$. Shrinking $X$ and $Y$ we can keep them affine and achieve that $f$ is of finite presentation. By Lemma \ref{bijlem2} there exists a finitely presented $S$-scheme $X'$ such that $X$ is a partial $\eta$-normalization of $X'$ and the morphism $X\to X'$ is bijective. Then $X$ is isomorphic to the projective limit of finite $\eta$-modifications $X_\alpha$ of $X'$ by Lemma \ref{Ypartlem}, and the projections $X\to X_\alp$ are bijective. By \cite[$\rm IV_3$, 8.8.2]{ega} and \cite[$\rm IV_4$, 17.7.8]{ega}, $f$ is the base change of an \'etale morphism $f_\alp{\colon}Y_\alpha\to X_\alpha$. Since $X^\an_\eta\toisom X_{\alp,\eta}^\an$ and $Y^\an_\eta\toisom Y_{\alp,\eta}^\an$ by Lemma \ref{bijlem2}(i), it suffices to prove the claim for $f_\alpha$. So, we can assume that $X$ is of finite $S$-presentation. Since $f$ is \'etale at $y$, so are the morphisms $f_n=f\times_S \Spec(\kcirc/(\pi^n))$. Hence the $(\pi)$-adic completion $\gtf{\colon}\gtY\to\gtX$ is \'etale at $y$ (see \cite[\S1]{bercontr} for the definition of \'etale morphisms of formal schemes). Then it follows from \cite[4.4]{bercontr} that $Y_y^\an\toisom X_x^\an$.

Assume, now, that $Y_y^\an\toisom X_x^\an$. Then it follows from the dimension considerations that $X$ and $Y$ are of equal dimension and $f$ is dominant, in particular, we obtain a finite extension of fields $k(Y)/k(X)$. We claim that the extension is separable, and to prove this let us assume to the contrary that $k(Y)/k(X)$ is inseparable. Then the morphism $Y\to X$ factors through a finite morphism $Y\to Z$ such that $Z$ is integral and $k(Y)/k(Z)$ is inseparable of degree $p$. Let $\calY\to\calZ\to\calX$ be the morphisms obtained from $Y\to Z\to X$ by applying $\cdot\otimes_{\kcirc}\hatk$. Then $\calZ$ is reduced because $\calY$ is reduced by the assumption of the theorem. Since $\calY\to\calZ$ is finite and generically inseparable of degree $p$, it follows that for any point $z\in\calZ^\an$ with $m_z=0$ the fiber over $z$ in $\calY^\an$ is of the form $\calM(\calC)$ where $\calC$ is a ramified local Artin $\calH(z)$-algebra of dimension $p$. In particular, the morphism $\phi{\colon}\calY^\an\to\calX^\an$ cannot be a local isomorphism at any point $t\in\calY^\an$ with $m_t=0$ because its $\calX^\an$-fiber is not geometrically reduced at $t$. This contradicts the assumption that the map $Y^\an_y\to X^\an_x$, which is the restriction of $\phi$ on open subspaces, is an isomorphism because the points with trivial maximal ideal are dense in any reduced analytic space. So, the assumption that $k(Y)/k(X)$ is inseparable was incorrect.

We will need the following lemma, where, as a matter of exception, we allow non-good spaces (in the proof, we will have to leave the framework of good spaces anyway).

\begin{lem}\label{closlem}
Let $Y$ be an $\eta$-nft $S$-scheme with reduced $Y_\eta\otimes_k\hatk$, and let $y,z\in Y_s$ be two points such that $y$ is a closed specialization of $z$. Then the analytic fiber $Y_z^\an$ is contained in the closure of the analytic fiber $Y_y^\an$.
\end{lem}
The assumption that $y$ is closed is unnecessary but simplifies the proof.
\begin{proof}
Set $Z=Y\setminus\{y\}$, then $\gtZ_\eta$ is an analytic domain in $\gtY_\eta$ obtained by
removing $Y_y^\an$. Choose any point $\gtz\in Y_z^\an$. The germ reductions $\wt{(\gtY_\eta)}_\gtz$ and $\wt{(\gtZ_\eta)}_\gtz$, as defined in \cite[\S2]{temred1}, are the birational spaces from the category $\bir_\tilk$ corresponding to the pointed schemes $\Spec(\wt{\calH(\gtz)})\to\oY$ and $\Spec(\wt{\calH(\gtz)})\to\oZ$, where $\oY$ and $\oZ$ are the Zariski closures of $z$ in $Y$ and $Z$, respectively. Since the open immersion $\oZ\to\oY$ is not an isomorphism, the embedding $\wt{(\gtZ_\eta)}_\gtz\to\wt{(\gtY_\eta)}_\gtz$  is not an isomorphism, and \cite[2.4]{temred1} implies that the embedding of germ subdomains $(\gtZ_\eta,\gtz)\to(\gtY_\eta,\gtz)$ is not an isomorphism. Thus, $\gtZ_\eta$ is not a neighborhood of $\gtz$ in $\gtY_\eta$, and we obtain that $\gtz$ belongs to the closure of $Y_y^\an$.
\end{proof}

\begin{lem}\label{discrlem}
If $f{\colon}Y\to X$ is a morphism of affine $\eta$-nft $S$-schemes with reduced $X_\eta\otimes_k\hatk$  and $Y_\eta\otimes_k\hatk$, $y\in Y_s$ is a closed point with $x=f(y)$ and $Y^\an_y\to X^\an_x$ is an isomorphism, then $y$ is discrete in the fiber over $x$.
\end{lem}
Our proof shows a more general result that $y$ is discrete if $Y^\an_y$ is a connected component of the fiber over $X_x^\an$.
\begin{proof}
Assume that $y$ is not discrete in the fiber contrary to the assertion of the lemma. Then there exists a point $z\in Y_s$ which is a generalization of $y$ and lies in the fiber of $x$. Since the reduction map $\gtY_\eta\to Y_s$ is surjective, there exists a point $\gtz\in\gtY_\eta$ in the analytic fiber over $z$. By the construction, $\gtz\notin Y_y^\an$ but its image in $\gtX_\eta$ lies in $X_x^\an$. Since $f^\an{\colon}\gtY_\eta\to\gtX_\eta$ induces the isomorphism $Y^\an_y\toisom X^\an_x$ of open subspaces and $\gtX_\eta$ and $\gtY_\eta$ are Hausdorff topological spaces, $\gtz$ is not contained in the closure of $Y_y^\an$. This contradicts Lemma \ref{closlem}, hence our assumption that $y$ is not discrete in the fiber was wrong.
\end{proof}

Now, we are prepared to prove that $f$ is strictly \'etale at $y$. Let $X=\Spec(A)$ and $Y=\Spec(B)$. We would like to use the \'etaleness criterion \ref{etaleprop}. Note that $Y$ is nft over $S$ by Lemma \ref{fintypelem}(ii) and $f$ is nft by Corollary \ref{nftprop}(ii). Since we proved that $y$ is discrete in its fiber over $X$, we have only to find a valuation ring $\calO$ as in Theorem \ref{etaleprop}(iv). Since $X$ is integral and with non-empty $X_s$, the completion homomorphism $A\to\hatA$ is injective by Remark \ref{genfibaffrem}(ii). Choose any point $z$ with $m_z=0$ in the analytic fiber over $x$, then the embeddings $A\into\hatA_\pi\into\calH(z)$ give rise to an embedding $k(X)\into\calH(z)$, and hence $z$ induces a valuation of height one on $k(X)$. Moreover, this valuation is centered on $x$ because $x=\pi_X(z)$. Let $\calO$ be the corresponding valuation ring of $k(X)$, i.e. $\calO=k(X)\cap\calH(z)^\circ$, and consider any extension $\calO'$ of $\calO$ to $k(Y)$ which is centered on $y$. Note that $\calO'$ induces a point $z'\in Y_y^\an$ with $\calH(z')^\circ=\hatcalO'$ because the homomorphism $B\into\hatcalO'$ factors through $\hatB$ (so, $z'$ corresponds to the character $\hatB_\pi\to\hatcalO'_\pi$). Obviously, $z'$ lies over $z$, hence by our assumption on the analytic fibers, $z'$ is uniquely determined and $\calH(z')\toisom\calH(z)$. In particular, $\calO'=\calH(z')^\circ\cap k(Y)$ is uniquely determined and the completions of $k(X)$ and $k(Y)$ along the valuations corresponding to $\calO$ and $\calO'$ are isomorphic. We proved earlier that $k(Y)/k(X)$ is separable, hence \cite[2.2.1]{temst} implies that $\calO'$ is local-\'etale over $\calO$. But the residue fields of $\calO$ and $\calO'$ are isomorphic to the residue field of the completions $\hatcalO\toisom\hatcalO'$, hence $\calO'/\calO$ is strictly local-\'etale and we are done.
\end{proof}

In the last part of the proof we used a connection between the birational and the analytic fibers over $x$. It is related to the following construction that will be used in the sequel.

\begin{rem}\label{complrem}
Assume that $X=\Spec(A)$ is an integral affine nft $S$-scheme and $\calO\in X^\bir_x$ is of height one. Then $\calO$ is a valuation of $K=k(X)$ such that $\calO\cap k=\kcirc$. Consider the homomorphism $A\to\calO$. Passing to the $\pi$-adic completions and inverting $\pi$ we obtain a character $\calA\to\hatK$ which gives rise to a point of $\gtX_\eta$. Then $A\cap\hatKcirccirc=A\cap m_\calO$ and hence the analytic point is contained in $X^\an_x$. This establishes a map $\psi_x{\colon}X^{\bir,1}_x\to X^\an_x$, where the source consists of all points of $X^\bir_x$ of height one. Clearly, this construction is of local nature and hence makes sense for any integral nft $S$-scheme $X$.
\end{rem}

For the sake of completeness, we discuss below how to extend the above construction to the whole $X^\bir_x$. We will not need the following remark in the sequel.

\begin{rem}\label{complrem1}
Note that restriction of the valuation induces a map of the Riemann-Zariski spaces $\RZ_K(X)\to\RZ_k(S)=S$ and let $\RZ_K(X)_s$ denote the preimage of the closed point of $S$. We will see that $\psi_x$ can be extended to a continuous map $\psi{\colon}\RZ_K(X)_s\to\gtX_\eta$, though we warn the reader that $\psi$ does not map the whole $X_x^\bir$ to $X_x^\an$. Here are two constructions of $\psi$. The first one is a straightforward generalization of the construction of $\psi_x$. The second one is less explicit, but its advantage is that the constructed map is obviously continuous.

(i) Let $X=\Spec(A)$. Given a valuation ring $\calO\in\RZ_K(X)_s$ consider the prime ideals $p_0=\cap_{n=0}^\infty \pi^n\calO$ and $p=\sqrt{\pi\calO}$. Then $R=\calO_p/p_0\calO_p$ is a valuation ring over $\kcirc$ of height one (we localized by elements $y$ such that
$|\pi|<|y^n|$ for any $n$ and we factored by elements $y$ such that $|y|<|\pi|^n$ for any $n$). Since $A\subset\calO$ in $k(X)$, we get a homomorphism $A\to R$. Taking the $(\pi)$-adic completion and inverting $\pi$ we obtain a continuous homomorphism $\calA=\hatA_\pi\to\hatK$ where $K=\Frac(R)$. Clearly, the image of $\calA$ is dense in $\hatK$, so we get a point $z$ with $\calH(z)\toisom\hatK$ in the space $\gtX_\eta=\calM(\calA)$.

(ii) Alternatively, $\psi$ naturally arises due to the following three facts known to experts: (a) $\RZ_K(X)$ is homeomorphic to the projective limit of blow ups of $X$, and
hence admits a natural map to the projective limit of the blow ups of $X$ along open ideals (in the $(\pi)$-adic topology), (b) the adic analytic space $\gtX^\ad_\eta$ is homeomorphic to the projective limit of all admissible formal blow ups of $\gtX$, so we get a map $\RZ_K(X)_s\to\gtX^\ad_\eta$, (c) $\gtX^\an_\eta=\gtX_\eta$ is the maximal Hausdorff quotient of $\gtX^\ad_\eta$. To the best of my knowledge, facts (b) and (c) are not proved in the literature, though they are not difficult and are mentioned in a letter of P. Deligne and in \cite{FK}.
\end{rem}

\subsection{Smooth-equivalence}\label{smsec}

\begin{defin}\label{smdef}
Let $S$ be a scheme and $X,Y$ be two $S$-schemes. We say that points $x\in X$ and $y\in Y$ are {\em smooth-equivalent over $S$} if there exists an $S$-scheme $Z$ with a point $z\in Z$ and smooth $S$-morphisms $Z\to X$ and $Z\to Y$ which map $z$ to $x$ and $y$, respectively (alternatively, one could say that $X$ and $Y$ are smooth-locally $S$-isomorphic at $x$ and $y$). Often, we will write that $(X,x)$ and $(Y,y)$ are smooth-equivalent to stress the dependence on $X$ and $Y$. Note that in Vakil's paper \cite{Va} on Murphy's law in algebraic geometry, pointed schemes $(X,x)$ and $(Y,y)$ are said to have the same singularity type if $x$ and $y$ are smooth-equivalent.
\end{defin}

For example, for a field $k$ and a $k$-variety $X$, a point $x\in X$ is smooth-equivalent to $(\Spec(k),\Spec(k))$ if and only if $X$ is $k$-smooth at $x$. We will use this notion to pass from a point $x\in X$ to a smooth-equivalent point $y\in Y$ with $\dim(Y)<\dim(X)$. If such $y$ and $Y$ exist then, in some sense, the essential dimension of the singularity at $x$ is smaller than the dimension of $X$ at $x$.

Let $l$ be an analytic field. In the sequel we will need a notion of {\em open} (resp. {\em closed}) {\em unit $l$-polydisc}, by which we mean the subdomain in the analytic space $\bfA^n_l=\Spec(l[t_1\. t_n])^\an$ given by $|t_i|<1$ (resp. $|t_i|\le 1$). In particular, $X$ is isomorphic to a closed unit $l$-polydisc if and only if it is of the form $\calM(l\{t_1\. t_n\})$.

\begin{theor}\label{smprop}
Let $l/k$ be a separable finite extension of valued fields of height $1$ with $S=\Spec(\kcirc)=\{\eta,s\}$ and $S'=\Spec(\lcirc)=\{\eta',s'\}$, and let $X=\Spec(A)$ be a geometrically reduced, normal, affine, $\eta$-nft $S$-scheme with a closed point $x\in X_s$.

(i) If the analytic fiber $X_x^\an$ is isomorphic to an open unit $\hatl$-polydisc then $(X,x)$ is smooth-equivalent to $(S',s')$.

(ii) If the analytic generic fiber $\gtX_\eta$ is isomorphic to a closed unit $\hatl$-polydisc then any point $z\in X_s$ is smooth-equivalent to $s'\in S'$.
\end{theor}

\begin{rem}\label{rem2}
The case when $l/k$ is unramified is not so interesting since $x$ is a smooth point in this case. As we remarked in the Introduction, our main case of interest is when $l/k$ is ramified and the valuation is not discrete. Note that in this case $S'$ and, hence, $X$ are not of finite type over $S$. On the other hand, normality at $x$ is crucial for the argument (which uses Theorem \ref{anfibth}), so we cannot work with finite type models.
\end{rem}

\begin{proof}
Let $Y$ denote the $\eta$-normalization of $X\times_SS'$ and let $\gtY$ be its formal completion. By Lemmas \ref{extcomlem} and \ref{normlem}, $\gtY$ is the maximal affine formal model of its generic fiber $\gtY_\eta$ and $\gtY_\eta\toisom\gtX_\eta\otimes_{\hatk}\hatl$.

To prove (i) we let $Y_x^\an$ denote the preimage of $X_x^\an$ in $\gtY_\eta$. Since $\hatl/\hatk$ is separable and $X_x^\an$ is an open unit $\hatl$-polydisc, $Y_x^\an$ contains a connected component which is projected isomorphically onto $X_x^\an$. By \cite[Satz 6.1]{Bo} the analytic fibers of the closed points of $\gtY$ are connected (similarly to Theorem \ref{Zarconth}, this is another manifestation of Zariski connectedness theorem) and hence the analytic fibers of the preimages of $x$ in $\gtY$ are precisely the connected components of $Y_x^\an$. In particular, there exists a point $y\in\gtY_{s'}$ sitting over $x$ and such that the natural projection $Y_y^\an\toisom X_x^\an$ is an isomorphism. Since the projection $Y\to X$ is strictly \'etale at $y$ by Theorem \ref{anfibth}, it remains to show that the projection $Y\to S'$ is smooth at
$y$. By \cite[Satz 6.3]{Bo} $y$ is a smooth point of the $\till$-variety $\gtY_{s'}=Y_{s'}$. Moreover, in the proof of loc.cit. it is shown that for any choice of $t_1\. t_n\in\calO_{Y,y}$ such that their images in $\calO_{Y_{s'},y}$ form a regular sequence of parameters, we have that $t_1\. t_n$ are coordinates of the unit $\hatl$-polydisc
$Y_y^\an$. In particular, it follows that for the natural morphism $f{\colon}Y\to Z=\Spec(\lcirc[t_1\. t_n])$ that takes $y$ to the origin $z\in Z$, the induced morphism $Y_y^\an\to Z_z^\an$ is an isomorphism. By Theorem \ref{anfibth}, $f$ is strictly \'etale at $y$ and hence the morphism $Y\to S'$ is smooth at $y$.

The proof of (ii) is similar. Note that the connected components of $\gtY$ are in bijection with the connected components of its generic fiber $\gtY_\eta$. Indeed, $\gtY_\eta=\calM(\calB)$ and $\gtY=\Spf(\calBcirc)$, where $\calB=(\wh{A_l})_\pi$. In particular, any idempotent function on $\gtY_\eta$ is already defined on $\gtY$. Now, we choose a component $\gtY'$ of $\gtY$ such that the corresponding component $\gtY'_\eta$ of $\gtY_\eta$ is a closed unit $\hatl$-polydisc mapping isomorphically onto $\gtX_\eta$.
After replacing $Y$ with a suitable open subscheme such that $\gtY=\gtY'$, the points of $X_s$ are smooth-equivalent to the points of $Y_s$ by Theorem \ref{anfibth} and it remains to show that the projection $Y\to S'$ is smooth. To prove the latter, we pick up coordinates $t_1\. t_n$ on the polydisc $\gtY_\eta$, move them slightly until $t_i$ belong to the dense subalgebra $\calO_Y(Y)\subset\calO_{\gtY_\eta}(\gtY_\eta)\toisom\lcirc\{t_1\. t_n\}$, and apply Theorem \ref{anfibth} once again to show that the induced morphism $Y\to \Spec(\lcirc[t_1\. t_n])$ is strictly \'etale along $Y_s$.
\end{proof}

In the following lemma we prove that smooth-equivalence descends from $\eta$-normalized filtered projective limits.

\begin{lem}\label{projlem}
Keep the notation of Situation \ref{projsit}. Assume that $x\in X$ and $y\in Y$ are points and let $x_\alp$ and $y_\alp$ be their images in $X_\alp$ and $Y_\alp$, respectively. Then $(X,x)$ and $(Y,y)$ are smooth-equivalent over $S$ if and only if there exists $\alp_0\in A$ such that for each $\alp\ge\alp_0$ the germs $(X_\alp,x_\alp)$ and $(Y_\alp,y_\alp)$ are
smooth-equivalent over $S_\alp$.
\end{lem}
\begin{proof}
The inverse implication follows from Lemma \ref{smetlem}(iv), so let us prove the direct implication. Find $z\in Z$ and smooth morphisms $f{\colon}Z\to Y$ and $g{\colon}Z\to X$ as in Definition \ref{smdef}. By Proposition \ref{normlimprop}, $f$ and $g$ come from smooth morphisms $f_\alp{\colon}Z_\alp\to Y_\alp$ and $g_\alp{\colon}Z_\alp\to X_\alp$ for sufficiently large $\alp$, and it is obvious that $x_\alp$ and $y_\alp$ are the images of the projection $z_\alp\in Z_\alp$ of $z$. This concludes the proof.
\end{proof}

We finish the section with one more easy lemma.

\begin{lem}\label{easysmlem}
Let $X\to S$ and $Y\to S$ be dominant morphisms between integral schemes and let $x\in X,y\in Y$ be points which are smooth-equivalent over $S$. Assume that $k'/k(S)$ is a finite purely inseparable extension and set $X'=\Nr_{k'k(X)}(X)$ and $Y'=\Nr_{k'k(Y)}(Y)$. Then the preimages $x'\in X'$ and $y'\in Y'$ of $x$ and $y$ are smooth-equivalent over $S$.
\end{lem}
\begin{proof}
Note that the composite extensions $k'k(X)$ and $k'k(Y)$ are well defined since $k'/k(S)$ is purely inseparable. Choose smooth $S$-morphisms $f{\colon}Z\to X$ and $Z\to Y$ such that $x$ and $y$ are the images of a point $z$, and set $Z'=\Nr_{k'k(Z)}(Z)$. The morphisms $X'\to X$, $Y'\to Y$ and $Z'\to Z$ are bijective, hence we should only check that the induced morphisms $f'{\colon}Z'\to X'$ and $Z'\to Y'$ are smooth. But the latter was proved in Lemma
\ref{smetlem}(ii).
\end{proof}

\section{Relative one-dimensional inseparable local uniformization}\label{chaptwo}
Throughout \S\ref{chaptwo}, $k$ is a valued field of height one and $p=\cha(\tilk)$. We allow the case of $p=0$ for the sake of completeness. Most of our work is trivial in this case but one has to use the exponential characteristic $p=1$ in the formulas, e.g. $k^{1/p^\infty}=k$. The main result of \S\ref{chaptwo} is Theorem \ref{dim1unif} which establishes inseparable local uniformization of non-Abhyankar valuations on curves over valuation rings of height one. This result will be deduced by decompletion from Theorem \ref{4th}, which provides inseparable local uniformization of terminal points on analytic curves.

\subsection{Discs over deeply ramified analytic fields}\label{discsec}

Throughout \S\S\ref{discsec}-\ref{termsec}, $k$ is analytic. Consider the $k$-analytic space $\bfA=\bfA^1_k$ with a fixed coordinate $T$. If $k$ is algebraically closed then the structure of $\bfA$ is described in \cite[\S1.4.4]{berbook}. In particular, the points of $\bfA$ are divided into four classes as follows. Type $1$ points are the Zariski closed points; they are parameterized by the elements of $k$, and we say that they are of radius $0$. Given an element $a\in k$ and a number $r>0$, let $E(a,r)\subset\bfA$ denote the closed disc of radius $r$ with center at $a$. This disc has a unique maximal point which will be denoted by $p(a,r)$. Type $2$ and $3$ points are the points of the form $p(a,r)$ of rational (i.e. from $\sqrt{|k^\times|}=|k^\times|$) or irrational radius $r>0$, respectively. Any type $4$ point $x$ is obtained as the intersection of a decreasing sequence $E_i=E(a_i,r_i)$ of discs with no common Zariski closed points. The number $r=\lim_i r_i$ is called the radius of $x$; it is positive by completeness of $k$.

For a general analytic field $k$ the space $\bfA$ is homeomorphic to the quotient $\bfA^1_\whka/\Gal(k^s/k)$ (see also \cite[\S4.2]{berbook} or \cite[\S3.6]{berihes}).
Zariski closed points come from $k^a$; such a point $a\in\bfA$ is completely determined by the monic generator $f_a(T)$ of its annihilator $m_a\subset k[T]$. By a closed disc $E=E_k(a,r)$ of radius $r=r(E)>0$ and with center at a Zariski closed point $a$ we mean the image of $E_\whka(\alpha,r)$, where $\alpha$ is any root of $f_a(T)$. By {\em type} of a point $x\in \bfA$ we mean the type of any of its preimages in $\bfA^1_\whka$. The type $1$ points are parameterized by $\whka/\Gal(k^s/k)$; these are exactly the points $x\in \bfA$ such that $\calH(x)\subseteq\whka$. A point $x\in \bfA$ is of type $2$ (resp. $3$) if and only if $F_{\calH(x)/k}=1$ (resp. $E_{\calH(x)/k}=1$). This happens if and only if $x$ is the maximal point of a disc of rational (resp. irrational) radius. Finally, any type $4$ point coincides with the intersection of all discs containing it, and $x$ is of type $4$ if and only if $\calH(x)$ is transcendentally immediate over $k$ and not contained in $\whka$. We define the radius $r(x)$ of a point $x$ as the infimum of the radii of discs containing $x$. Points of type $1$ are exactly the points of zero radius. The following remark will not be used in the sequel, so we state it without proof.

\begin{rem}
Another definition of radius was given in \cite[3.6]{berihes}: for a Zariski closed point $a$ with monic generator $f(T)$ of $m_a$ and a disc $E=\bfA^1_k\{s^{-d}f(T)\}$, where $d=\deg(f)$, one defines $r_\inv(E)=s$. The latter quantity is an interesting invariant of $E$. For example, $r_\inv$ depends only on the algebra $\calA=\calO(E)$, the coordinate $T\in\calA$ and the degree $[(\calA/T\calA):k]$. Note for the sake of comparison that $r=r(E)$ depends also on the embedding $k\into\calA$. For example, $r$ is not preserved when one deforms $k$ in $\calA$ while $r_\inv$ is preserved. However, it surprisingly turns out that opposite to an incorrect remark in \cite{berihes}, $r_\inv$ of a type $1$ point can be positive. Moreover, a deformation of $k$ in $\calA$ can change the type of a point (only not Zariski closed points of types $1$ and $4$ can switch their type).
\end{rem}

\begin{defin}
Let $X\subset\bfA^1_k$ be a $k$-disc (open or closed). By $k$-degree of $X$ we mean the number $\min_{x\in X}[\calH(x):k]$, and $X$ is called {\em split} if its degree is $1$, i.e. $X$ has a $k$-point. We say that $X$ is {\em almost split} if it is an intersection of split discs (so an open almost split disc is always split).
\end{defin}

A disc is isomorphic to a unit $l$-disc if and only if it is {\em $l$-split} (i.e. is defined over $l$ and split) and is of {\em integral} radius $r$ (i.e. $r\in|l^\times|$). Note that almost split but not split discs exist if and only if there exists $\alp\in k^a$ such that $\inf_{a\in k}|a-\alp|$ is not achieved. In particular, such discs can exist only when $k$ is not stable (otherwise $k(\alp)$ is a cartesian $k$-vector space and the infimum is achieved, see \cite[Prop. 3.6.2/4]{BGR}). If $E_k(a,r)\subset\bfA$ is a disc then its preimage in $\bfA^1_\whka$ equals to $\cup_{i=1}^d E(\alpha_i,r)$, where $\alpha_1\. \alpha_d$ are the roots of the monic generator $f=f(T)$ of $m_a$ and $d=\deg(f)$. In particular, the preimage is a disjoint union of at most $d$ discs of radius $r$. It follows that a disc $E_k(a,r)$ is isomorphic to an $l$-split disc for $l=k(\alpha_1)$ if and only if $r_f:=\min_{1<i\le d}|\alpha_1-\alpha_i|>r$. As a consequence, we obtain the following version of Krasner's lemma.

\begin{lem}
\label{kraslem} Let $a,f,\alpha_1\.\alp_d$ and $r_f$ be as above.

(i) Suppose that $K$ is an analytic $k$-field and $x\in K$ is an element such that $|f(x)|<R_f$ for $R_f=r_f\prod_{i=2}^d|\alp_1-\alpha_i|$. Then the embedding $k\into K$ extends to an embedding $l\into K$. In particular, if $x\in\whka$ satisfies $|x-\alpha_1|<r_f$ then $\alpha_1\in\wh{k(x)}$.

(ii) A disc $E=E_k(a,r)$ is defined over a non-trivial extension $k'/k$ if and only if $r<\max_{1\le i\le d}|\alpha_1-\alpha_i|$.
\end{lem}
\begin{proof}
To prove (i) we consider the morphism $\calM(K)\to\bfA$ induced by $x$ and note that its image is a point contained in the disc $E(a,r)$ for some $r<r_f$. Indeed, one easily sees that $|f(p(a,r_f))|=R_f$, hence a point $y\in\bfA$ is in the open disc $D(a,r_f)$ with center at $a$ and of radius $r_f$ if and only if $|f(y)|<R_f$. This gives a homomorphism $l\into\calO(E(a,r))\to K$ and so $l\into K$.

In (ii) we observe that $\bfA^1_k$ is geometrically reduced, hence so is $E$. In particular, only separable extension $k'$ may be contained in $\calO(E)$. Furthermore, $E$ is defined over a non-trivial separable extension $k'$ if and only if $E\otimes_k m$ is not connected for a sufficiently large finite separable extension $m/k$. If $r\ge\max_{1\le i\le d}|\alpha_1-\alpha_i|$ then even the preimage of $E$ in $\bfA^1_\whka$ is connected and hence no such $m$ exists. Conversely, if $r<|\alpha_1-\alpha_i|$ for some $i$ then by density of $k^s$ in $k^a$ we can find $m/k$ as above with $\beta,\beta'\in m$ such that $|\beta-\alp_1|<r$ and $|\beta'-\alp_i|<r$. Then the preimage of $E$ in $\bfA^1_m$ is disconnected because $E_m(\beta,r)$ and $E_m(\beta',r)$ are two distinct connected
components.
\end{proof}

\begin{cor}\label{impcor}
Assume that $\cha(k)=p>0$. The Galois groups of $k$ and of its completed perfection $K=\wh{k^{1/p^\infty}}$ are canonically isomorphic. In particular, for any finite extension $L/K$ the field $l=L\cap k^s$ satisfies $[L:K]=[l:k]$ and $L=lK$.
\end{cor}
\begin{proof}
Note that $L/K$ is separable because $K$ is perfect. By Krasner's lemma, any finite extension $L/K$ is obtained by completing a finite extension of $k'=k^{1/p^\infty}$. In its turn, $k'$ is induced from a finite separable extension of $k$, hence $L=lK$ for a finite separable extension $l/k$. It follows that the natural homomorphism $\Gal_K\to\Gal_k$ is injective. To prove that this homomorphism is surjective we have to show that $K\cap
k^s=k$. Suppose on the contrary that $K\cap k^s$ contains an element $\alp\in k^s\setminus k$. Then $\alp$ can be approximated by elements of $k'$ to any precision, and Lemma \ref{kraslem} (i) would imply that $\alp\in k'$, that is absurd.
\end{proof}

Also, one can deduce from Lemma \ref{kraslem} that any disc $E=E(a,r)$ containing a Zariski closed point $x$ with tamely ramified extension $\calH(x)/k$ is an $l$-split disc for $l=\calO(E)\cap k^a$. In addition, $l$ embeds into $\calH(x)$ and hence is tamely ramified. Indeed, replacing $k$ with $l$ we can assume that $\calO(E)$ does not contain non-trivial extensions of $k$, and then $r\ge\max|\alp-\alp_i|$ by Lemma \ref{kraslem}(ii), where $\alp=T(x)$ and $\calH(x)=\calH(\alp)$. By tameness of $\calH(x)$ we have that $\max_i|\alp-\alp_i|=\inf_{c\in k}|c-\alp|$ and hence $E$ is $k$-split. For the sake of comparison, we now consider a typical example of a disc of degree $p$ whose center is wildly ramified.

\begin{exam}
\label{exwilram} Let $\alpha$ be such that $l=k(\alpha)$ is a wildly ramified Galois extension of $k$ of degree $p$. Set $R=\inf_{c\in k}|\alpha-c|$ and $r=|\alpha-\alpha_2|$, where $\alp_2\neq\alp$ is a conjugate of $\alp$. Usually $r<R$ and it is always the case in the discretely valued case. For any $s$ with $r\le s< R$, the disc $E_s=E(\alpha,s)$ is neither $k$-split nor $l$-split.
\end{exam}

The following class of valued fields will be very important in the sequel. Recall that a valued field $k$ is called {\em deeply ramified} if it is not discretely valued and $\kcirc=(\kcirc)^p+p\kcirc$ (that is, the Frobenius is surjective on $\kcirc/p\kcirc$). In particular, if $p>1$ then this condition simply means that $k$ is perfect, and if $\cha(\tilk)=0$ then this condition means that $k$ is not discretely valued. We refer to \cite[6.6.6]{GR} for many equivalent (and non-trivial) characterizations of this condition. Here we note only that any $a$ in a deeply ramified $k$ can be approximated by a $p$-th power up to $|pa|$. Indeed, since $k$ is not discrete we can find $c\in k$ such that $|p|<|c^pa|\le 1$. Hence $|c^pa-b^p|\le|p|$ for some $b\in k$ and we obtain that $|c^pa|$ is a $p$-th power in $|k^\times|$. Thus $|k^\times|$ is $p$-divisible and we could actually take $c$ with $|c^pa|=1$ achieving that $|a-(\frac bc)^p|\le|pa|$.

\begin{lem}\label{conjlem}
Assume that $k$ is deeply ramified, and let $l=k(\alp)$ be a wildly ramified Galois extension of degree $p$ with a conjugate $\alpha_2\neq\alp$ of $\alp$. Then $|\alpha-\alpha_2|=\inf_{c\in k}|\alpha-c|$.
\end{lem}
\begin{proof}
Let $\alp=\alp_1,\alp_2\. \alp_p\in k^s$ be the conjugates of $\alp$, $r=|\alpha-\alpha_2|$ and $s=\inf_{c\in k}|\alpha-c|$. Then $r=|\alpha-\alpha_i|$ for any $1<i\le p$ by Galois conjugation (because $G$ is cyclic of order $p$), and so $r\le s$ by Lemma \ref{kraslem}(ii). Now, let us assume that the assertion of the lemma fails and $r<s$. Replacing $\alpha$ with its translate $\alp-c$ for $c\in k$ preserves the value of $r$, and we can achieve in this way that $t:=|\alp|$ is as close to $s$ as we want. In particular, we may and will assume that $|p|^{1/p}t<s$. Let $a\in k$ be the norm of $\alp$, then $a=\prod^p_{i=1} (\alp-(\alp-\alp_i))$ hence expanding the right hand side expression,
taking $\alp^p$ to the left hand side and estimating the remaining terms we obtain that $|a-\alp^p|\le rt^{p-1}$, or, that is equivalent, $|a^{1/p}-\alp|<t(r/t)^{1/p}$. Since $(r/t)^{1/p}$ is smaller than the fixed number $(r/s)^{1/p}<1$ and $t$ can be made very close to $s$, we can achieve that $t(r/t)^{1/p}<s$, and, in particular, $|\alp-a^{1/p}|<s$. To prove the lemma by a contradiction it remains to recall that $a^{1/p}$ can be
approximated by elements of $k$ with good enough precision. Namely, there exists $b\in k$ such that $|b^p-a|\le |pa|$. But then $|b-a^{1/p}|\le|pa|^{1/p}=|p|^{1/p}t<s$ and hence $|b-\alp|<s$, which is absurd.
\end{proof}

\begin{prop}\label{conjprop}
Assume that $k$ is deeply ramified, and let $\alpha\in k^a$ be an element with conjugates $\alp=\alp_1,\alp_2\.\alp_d$. Then $\max_{1\le i\le d}|\alpha-\alpha_i|=\inf_{c\in k}|\alpha-c|$.
\end{prop}
\begin{proof}
Set $l=k(\alp)$. First, we assume that $l/k$ has no non-trivial subextensions and establish the following three cases: (i) $l/k$ is wildly ramified and Galois, (ii) $l/k$ is tamely ramified but not unramified, (iii) $l/k$ is unramified. By basic Galois theory of valued fields, $[l:k]=p$ in case (i), $l/k$ is of prime degree $r\neq p$ in case (ii), and $\till/\tilk$ is a separable extension without non-trivial subextensions in case (iii). Case (i) was established in Lemma \ref{conjlem}. In cases (ii) and (iii) $l/k$ is defectless, hence the infimum $\inf_{c\in k}|\alpha-c|$ is achieved for some $c$. Replacing $\alp$ with $\alp-c$ we do not change $\max_{1\le i\le d}|\alpha-\alpha_i|$ and achieve that $|\alp|=\inf|\alp-k|$. In case (ii) this implies that $|\alp|\notin|k^\times|$ and $|\alp^r|\in|k^\times|$, say $|\alp^r-a|<|a|$ for some $a\in k$. A simple computation then shows that the conjugates of $\alp$ satisfy the inequalities $|\alp_i-\xi_r^i\alp|<|\alp|$ with $\xi_r$ a primitive $r$-th root of unity, and hence $|\alp-\alp_i|=|\alp|$, as claimed. In case (iii), $|k(\alp)^\times|=|k^\times|$ hence replacing $\alp$ with $\alp/a$ for some $a\in k$ we can also achieve that $|\alp|=1$. Then $\tilalpha$ generates $\till$ over $\tilk$ because there are no intermediate extensions, and one easily sees that $\tilalpha_i$ are precisely the conjugates of $\tilalpha$, which are all distinct by separability of $\till/\tilk$. Hence $|\alp-\alp_i|=1$, as claimed.

Now, consider the general case. By the theory of valued fields there exists a tower of fields $k_n/k_{n-1}/\dots/k_0=k$ such that $l\subset k_n$ and all extensions $k_{i+1}/k_i$ are as in cases (i), (ii) or (iii). The proposition is already proved for $n=1$, and using induction on $n$ we can assume that the proposition is known for any extension which embeds in a similar tower of a smaller length. For $j\in\{0,1\}$ set $r_j=\inf_{c\in k_j}|\alpha-c|$. Note that $r_0\ge\max_{1\le i\le d}|\alpha-\alpha_i|$ by Krasner's lemma, and hence we should only establish the opposite inequality, that is, find $i$ with
$|\alp-\alp_i|=r_0$. Obviously $r_1\le r_0$ and let us first assume that the exact equality holds. The proposition is assumed to hold for $\alpha$ over $k_1$ by the induction assumption, hence $r_0=r_1=\max|\alpha-\alpha_{i_j}|$ where the maximum is taken over the conjugates of $\alp$ over $k_1$.

So, we can assume that $r_1<r_0$. Then there exists $\beta\in k_1$ such that $|\alpha-\beta|<r_0$, and it follows that $\inf_{c\in k}|\beta-c|=r_0$. Since the proposition is known to hold for $\beta$ by one of the three above cases, we obtain that $|\beta-\beta_2|=r_0$ for a conjugate $\beta_2$ of $\beta$. Then $|\alpha-\beta|<|\alpha-\beta_2|$ and by conjugation there exists a conjugate $\alpha_i$ such that $|\alpha-\beta|=|\alpha_i-\beta_2|$. It then follows that $|\alpha-\alp_i|=|\beta-\beta_2|=r_0$ as required.
\end{proof}

\begin{cor}\label{perfcor}
If $k$ is deeply ramified then any disc $X=E(\alp,r)$ is isomorphic to an almost $l$-split disc for a finite extension $l/k$.
\end{cor}
\begin{proof}
Replacing $k$ with $l=k^a\cap\calO(X)$ we can assume that it is algebraically closed in $\calO(X)$, and then we have to show that $X$ is almost $k$-split. By Krasner's lemma (see Lemma \ref{kraslem} (ii)), if $\alpha_i$'s are the conjugates of $\alpha$ then $r\ge s:=\max_i|\alpha-\alpha_i|$. But $s=\inf_{c\in k}|c-\alpha|$ by Proposition \ref{conjprop}, and hence the disc is almost split. (It is not split if and only if $r=s$ and the infimum is not achieved.)
\end{proof}

\begin{cor}\label{discrcor}
Let $x\in\bfA^1_k$ be a point of radius $r$. For $s>r$ we let $E(x,s)$ denote the unique closed disc of radius $s$ that contains $x$. If $k$ is deeply ramified then there exists a discrete subset $S$ of the interval $(r,\infty)$ such that for any $s\in(r,\infty)\setminus S$ the disc $E(x,s)$ is $l$-split for a finite extension $l/k$.
\end{cor}
\begin{proof}
Take $S$ to be the set of {\em critical radii} $s$ for which $E(x,s)$ is almost $m$-split but not $m$-split for a finite extension $m/k$. If $s_1>s_2$ are two critical radii then the corresponding fields are strictly embedded $m_1\subsetneq m_2$. Since $m_1\subset\calO(E(x,s_2))$ and $\calO(E(x,s_2))\cap k^s$ is finite over $k$, we obtain that each closed subinterval of $(r,\infty)$ contains finitely many elements of $S$. So, $S$ is discrete in $(r,\infty)$, as required. (Note that $S$ does not have to be discrete at $r$ because $\calH(x)\cap k^s$ can be infinite over $k$ for a point $x$ of type 1 or 4.)
\end{proof}

Assume that $m/k$ is a finite extension, $Y=\calM(m\{T\})$, $X=\calM(k\{T'\})$ and $f\: Y\to X$ is a morphism. Then $f$ is given by the image $f(T)\in m\{T\}$ of $T'$ and $f$ is generically \'etale (i.e. \'etale outside of a Zariski closed set) if and only if $f'(T)$ does not vanish identically. Note that the latter happens if and only if $f(T)$ is not of the form $h(T^p)$. Assume now that $f$ is non-constant. Then we can split it into a composition of a power of Frobenius $\Fr^n:Y\to Y$, $T\mapsto T^{p^n}$ and a generically \'etale morphism $g\:Y\to X$. We say that $y\in Y$ is a {\em critical} point of $f$ if $g$ is not \'etale at $\Fr^n(y)$. The geometrical meaning of critical points is as follows: the cardinality of non-empty geometric fibers is constant outside of a Zariski closed subset of $X$, where it drops. This set is the image of the set of critical points.

\begin{lem}\label{preimdisclem}
Let $m/k$ and $f\:Y\to X$ be as above and assume that $k$ is deeply ramified. Let $x\in X$ be a terminal point of radius $r=r(x)$ and $y\in f^{-1}(x)$. For any $r<s\le 1$ let $X_s=E(x,s)$ be the disc around $x$ of radius $s$ and let $Y_s$ denote the connected component of $f^{-1}(X_s)$ that contains $y$. Then each $Y_s$ is a disc. Moreover, if $y$ is not critical then there exists $r(f)>r$ and a discrete set $S(f)\subset(r,r(f))$ such that for any $s\in|k^\times|$ with $r<s<r(f)$ and $s\notin S(f)$, $Y_s$ is isomorphic to a unit $m(s)$-split disc.
\end{lem}
\begin{proof}
The claim that $Y_s$ is a disc is well known (the proof of this reduces to the easy claim that $Y\{s^{-1}f\}$ is a disjoint union of discs with centers at the roots of $f$). Let $t(s)$ denote the radius of $Y_s$; obviously, $t(s)$ is a monotonically increasing function and $t(r)$ is the radius of $y$. By Corollary \ref{discrcor}, the disc $Y_s$ is $m(s)$-split outside of a discrete set $S(f)$. So, it remains to prove

Claim 1. {\em If $s\in ((r,1)\cap |k^\times|)\setminus S(f)$ is close enough to $r$ then $t(s)\in|m(s)^\times|$.} Note that $|m(s)^\times|$ is $p$-divisible (because $k$ is deeply ramified) hence it suffices to show that for a sufficiently small $s$ we have that $t(s)=as^{1/p^n}$ for $a\in|m(s)^\times|$. For any $s\in(r,1)\setminus S(f)$ choose an isomorphism $\psi_s:\calO(Y_s)\toisom m(s)\{t(s)^{-1}T\}$ and let $f_s(T)$ be the image of the coordinate of $X$. Then Claim 1 reduces to

Claim 2. {\em For small enough $s$ the dominant non-constant term of $f_s(T)$ is of the form $c_sT^d$ with $d=p^n$, where $d$ and $|c_s|$ are fixed.} Here we use the standard lexicographical order on monomials: $aT^n>bT^m$ if either $|a|t(s)^n>|b|t(s)^m$ or they are equal and $n>m$. Note that the claim makes sense since $d$ and $|c_s|$ are invariants of $Y_s$ and $t(s)$ (i.e. they are independent of the choice the choice of $\psi_s$). Claim 2 can be checked over $\whka$. Indeed, after applying $\wtimes_k\whka$ any disc splits into a disjoint union of discs of the same radius, so instead of $x$, $y$, $X$ and $Y$ it suffices to prove the claim for compatible liftings of $x$ and $y$ to $X\wtimes_k\whka$ and $Y\wtimes_k\whka$ and connected components containing these liftings. So, we assume in the sequel that $k=k^a$. In particular, all discs are split.

If $y$ is of type $1$ then $r=0$ and $y$ is (now) Zariski closed. So we can assume that $y=0$ and we can choose $T$ to be the coordinate on all discs around $0$. Let $f=g\circ\Fr^n$ with generically \'etale $g$. Then $f(T)=a_0+a_dT^d+\dots$, where $d=ep^n$ and $e$ is the ramification degree of $g$ at $\Fr^n(y)$. For small enough $s$, $a_dT^d$ becomes the dominant non-constant term of the power expansion of $f=f_s$ on $E(0,t(s))$. In particular, if $y$ is not critical then $e=1$ and $d=p^n$, as claimed. Invariance of $|c_s|=|a_d|$ is obvious.

Assume now that $x$ is of type $4$. Fix a coordinate $T$ on $Y_1$ and let $f(T)=f_1(T)=\sum a_iT^i$. For each other $s$ choose a coordinate on $Y_s$ of the form $T-\alp_s$ for some $\alp_s\in k$. In particular, $Y_s=E(\alp_s,t(s))$ and $f_s(T)=f(T+\alp_s)$. Since $f$ is non-constant, $u:=\inf_{b\in k}|(f-b)(y)|>0$. We can safely remove from $f(T)$ all terms with $|a_i|<u$ achieving that $f$ is a polynomial of degree $N$. For a polynomial $h(T)=\sum_{i=0}^nh_iT^i$ let $\partial_lh=\sum_{i=l}^{n}\binom{i}{l}h_iT^{i-l}$ denote its $l$-th divided power derivative. Clearly, it is compatible with linear changes of variables, so $\partial_lf_s(T)=\partial_lf(T+\alp_s)$. Since $y$ is not Zariski closed, we can choose small enough $s$ so that for each $0\le l\le n$, if $\partial_lf$ does not vanish identically then it has no zeros on $Y_s$.

Let $c_sT^d$ be the dominant non-constant term of $f_s(T)$ on $E(0,t(s))$. We claim that $d=p^n$. Indeed, if $d=mp^n$ with $m>1$, $(p,m)=1$ then $|\binom{d}{p^n}|=1$ and we obtain that $\binom{d}{p^n}a_dT^{d-p^n}$ is the dominant term of $\partial_{p^n}f_s(T)$. Since $d-p^n>0$, this implies that $\partial_{p^n}f$ has a root in $\alp_s+E(0,t(s))=Y_s$, a contradiction. It remains to show that $d=p^n$ and $|c_s|$ does not change when we pass to a smaller disc $Y_{s'}$. This is a straightforward check that we only outline: one simply writes $f_{s'}(T)=f_s(T+\alp_{s'}-\alp_s)$ with $|\alp_s-\alp_{s'}|\le t(s)$, opens the brackets using binomial coefficients, and checks that the dominant term will be $c_{s'}T^d$ with $|c_s-c_{s'}|<|c_s|$.
\end{proof}

We say that an analytic $k$-field $K$ is {\em $k$-split} if $\inf_{c\in k}|T-c|=\inf_{c\in k^a}|T-c|$ for any $T\in K$ (the second infimum is computed in the analytic field $\wh{k^aK}$, which is unique up to a (non-unique) isometry).

\begin{cor}\label{splitcor}
Assume that $k$ is deeply ramified. Then $K$ is $k$-split if and only if $k$ is algebraically closed in $K$.
\end{cor}
\begin{proof}
Obviously, if $K$ is $k$-split then $k^a\cap K=k$. Conversely, assume that $K$ is not $k$-split, and let $T\in K$ and $\alp\in k^a$ be such that $|T-\alp|<\inf_{c\in k}|T-c|$. Note that $T$ induces a morphism from $\calM(K)$ to $E=E(\alp,|T-\alp|)$, hence it suffices to show that $E$ is defined over a non-trivial extension of $k$. Since $|T-\alp|<\inf_{c\in k}|\alp-c|$, Proposition \ref{conjprop} implies that $|T-\alp|<|\alp-\alp'|$ for a conjugate $\alp'$ of $\alp$, and applying Lemma \ref{kraslem} (ii) we obtain that $E$ is as we need.
\end{proof}

\begin{rem}
(i) The corollary implies the Ax-Sen theorem for a deeply ramified analytic field $k$. (Recall that the latter states that for any $K\into\whka$, $K\cap k^a$ is dense in $K$).

(ii) It can happen that $k$ is algebraically closed in $K$ but the latter is not {\em strictly split} in the following sense: there exists $T\in K$ such that $\inf_{c\in k}|T-c|$ is not achieved but $\inf_{c\in k^a}|T-c|$ is achieved (both are equal by Corollary \ref{splitcor}). For example, if $x$ is the maximal point of an almost split but not split disc then $\calH(x)$ is split but not strictly split over $k$.
\end{rem}

We will also need the following well known fact, which seems to be missing in the literature.

\begin{lem}\label{weilem}
Let $C=\calM(\calA)$ be a $k$-affinoid curve with a connected compact analytic domain $X\into C$ such that $X\otimes_kl=\coprod_{i=1}^n X_i$ is a disjoint union of $m_i$-split discs for finite extensions $l/k$ and $m_i/l$. Then $X$ is a Weierstrass affinoid domain in $C$.
\end{lem}
\begin{proof}
Enlarging $l$ (and $n$) if necessary we can achieve that all discs are $l$-split and of integral radius. One easily sees that either $X_i$ coincides with a connected component of $C_l=C\otimes_kl$ or $X_i$ is contained in a larger disc $X_i\subsetneq X'_i\subset C_l$.
Indeed: choose a coordinate $T$ on $X_i$ and move it slightly so that $T\in\calO_{C_l,x_i}\subset\calH(x_i)$, where $x_i$ is the maximal point of $X_i$. Then $T$ induces an isomorphism of a neighborhood of $X_i$ in $C_l$ onto an analytic domain in $\bfA^1_l$ containing the disc $T(X_i)$, and our claim follows. If $X_i$ is a connected component of $C_l$ then $X$ is a connected component of $C$ and we have that $A=B\times D$, where $X=\calM(B)$. In particular, $X$ is the Weierstrass domain given by the idempotent of $D$, e.g. $X=C\{|1_D|\le 1/2\}$.

Assume, now, that $X_i$ are contained in larger discs $X'_i$. Note that the preimage of the boundary $\partial(X)$ in $X_i$ is contained in $\partial(X_i)=\{x_i\}$ by \cite[2.5.8(iii)]{berbook}. Hence $\partial(X)=\{x\}$ and $\{x_1\. x_n\}$ is the preimage of $x$ in $C_l$. Now, choose any function $f\in\calA$ that has a zero in $X$ and does not vanish identically on it, and set $r=|f(x)|$. Then $f$ can be viewed as a function on each disc $X'_i$ such that $|f(x_i)|=r$ and $f$ has a zero in $X_i$ (by connectedness of $X$).
It follows that each $X_i$ is a connected component of $X'_i\{r^{-1}f\}$ and hence of $C_l\{r^{-1}f\}$. Therefore, $X$ is a connected component of $C\{r^{-1}f\}$. As we proved above, this implies that $X$ is a Weierstrass domain in  $C\{r^{-1}f\}$, and by transitivity of Weierstrass domains we obtain the assertion of the lemma.
\end{proof}

\subsection{Analytic inseparable uniformization of terminal points}\label{termsec}
An extension of analytic fields $K/k$ (we automatically assume that the valuations agree) is called {\em one-dimensional} if for some choice of $x\in K\setminus\whka$, $K$ is finite over the closure of $k(x)$ in $K$. The latter field will be denoted $\ol{k(x)}$ in the sequel; it is isomorphic to the completion $\wh{k(x)}$. It is proved in \cite[6.3.4]{temst} that such a $K$ is finite over any subfield $\ol{k(y)}$ with $y\in K\setminus\whka$. In \cite[\S6.2]{temst} one-dimensional fields are divided into types as follows: if $F=F_{K/k}$ and $E=E_{K/k}$ then the sum $E+F$ does not exceed one, and we say that $K$ is of {\em type 2} (resp. {\em 3}, resp. {\em 4}) if $F=1$ (resp. $E=1$, resp. $E=F=0$). In particular, $K$ is of type 4 if and only if it is {\em transcendentally immediate}. In addition, {\em type 1} fields will refer to subfields of $\whka$.

In the sequel we will work with a good strictly $k$-analytic curve $C$. Note that by Noether normalization, $C$ is a finite cover of a disc locally at any point $x\in C$. Though all our results hold without the goodness and strictness assumptions, we impose them for the reader's convenience; such generality covers our applications, but requires less familiarity with analytic geometry. We classify points on $C$ according to the types of their completed residue fields. One can easily see that this classification agrees on $\bfA_k^1$ with the classification from \S\ref{discsec}. Note that for a point $x\in C$ the following conditions are equivalent: $m_x\neq 0$, $x$ is Zariski closed, $x$ corresponds to a classical rigid point, $[\calH(x):k]<\infty$. In particular, if $x$ is Zariski closed then it is of type $1$, and the converse is true for an algebraically closed $k$. Also, $\calO_{C,x}=\kappa(x)$ if and only if $x$ is not Zariski closed.

Finally, we say that $x\in C$ is a {\em terminal point} if it is either of type $4$ or of type $1$. Thus, $x$ is terminal if and only if $\calH(x)/k$ is transcendentally immediate. It follows from \cite[2.5.2(d)]{berbook} that any terminal point $x$ is {\em inner}, i.e. $x\in\Int(C)$.

Recall that for any field $K$ of positive characteristic $p$, its {\em $p$-rank} is the number $n$ (possibly infinite) such that $p^n=[K:K^p]$.

\begin{lem}\label{cur2lem}
Assume that $k$ is perfect, and let $C$ be a good strictly $k$-analytic curve with a point $x$ that is not Zariski closed.

(i) Assume that $\cha(k)=p>0$. Then the $p$-rank of $\kappa(x)$ equals to one, and the $p$-rank of $\calH(x)$ equals to zero for $x$ of type $1$ and equals to one for $x$ of any other type.

(ii) If $x$ is of type $1$ and $\cha(k)>0$ then neither $\kappa(x)$ is algebraically closed in $\calH(x)$ nor $\calH(x)$ is separable over $\kappa(x)$. In other cases, both $\kappa(x)$ is algebraically closed in $\calH(x)$ and $\calH(x)/\kappa(x)$ is separable.

(iii) Assume that $x$ is inner and not of type $1$ (for example, any point of type $4$). Then $x$ possesses a neighborhood $C'$ embeddable into $\bfA^1_k$ if and only if $\calH(x)$ is topologically generated by one element, i.e. $\calH(x)=\ol{k(T)}$ for an appropriate choice of $T\in\calH(x)$.
\end{lem}
\begin{proof}
We deal with (i) and (ii) first. Note that (i) implies the first part of (ii), so we will deal only with its second part. Recall that $\kappa(x)$ is separably closed in $\calH(x)$ by \cite[2.3.3]{berihes}, hence we have to consider only the case when $\cha(k)>0$. Note that we can replace $C$ with an affinoid neighborhood of $x$. Moreover, we claim that if $C\to Y$ is a finite map taking $x$ to $y$ and (i) and (ii) are satisfied for $y$ then they are satisfied for $x$ too. For (i) this is clear because $\kappa(x)$ is finite over $\kappa(y)$, hence they have equal $p$-rank. Assume that (ii) holds for $y$. Note that $\calH(x)$ is the composite $\kappa(x)\calH(y)$ and for any finite extension $l/\kappa(y)$ we have that $l$ is algebraically closed in $l\calH(y)$ and $l\calH(y)$ is separable over $l$. So, (ii) holds for $x$.

Using Noether normalization theorem we can assume that $y$ is a point in a disc $Y=\calM(k\{s^{-1}T\})$ of radius $s>1$. Moreover, if $x$ is of type $2$ (resp. $3$) then we can achieve that $y$ is the maximal point of a disc $E(0,r)$ with $r=1$ (resp. $r\notin\sqrt{|k^\times|}$). For any connected rational affinoid domain $\calM(\calA)$ in $Y$, the ring $\calA$ is an integral domain and the subring $k(T)\cap\calA$ of $L:=\Frac(\calA)$ is dense in $\calA$, hence $L(T^{1/p})$ is the only inseparable $p$-extension of $L$ and the $p$-rank of $L$ is one. It follows that the $p$-rank of $\kappa(y)$ cannot exceed one. Since $T^{1/p}$ is not contained in $\kappa(y)=\calO_{E,y}$, the latter has $p$-rank one. The $p$-rank can only drop under completions, hence the $p$-rank of $\calH(y)$ cannot exceed $1$. In addition, since the $p$-rank does not exceed one, $\calH(y)$ is separable over $\kappa(y)$ if and only if $\kappa(y)$ is algebraically closed in $\calH(y)$.

Now we will use arguments that separate types. If $y$ is of type $1$ then $\calH(y)\subset\whka$ hence $l=k^a\cap\calH(y)$ is dense in $\calH(y)$ by the Ax-Sen theorem. Since $l$ is perfect, $\calH(y)$ has zero $p$-rank. In particular, $\calH(y)$ is not separable over $\kappa(y)$ and we obtain (i) and (ii) for type $1$. Assume that $y$ is not of type $1$. Since $T^{1/p}\notin\kappa(y)$ and any inseparable extension of $\kappa(y)$ contains $T^{1/p}$ (because the $p$-rank is one), it suffices to show that $T^{1/p}\notin\calH(y)$. Note that it suffices to check that $T\notin\calH(y')^p$, where $y'$ is any preimage of $y$ in $Y\wtimes_k\whka$, hence we can assume that $k=k^a$. Then the type $4$ case is proved in \cite[6.2.8]{temst}. The case of type $2$ (resp. $3$) follows from the observation that $\tilT$ (resp. $|T|=r$) is not a $p$-th power in $\wHy=\tilk(\tilT)$ (resp. $|\calH(y)^\times|=|k^\times|\oplus r^\bfZ$).

Now, let us prove (iii). If $x\in C'\subset\bfA^1_k$ then any coordinate on $\bfA^1_k$ topologically generates $\calH(x)$. Conversely, let us assume that $\calH(x)=\ol{k(T)}$. Since $\calH(x)$ is one-dimensional, $T\notin\whka$ and it follows from \cite[6.3.3]{temst} that $\calH(x)=\ol{k(T')}$ for any $T'$ with $|T-T'|<\inf|T-k^a|$. In particular, moving $T$ slightly we can assume that $T\in\kappa(x)$, and then $T$ induces a morphism $f{\colon}C'\to\bfA_k^1$ from a neighborhood of $x$.

Note that $x\in\Int(C')\subset\Int(C'/\bfA^1_k)$ by \cite[2.5.8(iii)]{berbook}. Since $f$ is not locally constant at $x$ by our assumption on $T$, the fiber over $y=f(x)$ is discrete. Hence $f$ is finite at $x$ by \cite[3.1.10]{berihes}. Since $f$ induces an isomorphism $\calH(y)\toisom\calH(x)$ and $\kappa$'s are algebraically closed in $\calH$'s, $f$ also induces an isomorphism $\kappa(y)\toisom\kappa(x)$. It follows that $f$ is a local isomorphism at $x$ (see the first step of the proof of \cite[3.4.1]{berihes}).
\end{proof}

\begin{cor}\label{cur2cor}
If $k$ is a perfect analytic field then for any one-dimensional field $K$ there exists a projective $k$-analytic curve $C$ with a point $x$ such that $\calH(x)\toisom K$.
\end{cor}
\begin{proof}
Choose $T\in K\setminus\whka$. If $K$ is not perfect then we can also choose $T$ so that $T\notin K^p$. Recall that $K$ is finite over the subfield $K_0=\ol{k(T)}$ and $K_0\toisom\calH(y)$, where $y$ is the point on $\bfA^1_k$ corresponding to the norm that $K_0$ induces on $k[T]$. It follows that $K/K_0$ is separable because the $p$-rank of $K_0$ is one by Lemma \ref{cur2lem} and hence any inseparable extension of $K_0$ contains $K_0(T^{1/p})$. Now, by \cite[3.4.1]{berihes} there exists an \'etale morphism $C\to\bfA^1_k$ and a point $x$ over $y$ such that $\calH(x)/\calH(y)$ is isomorphic to the extension $K/K_0$.
\end{proof}

It follows easily from the stable reduction theorem that if $k$ is algebraically closed and $C$ is a smooth $k$-analytic curve then any terminal point $x\in C$ has a neighborhood isomorphic to a disc, see \cite[4.3.1]{berbook}. (Note that other points have more complicated basic neighborhoods.) This statement is easy for type one points, but is a surprisingly deep fact for a type 4 point $x$. By Lemma \ref{cur2lem}, it is equivalent to a claim that any type $4$ field is of the form $\ol{k(T)}$, and the first direct proof of the latter result was given by Matignon (unpublished).

Another direct proof of this result was given by the author in \cite[6.3.1]{temst}. An important feature of that proof is that it works in the more general case when the ground field $k$ is deeply ramified. This enables us to describe in Theorems \ref{type4th} and \ref{terminalth} terminal points over any such $k$. We will consider in the proofs only the case when $p=\cha(\tilk)>0$ since it is substantially more difficult and it is the case we
will need for the applications. The author does not know about any other proof of such a description of terminal points; in particular, it cannot be deduced straightforwardly from the stable reduction theorem. Note also that even the description of type 1 points is not so obvious over a general perfect field because, as we will see, its proof makes use of the Ax-Sen theorem.

\begin{theor}\label{type4th}
Let $k$ be a deeply ramified field and let $K$ be a one-dimensional analytic $k$-field of type $4$. Then $K$ contains a $k$-finite subfield $l$ and an element $T$ such that $K=\ol{l(T)}$.
\end{theor}
\begin{proof}
First, assume that $K$ is $k$-split and $k$ coincides with the maximal tamely ramified extension $k^\tr$. Then $\tilK=\tilk$ is algebraically closed and applying \cite[6.3.1(i)]{temst} we obtain that the theorem holds with $l=k$. Our next aim is to remove the condition $k=k^\tr$.

Step 1. {\it The theorem holds when $K$ is $k$-split.} It follows from Corollary \ref{splitcor} that $K_m=\wh{k^\tr K}$ is split over $k_m=\wh{k^\tr}$, hence $K_m=\ol{k_m(T_m)}$ by the above case. By Corollary \ref{cur2cor} there exists a $k$-affinoid curve $C=\calM(\calA)$ with a point $x$ such that $\calH(x)\toisom K$.
Then the curve $C_m=C\wtimes k_m$ contains a point $x_m$ sitting over $x$ and such that $\calH(x_m)\toisom K_m$. Furthermore, $x_m$ is the only preimage of $x$ in $C_m$ because $k$ is algebraically closed in $\calH(x)$. It follows from Lemma \ref{cur2lem} that $x_m$ possesses a neighborhood $C'_m$ isomorphic to a disc of integral radius, and since $K_m$ is $k_m$-split, it must be a $k_m$-split disc. So, $C'_m\toisom\calM(k_m\{T'\})$.

It is a standard fact that the affinoid domain $C'_m$ can be defined already over a finite extension $l/k$ (see, for example, \cite[1.4]{BL0}), i.e. $C'_m$ is the preimage of an
affinoid domain $C'_l=\calM(\calA'_l)$ in $C_l:=C\otimes_k l$. Since $k_m\{T'\}=k_m\{T''\}$ for any $T''\in\ k_m\{T'\}$ with $|T'-T''|<1$, we can move $T'\in k_m\{T'\}\toisom\calA'_l\wtimes_l k_m$ and enlarge $l\subset k_m$ so that $T'\in\calA'_l$. Then a natural homomorphism $\phi{\colon}l\{T'\}\to\calA'_l$ arises, and it has to be an isomorphism because $\phi\wtimes_l k_m$ is the isomorphism $k_m\{T'\}\toisom\calO(C'_m)$. In particular, $C'_l$ is an $l$-split disc and $\calH(x_l)=\ol{l(T')}$, where $x_l\in C'_l$ is the preimage of $x$.

We thus descended from the infinite base change $C'_m$ to a finite base change $C'_l$, but it remains to descend further to $C$. The extension $k_m/k$ is Galois and hence we can replace $l$ with its Galois closure (which is contained in $k_m$). Note that $x_l$ is fixed by $G=\Gal_{l/k}$ because $K\cap k^a=k$. Since $l$ is deeply ramified and $\calH(x_l)\cap k^a=lK\cap k^a=l$, Corollary \ref{discrcor} implies that $x_l$ is the intersection of $l$-split open discs $C''_l\subset C'_l$ that contain it. Choosing $C''_l$ small enough we achieve that $gC''_l\subset C'_l$ for any $g\in G$. If two open discs in $C'_l$ are not disjoint then one of them is contained in another one, therefore $C''_l:=\cap_{g\in G}gC'_l$ coincides with some $gC''_l$. So, $gC''_l$, and hence $C''_l$, is $G$-invariant.
The image $C''\subset C$ of $C''_l$ is an open neighborhood of $x$ which is a tamely ramified form of an open disc, i.e. $C''\otimes_kl\toisom C''_l$ is an open $l$-split disc for a tamely ramified extension $l/k$. It was proved by A. Ducros that any such form $C''$ is itself isomorphic to an open $k$-split disc -- this is the assertion of \cite[Th. 3.6]{Duc}. Since $C''$ is a $k$-split disc and $x\in C''$ we obtain that $K=\ol{k(T)}$.

Step 2. {\it The theorem holds in general.} If $L\subset K$ is the completion of the field $K\cap k^a$ (which can be infinite over $k$) then $L$ is algebraically closed in $K$ by the Ax-Sen theorem. Note that $K$ is a one-dimensional $L$-field of type $4$ because it is finite over a subfield of the form $\ol{L(T')}$ with $T'\notin\whka=\wh{L^a}$. In addition, $K$ is $L$-split by Corollary \ref{splitcor}, so $K=\ol{L(T)}$ by the first stage. It remains to recall that the extension $K/\ol{k(T)}$ is finite by \cite[6.3.4]{temst}, hence $K$ coincides with $\ol{l(T)}$ already for a $k$-finite subfield $l\into L$.
\end{proof}

Now we are in a position to describe the local structure of terminal points over a deeply ramified field. Recall that a good analytic space $X$ is called {\em regular} if all its local rings are regular, and it is called {\em rig-smooth} if $X_K=X\wtimes_k K$ is regular for any analytic $k$-field $K$ (one can show that for perfect $k$ both notions coincide). Note that in \cite{berbook} rig-smooth spaces were called smooth, but now smoothness is used to denote rig-smooth spaces without boundary (thus rig-smoothness corresponds to smoothness in rigid geometry). An important difference between rig-smoothness and smoothness is that the former is inherited by analytic subdomains. We remark that in the following theorem the rig-smoothness assumption is needed only to include the (rather obvious) case of Zariski closed points because a reduced analytic curve over a perfect field is automatically rig-smooth at all other points.

\begin{theor}\label{terminalth}
Let $k$ be a deeply ramified analytic field and let $C$ be a rig-smooth $k$-analytic curve with a terminal point $x$. Then $x$ has a neighborhood $C'$ which is isomorphic to a closed unit $l$-disc for a finite extension $l/k$.
\end{theor}
\begin{proof}
First we note that it suffices to find a neighborhood of $x$ which is isomorphic to a domain in $\bfA^1_k$. Then Corollary \ref{discrcor} would imply that $x$ lies in an $l$-split disc, and, moreover, the radius can be chosen integral because $|k^\times|$ is not discrete (it is even $p$-divisible, since $k$ is deeply ramified). Normalizing the coordinate we can achieve that the disc is a unit $l$-disc. We will need the following simple lemma.

\begin{lem}\label{anfinlem}
Let $k$ be an analytic field, let $C$ be a good $k$-analytic curve with a terminal point $x$ such that the local ring $\calO_{C,x}$ is an integral domain, and let $T\in\calO_{C,x}$ be an element which is not algebraic over $k$. Then $T$ induces a map $f{\colon}C'\to\bfA^1_k$ from a neighborhood of $x$ and $f$ is finite at $x$.
\end{lem}
Recall that by definition \cite[3.1.1]{berihes}, $f$ is finite at $x$ if it induces a finite morphism $U\to V$ where $U$ (resp. $V$) is a neighborhood of $x$ (resp. $f(x)$).
\begin{proof}
Shrinking $C$ we can assume that it is reduced and irreducible. Obviously, $T$ induces a morphism $f{\colon}C'\to\bfA^1_k$, and it was observed earlier that $x$ is inner with respect to $f$. Since $f$ is not locally constant at $x$ by our assumption on $T$, the fiber of $y=f(x)$ is discrete. Hence $f$ is finite at $x$ by \cite[3.1.10]{berihes}.
\end{proof}

We now prove the theorem by dealing separately with three cases. Set $K=\calH(x)$. The case of a Zariski closed $x$ (i.e. $K/k$ is finite) is the easiest one. Any regular parameter $T\in m_x$ induces a morphism $f{\colon}C'\to Y=\bfA^1_k$ on an appropriate neighborhood of $x$ such that $f$ takes $x$ to the origin $y$ and is finite at $x$. Then $\calO_{Y,y}\to\calO_{C,x}$ is a finite homomorphism of one-dimensional regular local rings, which takes the regular parameter $t\in\calO_{Y,y}$ to the regular parameter $T$ and induces a separable extension $\calH(x)/k$ of the residue fields because $k$ is perfect.
Hence $\calO_{C,x}$ is \'etale over $\calO_{Y,y}$, and by \cite[3.3.6]{berihes} $f$ is \'etale at $x$. By \cite[3.4.1]{berihes} locally at $x$ the morphism $f$ is determined by the field extension $\calH(x)/k$, hence $C$ and $Y\otimes_k K=\bfA^1_K$ are locally isomorphic at $x$ and at the origin, respectively. So, the theorem holds true with $l=K$.

Next, we assume that $x$ is of type $1$ and is not Zariski closed. In particular, $\calO_{C,x}=\kappa(x)\into\whka$. Choose any element $T\in\kappa(x)\setminus k^a$. If $\cha(k)=p$ then by Lemma \ref{cur2lem}(i) we can manage that $T$ is not a $p$-th power. Let $f{\colon}C'\to Y=\bfA^1_k$ be the morphism induced by $T$ on an appropriate neighborhood $C'$ of $x$, and set $y=f(x)$. Then $f$ is finite at $x$ by Lemma \ref{anfinlem}, and, moreover, it is \'etale at $x$ by \cite[3.3.6]{berihes}. Indeed, $m_y=0$ and the finite extension of the residue fields $\kappa(x)/\kappa(y)$ is separable because in the positive characteristic case $T\in\kappa(y)$ and $T^{1/p}\notin\kappa(x)$. We claim that there exists a finite extension $l/k$ such that $\kappa(x)=l\kappa(y)$. First, we note that since $K\cap k^a$ is dense in $K\subset\whka$ by the Ax-Sen theorem, there exists a finite extension $l/k$ such that $K=l\calH(y)$. Let us check that one also has that $\kappa(x)=l\kappa(y)$. Indeed, $\kappa$'s are separably closed in $\calH$'s and $l/k$ is separable, hence $l\subset\kappa(x)$ and then the separable extension $\kappa(x)/l\kappa(y)$ is trivial because $l\kappa(y)$ is separably closed in its completion $l\calH(y)=K=\calH(x)$. Now, the same argument with \cite[3.4.1]{berihes} as was used for Zariski closed points implies that $C$ at $x$ and $\bfA^1_l$ at a preimage of $y$ are locally isomorphic.

Finally, we assume that $x$ is of type $4$. This is the most difficult case but the main work has already been done in Theorem \ref{type4th}, which implies that $K=\ol{k'(T)}$ for a finite extension $k'/k$ and some $T\in K$. Since $k'$ is separable over $k$, it is contained already in $\kappa(x)$. It follows that a sufficiently small neighborhood of $x$ is defined over $k'$, and then a smaller neighborhood embeds into $\bfA^1_{k'}$ by Lemma \ref{cur2lem}(iii). So, $x$ admits a neighborhood isomorphic to a $k'$-disc. By Corollaries \ref{perfcor} and \ref{discrcor}, we can take this disc to be an $l$-split disc of an integral radius for a finite extension $l/k'$. The theorem is proved.
\end{proof}

Theorem \ref{terminalth} provides an inseparable local uniformization of terminal points. As was explained in the Introduction, we will have to use a stronger simultaneous uniformization result in order to run induction in the proof of our main result in \S\ref{insepunifsec}. Therefore, we have to strengthen Theorem \ref{terminalth} to the following technically looking statement. For simplicity, we exclude the mixed characteristic case now, but see Remark \ref{4rem} below.

\begin{theor} \label{4th}
Let $k$ be an equicharacteristic analytic field, $g{\colon}\oC\to C$ be a morphism of rig-smooth $k$-analytic curves and $x\in C$ be a terminal point with a finite fiber $g^{-1}(x)=\{x_1\. x_n\}$. Then there exists a finite purely inseparable extension $k'/k$, finite separable $k'$-fields $l_1\. l_n$ and an affinoid neighborhood $C'$ of $x$ such that each $C'_i\otimes_k k'$ is isomorphic to an $l_i$-split disc, where $C'_i$ is the connected component of $g^{-1}(C')$ that contains $x_i$. Moreover, if $x_i$ is not critical (e.g. $x_i$ is not Zariski closed) then one can achieve that $C'_i\otimes_k k'$ is isomorphic to a unit $l_i$-disc. Finally, if a $k$-field $k_p\subset k^{1/p^\infty}$ is dense in the completed perfection $\ol{k^{1/p^\infty}}$ then one can take $k'\subset k_p$.
\end{theor}
\begin{proof}
The particular case when $\oC=C$ and $k$ is perfect was established in Theorem \ref{terminalth}. We will drop these two assumptions in two stages.

Step 1. {\em The theorem holds when $k$ is perfect.} By Theorem \ref{terminalth}, $x$ lies in a unit $l$-disc $E=\calM(l\{T\})$ for a finite extension $l/k$. Similarly, for each $1\le i\le n$ we can find a neighborhood $C_i\subset g^{-1}(E)$ of $x_i$ which is isomorphic to a unit $l_i$-disc for a finite extension $l_i/k$. Let $X_r\subset E$ be the disc of radius $r$ containing $x$, where $r$ is taken between the radius $r(x)$ of $x$ and $1$. Consider the preimage of $X_r$ under the morphism $C_i\to E$, and let $X_{r,i}$ be its connected component containing $x_i$. Fix $i$. By Lemma \ref{preimdisclem}, for any $r\in|l^\times|$ sufficiently close to $r(x)$ and not contained in a discrete set $S_i$ we have that $X_{r,i}$ is isomorphic to an $l_i(r)$-split disc and for each non-critical $x_i$ it is even the unit $l_i(r)$-disc. It remains to note that $\{X_{r,i}\}_{r(x)<r\le 1}$ is a decreasing family of discs whose intersection is $x_i$. In particular, $X_{r,i}$ is strictly smaller than $C_i$ for sufficiently small $r$'s, and then $X_{r,i}$ is the connected component of $g^{-1}(X_r)$ that contains $x_i$. Thus, we can set $C'=X_r$, where $r\in|l^\times|$ is sufficiently close to $r(x)$ and is not contained in the discrete set $\cup_{i=1}^nS_i$.

Step 2. {\em The general case.} If $k_p$ is not specified in the theorem then we make the default choice $k_p=k^{1/p^\infty}$. In particular, $k'':=\ol{k_p}$ is the completed perfection of $k$ in any case. Set $Y=C\wtimes_k k''$, $\oY=\oC\wtimes_k k''$ and $h=g\wtimes_k k''$, and let $y\in Y$ and $y_i\in Y_i$ be the preimages of $x$ and $x_i$ under the homeomorphisms $Y\to C$ and $\oY\to\oC$, respectively. By the previous step, $y$ possesses a neighborhood $Y'$ such that each $y_i$ is contained in a connected component $\oY_i\subset h^{-1}(Y')$ which is isomorphic to an $l''_i$-split disc for a finite extension $l''_i/k''$. The image of $Y'$ in $C$ is easily seen to be an affinoid domain which we denote $C'$ (for example, the preimage of $C'$ in $C\wtimes_k\whka$ is an affinoid domain preserved by the action of $\Gal_{k^s/k}$, hence the argument from Step 1 in the proof of Theorem \ref{type4th} and \cite[6.3.3/3]{BGR} imply that $C'$ is affinoid). We claim that $C'$ is a neighborhood of $x$ as required. The connected component $C'_i\subset g^{-1}(C)$ containing $x_i$ is the image of $\oY_i$ in $\oC$, hence $\oY_i\toisom C'_i\wtimes_k k''$. Our assertion now follows from Lemma \ref{newlem} below.
\end{proof}

\begin{lem} \label{newlem}
Assume that $X=\calM(\calA)$ is a $k$-affinoid space and $k_p/k$ is a purely inseparable extension such that $k_p$ is dense in $k''=\ol{k^{1/p^\infty}}$ and $X''=X\wtimes_k k''$ is isomorphic to a (resp. unit) $l''$-split disc, where $l''/k''$ is a finite extension. Then there exists a field $m$ such that $k\subseteq m\subseteq k_p$, $[m:k]<\infty$ and for any field $k'$ with $m\subseteq k'\subseteq k_p$ the space $X'=X\otimes_k k'$ is isomorphic to a (resp. unit) $l'$-split disc, where $l'/k'$ is a finite separable extension.
\end{lem}
\begin{proof}
By Corollary \ref{impcor}, $l=l''\cap k^s$ is finite over $k$ and $lk''=l''$ is the completed perfection of $l$. In particular, the morphism $\phi{\colon}(X\otimes_k l)\wtimes_l l''\to X\otimes_k l$ is a homeomorphism. Since $X''$ is defined over $l''$, there exists a connected component $Z''\subset (X\otimes_k l)\wtimes_l l''\toisom X''\otimes_{k''}l''$ which is projected isomorphically onto $X''$. This component is mapped by $\phi$ onto a connected component $Z\into X\otimes_kl$, and we observe that the projection $p{\colon}Z\to X$ is an isomorphism because $p\wtimes_kk''$ is the isomorphism $Z''\toisom X''$. The existence of such $Z$ implies that $X$ is defined over $l$, in the sense that $l$ embeds into $\calA$.

Next, choose a coordinate $T$ on the $l''$-split disc $X''$, i.e. fix an isomorphism $X''\toisom\calM(l''\{r^{-1}T\})$. (In the case of the unit disc we take $r=1$.) Any other element $T'\in l''\{r^{-1}T\}$ with $|T-T'|<r$ is a coordinate on $X''$ too, and, obviously, $\calA\otimes_k k_p$ is dense in $\calA\wtimes_k k''\toisom l''\{r^{-1}T\}$. Hence we can move $T$ so that $T\in\calA\otimes_k k_p$, and then $T\in\calA\otimes_k k'$ already for a $k$-finite subfield $k'\subset k_p$. Set $l'=lk'$ and note that $\calA\otimes_k k'\into l''\{r^{-1}T\}$ contains $l'\{r^{-1}T\}$ as a subalgebra. Moreover, the embedding $\phi{\colon}l'\{r^{-1}T\}\into\calA\otimes_k k'$ is actually an isomorphism because its base change $\phi\wtimes_{k'}k''$ is the isomorphism $l''\{r^{-1}T\}\toisom\calA\wtimes_k k''$. So, $X\otimes_k k'\toisom\calM(l'\{r^{-1}T\})$ is an $l'$-split disc, and clearly we can take $m=k'$.
\end{proof}

\begin{rem}\label{4rem}
It seems that Theorem \ref{4th} holds for any base field $k$ with any field $k_p\subset k^a$ such that $\wh{k_p}$ is deeply ramified, and the proof is essentially the same. For example, if $k$ is embedded in the completed algebraic closure of a valued field $\bfQ_p(T_1\. T_n)$ then one can take $k_p$ equal to either $k(1^{1/p^\infty},T_1^{1/p^\infty}\. T_n^{1/p^\infty})$ or $k(p^{1/p^\infty},T_1^{1/p^\infty}\. T_n^{1/p^\infty})$.
\end{rem}

\subsection{Decompletion}\label{decomplsec}
Throughout this section $k$ is a valued field of height $1$ and positive characteristic $p$, and $S=\Spec(\kcirc)$ with generic point $\eta=\Spec(k)$. Let $K/k$ be a finitely generated extension of valued fields of transcendence degree one and let $C=\Spec(A)$ be an affine {\em normalized $S$-model} of $\Kcirc$ in the sense that $C$ is a normal nft $S$-scheme with generic point $\Spec(K)\to C$ and such that $\Kcirc$ is centered on $C$. We assume that $K$ is of height one and that the extension $K/k$ is transcendentally immediate. Note that $C$ is $\eta$-nfp over $S$ by Lemma \ref{fintypelem}. Finally, let $K_1/K\. K_n/K$ be finite extensions of valued fields.

\begin{theor}\label{dim1unif}
Keep the notation of \S\ref{decomplsec}. Then there exists an affine normalized $S$-model $C'$ which refines $C$ and finite extensions of valued fields $l/k$ and $m_i/l$ for $1\le i\le n$ such that $l/k$ is purely inseparable, $m_i/l$ are separable, and the following conditions hold. Let $L_i$ denote the field $lK_i$ with the valuation extending that of $K_i$ and let $z_i$ denote the center of $\Lcirc_i$ on $\Nr_{L_i}(C')$. Then $z_i$ is smooth-equivalent over $S$ to the closed point of $S_i=\Spec(\mcirc_i)$.
\end{theor}
\begin{proof}
Set $C_i=\Nr_{K_i}(C)$ and let $\gtC$ and $\gtC_i$ denote the formal $\pi$-adic completions of $C$ and $C_i$ (as usual, $\pi$ is a non-zero element of $\kcirccirc$). Also we denote by $\gtC_\eta$ and $\gtC_{i,\eta}$ the analytic generic fibers as defined in \S\ref{agfsec}. In order to use uniform and simultaneous notation for $C$ and all $C_i$'s it will be also convenient to set $C_0=C$ and $\oC=\coprod_{i=0}^nC_i$ and to define $\ogtC$ as the formal completion of $\oC$. We start the proof with three preliminary steps.

Step 1. {\it Reduction to the case when $\oC_\eta$ is $k$-smooth and $\ogtC_\eta$ is rig-smooth over $\hatk$.} The $k$-curve $\oC_\eta$ can be made smooth by finite purely inseparable extension of the base field and subsequent normalization; that is, there exists a finite and purely inseparable extension $F/k$ such that the curve $\Nr_{FK}(\oC_\eta)$ is $F$-smooth. We claim that it suffices to prove the theorem for $F$, $C_F=\Nr_{FK}(C)$ and $FK_i$'s instead of $k$, $C$ and $K_i$'s. Indeed, assume that $m_i/l/F$ and $C'_F=\Spec(A_F)$ satisfy the assertion of the theorem for the former triple (so, $C'_F$ is a model of $FK$ and $A_F\subset FK$). Then $m_i/l/k$ and $C'=\Spec(A')$, where $A'=A_F\cap K$, satisfy all assertions of the theorem. Indeed, the only non-obvious claims here are that $A'$ is the normalization of a finitely generated $\kcirc$-algebra and $\Nr_{FK}(A')=A_F$, but both follow from the fact that $A'\supset A_F^{p^n}$ for large enough $n$ because $FK/K$ is purely inseparable. So, we can extend the ground valued field $k$ to $F$, achieving that the generic fibers are $k$-smooth. We thereby achieve that $\oC_\eta$ is $k$-smooth, and we claim that $\ogtC_\eta$ is then rig-smooth. Indeed, the Stein space $X=(\oC_\eta\otimes_k\hatk)^\an$ is rig-smooth by a GAGA type result \cite[3.4.3]{berbook} (it then even follows that $X$ is smooth because it has no boundary), and $\ogtC_\eta$ is an affinoid domain in $X$ by Remark \ref{genfibaffrem}(i).

Step 2. {\it Use of Theorem \ref{4th} and algebraization of extensions of $\hatk$.} Let $x$ be the center of $\Kcirc$ on $C$ and let $\hatx\in\gtC_\eta\into\ogtC_\eta$ be the point that corresponds to $\Kcirc$ via the map $\psi_x{\colon}C^{\bir,1}_x\to C_x^\an$ described in Remark \ref{complrem}. Also, we associate to each $\Kcirc_i$ a point $\hatx_i\in\gtC_{i,\eta}\into\ogtC_\eta$ in a similar way. The field $\calH(\hatx)\toisom\hatK$ is transcendentally immediate over $\hatk$, in particular, $\hatx$ is a terminal point. By Theorem \ref{4th} there exists a connected $\hatk$-affinoid neighborhood $W$ of $\hatx$, a finite purely inseparable extension $\oll/\hatk$ and finite separable extensions $\om_i/\oll$ such that the following condition holds: the preimage
of $W$ in $\ogtC_\eta$ contains connected components $W_i\ni\hatx_i$ such that $W_{i,l}:=W_i\otimes_{\hatk}\oll$ is a closed unit $\om_i$-disc. Since $k^{1/p^\infty}$ is dense in $(\hatk)^{1/p^\infty}$, Theorem \ref{4th} also states that we can choose $\oll$ of the form $\hatl$ for a finite purely inseparable extension $l/k$. The algebraization of $\om_i$'s is possible by Krasner's lemma; that is, there exist finite separable extensions $m_i/l$ such that $\om_i=\hatm_i$. Finally, we set $m=m_0$ and note that $W_0\toisom W$ because $C$ is the zeroth connected component of $\oC$, and hence $W\otimes_{\hatk}\oll$ is a closed unit $\om$-disc.

Step 3. {\em The affinoid domain $W$ algebraizes to an affine normalized $S$-model $V=\Spec(B)$ of $\Kcirc$, in the sense that the $W\toisom\gtV_\eta$ and the embedding $W\into\gtC_\eta$ is the analytification of a refinement of models $V\to C$.} Since $C=\Spec(A)$, we have that $\gtC_\eta=\calM(\calA)$ for the $\hatk$-affinoid algebra $\calA=\hatA_\pi$. Since $W\otimes_\hatk\oll$ is an $m$-split disc, Lemma \ref{weilem} asserts that $W$ is a Weierstrass domain in $\gtC_\eta$, say $W=\gtC\{f_1\. f_n\}$ with $f_i\in\calA$. Choose $\pi\in\kcirc\setminus\{0\}$ such that $g_i=\pi f_i\in\calAcirc$. Clearly, $W=\gtC\{g_1/\pi\. g_n/\pi\}$, and the same equality holds if we modify $g_i$'s by adding to them elements from $\pi\calAcirc$. Since $\hatA=\calAcirc$ by Lemma \ref{normlem}, we can achieve that $g_i\in A$. Set, now, $D=A[g_1/\pi\. g_n/\pi]$, $B=\Nr_K(D)$ and $V=\Spec(B)$. Clearly, $A_\pi\toisom B_\pi$ and so $V_\eta\toisom C_\eta$. In addition, $\calB:=\hatB_\pi$ is isomorphic to $\hatD_\pi$ by Lemma \ref{normlem}, and it remains to note that $\hatD\toisom\calAcirc\{g_1/\pi\. g_n/\pi\}$, and so $\hatD_\pi\toisom\calA\{g_1/\pi\. g_n/\pi\}$ and $\gtV_\eta=\calM(\hatB_\pi)\toisom\calM(\hatD_\pi)\toisom W$.

Now, we are prepared to prove the theorem. We have already introduced $l$, so set $L_i=lK_i$ as in the formulation of the theorem, and consider the schemes $C'=V$, $C'_i=\Nr_{K_i}(C')$ and $C'_{i,l}=\Nr_{L_i}(C')$ with formal completions $\gtC'$, $\gtC'_i$ and $\gtC'_{i,l}$. Note that $\gtC'_{i,\eta}$ is the preimage of $\gtC'_\eta=\gtV_\eta\toisom W$ in $\gtC_{i,\eta}$ because $C'_i$ is the $\eta$-normalization of $C_i\times_CC'$ (we use here that $C'_\eta\toisom C_\eta$ by Step 3). In particular, $W_i$ is a connected component of $\gtC'_{i,\eta}$ by Step 2. Each field $K_i$ is separable over $k$ by $k$-smoothness of $\oC_\eta$, hence $K_i\otimes_k l\toisom L_i$. Taking into account that $\lcirc=\Nr_l(\kcirc)$ because $l/k$ is purely inseparable, we deduce that $C'_{i,l}=\Nr_{L_i}(C'\otimes_{\kcirc}\lcirc)$. Therefore, its analytic generic fiber is $\gtC'_{i,l,\eta}\toisom\gtC'_{i,\eta}\otimes_{\hatk}\hatl$ by Lemma \ref{extcomlem}, and we obtain that $W_{i,l}$ is a connected component of $\gtC'_{i,l,\eta}$. By Lemma \ref{normlem}, $\gtC'_{i,l}$ is the maximal affine formal model of its generic fiber $\gtC'_{i,l,\eta}$, hence $\gtC'_{i,l}$ contains a connected component $\gtW_{i,l}$ with the generic fiber $W_{i,l}$. Let $Z_i$ be the closed subset of $C'_{i,l}$ that corresponds to $\gtW_{i,l}$. By Theorem \ref{smprop}(ii) any point of $Z_i$ is smooth-equivalent over $S$ to the closed point of $\Spec(\mcirc_i)$. It remains to note that $\Lcirc_i$ is centered on $Z_i$ because the corresponding analytic point of $\gtC'_{i,l,\eta}$ is the preimage of $\hatx_i$ and is, therefore, contained in $W_{i,l}$. So, $C'$ and $m_i/l/k$ are as required.
\end{proof}

We will also need the following lemma which will help us to treat valuations of height larger than one. Consider the following situation: $X=\Spec(A)$ is an affine normal nft $S$-scheme and $x\in X_\eta$ is a closed point of the generic fiber. Assume that the finite $k$-field $m=k(x)$ is provided with a valuation extending that of $k$ and such that the closed immersion $i_x{\colon}\Spec(m)\to X_\eta$ extends to a morphism $i{\colon}S_m\to X$, where $S_m=\Spec(\mcirc)$.

\begin{lem}\label{valuniflem}
Keep the above notation and assume that $x$ is a simple $k$-smooth point. Then there exists an affine nft $S$-scheme $X'$ and a morphism $f{\colon}X'\to X$ such that $f_\eta$ is an isomorphism, the closed immersion $i_x{\colon}\Spec(m)\to X'_\eta$ extends to a lifting $i'{\colon}S_m\to X'$ of $i$, and the image of the closed point of $S_m$ under $i'$ is smooth-equivalent to the closed point of $S_m$.
\end{lem}
\begin{proof}
Consider the homomorphism $A\to\mcirc$ corresponding to $i$ and apply the same construction as was used in Remark \ref{complrem}, i.e. complete it and invert a non-zero $\pi\in\kcirccirc$. In this way, we obtain a character $\calA\to\hatm$ which gives rise to a smooth $\hatm$-point $\hatx\in\gtX_\eta=\calM(\calA)$ which is Zariski closed because $\hatm$ is finite over $\hatk$. Let $T=(T_1\. T_n)$ be a system of regular parameters of $\calO_{\gtX_\eta,\hatx}$., The morphism $U\to\bfA^n_\hatk$, which $T$ induces on a sufficiently small affinoid neighborhood $U$ of $\hatx$, is \'etale at $\hatx$ by \cite[3.3.6]{berihes}. Then \cite[3.4.1]{berihes} implies that locally at $\hatx$, $f$ is determined by the field extension $\hatm/\hatk$, and hence it is locally isomorphic to the projection $\bfA_\hatm^n\to\bfA^n_\hatk$. It follows that for sufficiently small $r\in|k^\times|$ the Weierstrass domain $U\{r^{-1}T\}$ is isomorphic to a unit $\hatm$-polydisc, i.e. is of the form $\calM(\hatm\{T_1\. T_n\})$.

We claim that for small $r$'s each $U\{r^{-1}T\}$ is a Weierstrass domain in $\gtX_\eta$. Indeed, since $\hatx$ is Zariski closed, it possesses a fundamental system of Weierstrass neighborhoods in $\gtX_\eta$; in particular, we can find such a neighborhood $W'\subset U$. Obviously, $W'$ contains some $W:=U\{r^{-1}T\}$, and then $W=W'\{r^{-1}T\}$ is a Weierstrass neighborhood of $\hatx$ in $W'$, and we obtain that $W$ is a Weierstrass neighborhood of $\hatx$ in $\gtX_\eta$ by the transitive property of Weierstrass domains.

Now, we can act exactly as in the end of the proof of Theorem \ref{dim1unif}. First we algebraize $W$. By the definition of Weierstrass domains, $W$ is of the form $\gtX_\eta\{f/\pi\}$ where $\pi\in\kcirccirc$ and $f=(f_1\. f_m)\subset\calA$. Multiplying $f$ and $\pi$ by a large power of $\pi$ we achieve that $f\subset\calAcirc=\hatA$. Furthermore, we can add to each $f_j$ any element whose spectral norm is less than $|\pi|$ and hence we can harmlessly assume that $f_j\in A$. Then, we claim that $X'=\Nr(\Spec(A[f/\pi]))$ is as required. Obviously, $X'_\eta\toisom X_\eta$. Since $\hatx\in W\toisom\gtX'_\eta$ one has that $|f_j(\hatx)|\le |\pi|$. Hence $|f_j(x)|\le|\pi|$ in $m$, and so $f_j(x)/\pi\in\mcirc$. Existence of $i$ means that the image of $A$ in $m=k(x)$ lies in $\mcirc$. We have just shown that the images of $f_j/\pi$ in $m$ lie in $\mcirc$, hence the image of $A[f/\pi]$ is contained in $\mcirc$, and we obtain that $i$ lifts to $i'{\colon}S_m\to X'$. Finally, $W$ is a unit $\hatm$-polydisc, hence any point of the closed fiber $X'_s$ is smooth-equivalent to the closed point of $S_m$ by Theorem \ref{smprop}(ii).
\end{proof}

\section{Inseparable local uniformization}\label{insepunifsec}
We prove Theorem \ref{insepunif} in \S\ref{insepunifsec}. Strictly speaking, we deduce the theorem from the (relatively easy) case of Abhyankar valuation, which will be proved in a much stronger form in \S\ref{abhsec}. Our formulation and proof of the latter result involve logarithmic geometry, so, for expository reasons, we prefer to postpone dealing with it until \S\ref{simulsec} (no circular reasoning occurs here). We will establish the height one case of the Theorem in \S\ref{honesec} and will conclude the proof by induction on height in \S\ref{indsec}.

\subsection{Height one case}\label{honesec}
We will prove Theorem \ref{insepunif} by induction on the transcendence degree. However, to make the induction work we have to prove a more general statement (see Remark \ref{indrem} below). We will uniformize valuations by log smooth points $x$ of pairs $(X,D)$ where $X$ is normal and the closed subset $D\subset X$ is a $\bfQ$-Cartier divisor (i.e. it underlies a Cartier divisor of $X$). Log smoothness of $x$ means that it is a log smooth point of the log scheme $(X,M(D))$ or a toroidal point of $(X,X\setminus D)$ (see \S\ref{logsec} for references and comments on these notions, in particular, see the definition of log smooth points of simplicial shape and Remark \ref{logsmoothrem}).

\begin{theor}\label{equivunif}
Assume that $K/k$ is a finitely generated extension of valued fields such that $k$ is trivially valued and the height of $K$ is most one. Assume also that $X$ is a normal affine $k$-model of $K$ and $K_1/K\. K_n/K$ are finite extensions of valued fields. Given finite purely inseparable extensions $l/k$ and $L/lK$ and an affine model $X'$ of $\Kcirc$ with a $\bfQ$-Cartier divisor $D'\subset X'$ containing the center of $\Kcirc$ consider the following objects: fields $L_i=LK_i$ with the unique extension of $\Kcirc_i$, their models $X_i=\Nr_{L_i}(X')$, the preimages $D_i\subset X_i$ of $D'$ and the centers $x_i\in X_i$ of $\Lcirc_i$. Then there exists a choice of $l/k$, $L/lk$, $X'$ and $D'$ such that $X'$ refines $X$, each $x_i$ is a log smooth point of simplicial shape of the $l$-pair $(X_i,D_i)$, and $x_1$ is even a simple $l$-smooth point (in particular, $D_1$ is a normal crossings divisor at $x_1$).
\end{theor}
\begin{rem}
The case of $n=1$ in Theorem \ref{equivunif} covers our needs, but we establish the general simultaneous log uniformization because the proof is essentially the same.
\end{rem}
\begin{proof}
Note that the case of valued fields of height zero reduces to the classical theorem on the existence of a separating transcendence basis, so we can assume that $K$ and all $K_i$'s are of height one.

Step 0. {\it A general setup.} Our proof runs by induction on the transcendence defect $D_{K/k}$ of $K$ over $k$. The induction base $D_{K/k}=0$ corresponds to the case of Abhyankar valuations, which will be established in \S\ref{simulsec}: it is a particular case of Theorem \ref{Abhth}. Thus, in the sequel we assume that $D=D_{K/k}>0$ and the theorem is proved for smaller $D$'s.

It suffices to prove the theorem for any affine model of $\Kcirc$ which is finer than $X$, so we will replace $X$ with a refinement a few times during the proof. Note also that if $F/K$ is a finite purely inseparable extension then $X'=\Nr_F(X)$ is an affine model of $\Fcirc$ and for any normal affine refinement $Y'=\Spec(B)$ of $X'$, the scheme $Y=\Spec(B\cap K)$ is an affine refinement of $X$ satisfying $\Nr_F(Y)\toisom Y'$. Indeed,
$B^{p^n}\subset B\cap K$ for a large $n$, hence $\Nr_F(B\cap K)=B$. In addition, $\Nr_F(C)=B$ for a finitely generated $k$-subalgebra $C\subset B\cap K$, and so $B\cap K=\Nr_K(C)$ is finitely generated over $k$. The above observation implies that it suffices to prove the theorem for $F$, $X'$ and $FK_i$'s instead of the original $K$, $X$ and $K_i$'s, i.e. we can replace the field $K$ with a finite purely inseparable extension and update $X$ and $K_i$'s accordingly during the proof.

Step 1. {\it Fiber $X$ by curves and apply Theorem \ref{dim1unif}.} Since $D_{K/k}>0$, it follows from Remark \ref{Abhrem} that there exists a valued subfield $\ok\into K$ containing $k$ and such that $\trdeg_{\ok}(K)=1$ and $K/\ok$ is transcendentally immediate; in particular, $D_{\ok/k}=D-1$. Choose an affine $k$-model $Y$ of $\ok^\circ$ and refine $X$ so that the embedding $\ok\into K$ induces a morphism $X\to Y$. Set $S=\Spec(\okcirc)$ and $\oeta=\Spec(\ok)$, and consider $C=\Nr_K(X\times_Y S)$, which is an integral nft scheme over $S$ and with $K\toisom k(C)$. The morphism $\Spec(\Kcirc)\to X$ factors through $C$ because $\okcirc$ is centered on $Y$, and so Theorem \ref{dim1unif} applies to $C$, $K_i/K$ and $S$. Thus, we can find towers $\om_i/\oll/\ok$ of finite extensions of valued fields with separable $\om_i/\oll$ and purely inseparable $\oll/\ok$ and a refinement $f_C{\colon}C'\to C$ of affine normalized $S$-models of $\Kcirc$ such that the center $z_i$ of $\oll K_i$ on $C_i:=\Nr_{\oll K_i}(C')$ is smooth-equivalent to the closed point $s_i$ of $S_i:=\Spec(\omcirc_i)$. The situation is illustrated by the following commutative diagram, where $S_\oll=\Spec(\olcirc)$ and the dotted arrow symbolizes that the points are smooth-equivalent.
\begin{equation}\label{diag1}
\xymatrix{
&z_i\ar[r]\ar@{.>}[ld]|-{\rm sm}&  C_i \ar[d]\ar[r] & C' \ar[d]\ar^{f_C}[r] & C \ar[dl] \\
s_i\ar[r] & S_i \ar[r] & S_\oll \ar[r] & S & }
\end{equation}

Step 2. {\it Refine $X$ and $Y$ and extend $K$ so that the following conditions are satisfied  in diagram (\ref{diag1}): the $\oeta$-fiber of $X$ is geometrically normal, $f_C$ is an identity and $\oll=\ok$.} Since $C'=\Spec(A)$, where $A$ is the normalization of a subring $\okcirc[f_1\. f_n]\subset\Kcirc$, we can use $f_i$'s to define an affine refinement $X'\to X$ such that $C'=\Nr_K(X'\times_Y S)$. Refining $X$ in this way, we achieve that $C'\toisom C$. Next, we extend the field $K$ by replacing it with $L:=\oll K$. Then $X$ is replaced with $X_L:=\Nr_L(X)$ and we can just replace $Y$ and $C$ with $Y_\oll:=\Nr_\oll(Y)$ and $\Nr_L(X_L\times_{Y_l}S_\oll)\toisom\Nr_L(C)$. At this stage the above diagram simplifies as follows
\begin{equation}\label{diag2}
\xymatrix{
&z_i\ar[r]\ar@{.>}[ld]|-{\rm sm}&  C_i \ar[d]\ar[r] & C \ar[dl] \\
s_i\ar[r] & S_i \ar[r] & S & }
\end{equation}
where $C_i:=\Nr_{K_i}(C)$. Finally, we can achieve that $C_\oeta=X_\oeta$ is geometrically normal by an additional purely inseparable extension of $\ok$ (choose a finite purely inseparable extension $\oll/\ok$ such that $\Nr(X_\oeta\otimes_\ok\oll)$ is geometrically normal, replace $K$ with $\oll K$, etc.).

Note that $C=\Nr_K(X\times_Y S)\toisom\Nr_\oeta(X\times_Y S)$ because $X_\oeta$ is normal, and similarly $C_i\toisom\Nr_\oeta(X_i\times_Y S)$ where $X_i=\Nr_{K_i}(X)$. Set also $Y_i=\Nr_{\om_i}(Y)$. Actually, it will be equivalent in the sequel to perform either normalization or $\oeta$-normalization, and we prefer to switch to the language of $\oeta$-normalizations. Now, diagram (\ref{diag2}) is obtained by the $\oeta$-normalized base change with respect to the morphism $S\to Y$ from the following diagram, where $x_i$ and $y_i$ are the centers of $\Kcirc_i$ and $\om_i$, respectively, and no smooth-equivalence is established so far
\begin{equation}\label{diag3}
\xymatrix{
&x_i\ar[r]\ar@{.>}[ld]|-{\rm ?}&  X_i \ar[d]\ar[r] & X \ar[dl] \\
y_i\ar[r] & Y_i \ar[r] & Y & }
\end{equation}

Step 3. {\it Refine $Y$ and replace the other entries of diagram (\ref{diag3}) with the $\oeta$-normalized base changes so that $x_i$ and $y_i$ become smooth-equivalent.} In the sequel, it will be convenient to refine $Y$ as described below. Let $\{Y_\alp\}_{\alp\in A}$ be the projective family of all affine refinements of $Y$ (i.e. they are $k$-models of $\okcirc$). This family is filtered and $S\toisom\projlim_\alp Y_\alp$. Note that $X_\alp:=\Nr_\oeta(X\times_Y Y_\alp)$ is a normalized $k$-model of $\Kcirc$ which refines $X$ and satisfies $\Nr_\oeta(X_\alp\times_{Y_\alp}S)\toisom C$; in particular, such refining has no impact on diagram (\ref{diag2}). Thus, we can freely refine $Y$ by replacing $Y$, $X$, $Y_i$ and $X_i$ with $Y_\alp$, $X_\alp$, $Y_{i,\alp}:=\Nr_{\om_i}(Y_\alp)\toisom\Nr_\oeta(Y_i\times_Y Y_\alp)$ and $X_{i,\alp}:=\Nr_{K_i}(X_\alp)\toisom\Nr_\oeta(X_i\times_Y Y_\alp)$, respectively. Note that $C_i$ is the projective limit of $X_{i,\alp}$'s by Proposition \ref{normlimprop}(i) and similarly $\Nr_{\om_i}(S)$ is the projective limit of $Y_{i,\alp}$'s. Recall that $S_i$ is open in $\Nr_{\om_i}(S)$ (this is even true for any valuation ring of finite height). Finally, let $x_{i,\alp}\in X_{i,\alp}$ and $y_{i,\alp}\in Y_{i,\alp}$ be the centers of $K_i$ and $\om_i$, respectively. Obviously, they are the images of $z_i$ and $s_i$, respectively, hence by Lemma \ref{projlem} there exists $\alp$ such that the points $x_{i,\alp}$ and $y_{i,\alp}$ are smooth-equivalent over $Y_\alp$ for each $1\le i\le n$. Refining everything with respect to the morphism $Y_\alp\to Y$ we finish the Step.

\begin{rem}\label{indrem}
Now, each $K_i$ is centered on a point which does not have to be smooth yet, but is at least smooth-equivalent to the point $y_i$ living in a smaller dimension. Naturally, we have to invoke the induction hypothesis at this stage. We will smoothen $y_i$ by an additional refinement, but we have to refine $Y$ rather than $Y_i$. This explains why we could not prove Theorem \ref{insepunif} in its original form and had to strengthen its assertion at least to a descent version of inseparable local uniformization (the $n=1$ case of Theorem \ref{equivunif}).
\end{rem}

Step 4. {\it Smoothen the points $y_i$ by an additional refining of $Y$ and a purely inseparable extension of $\ok$.} We will only consider log smooth points of simplicial shape, so usually we will omit the words "of simplicial shape". Since $D_{\ok/k}=D-1$, the induction assumption applies to the scheme $Y$ and the extensions $\om_i/\ok$ of valued fields. So, there exists an affine refinement, which without loss of generality can be denoted $Y_\alp\to Y$, a $\bfQ$-Cartier divisor $E\subset Y_\alp$ and finite purely inseparable extensions of valued fields $l/k$ and $\oll/l\ok$  that satisfies the assertion of the theorem. Explicitly, consider the schemes $Y_{i,\alp}=\Nr_{\oll\om_i}(Y_\alp)$ with the preimages $E_{i,\alp}\into Y_{i,\alp}$ of $E$ and let $y_{i,\alp}\in Y_{i,\alp}$ be the centers of the valued field $\oll\om_i$. Then we can achieve that each $y_{i,\alp}$ is a log smooth point of the $\oll$-pair $(Y_{i,\alp},E_{i,\alp})$ and $y_{1,\alp}$ is even an $\oll$-smooth point of $Y_{1,\alp}$.

Refining $Y$ we can assume that $Y=Y_\alp$ because we have already seen that such operation preserves everything in the construction of diagram (\ref{diag3}) (smooth-equivalence is preserved because $\oeta$-normalized base changes preserve smoothness by Lemma \ref{smetlem}(iv)). Next, we extend $\ok$ as follows: replace $\ok$, $\om_i$, $K$, $K_i$ with $\oll$, $\oll\om_i$, $\oll K$, $\oll K_i$, respectively; replace $Y$, $Y_i$, $X$, $X_i$ with their normalizations in these fields, respectively, and update $x_i$ and $y_i$, accordingly. Also, let $E_i\subset Y_i$ and $D_i\subset X_i$ be the preimages of $E$. Then (the new) $y_i$'s are log smooth and $y_1$ is smooth over $l$ by the construction, and $x_i$ are still smooth-equivalent to $y_i$ by Lemma \ref{easysmlem} (we can take $Y$ for the base scheme $S$ in the lemma). In particular, $x_1$ is $l$-smooth, and, replacing $l$ with a purely inseparable extension, we can also arrange that $x_1$ is a simple $l$-smooth point. (Note that the "last $K_i$" is of the form $LK_i$ for a purely inseparable extension $L/K$ accumulated in the process of proof, and similarly for the "last $X$", which accumulated refinements of the original $X$ and extensions of $K$.)

It remains to show that each $l$-pair $(X_i,D_i)$ is log smooth at $x_i$. Fix a $Y$-scheme $Z_i$, a point $z_i\in Z_i$ and smooth $Y$-morphisms $Z_i\to Y_i$ and $Z_i\to X_i$ taking $z_i$ to $y_i$ and $x_i$, respectively. Let $T_i$ be the preimage of $E$ in $Z_i$. Then the morphisms $(Z_i,T_i)\to (X_i,D_i)$ and $(Z_i,T_i)\to (Y_i,E_i)$ satisfy the assumptions of Lemma \ref{descentlem} and we obtain that $(Z_i,T_i)$ is log smooth at $z_i$ and $(X_i,D_i)$ is log smooth at $x_i$. For expository reasons, Lemma \ref{descentlem} will be given in \S\ref{simulsec}.
\end{proof}

\subsection{Induction on height}\label{indsec}
In this section, we prove Theorem \ref{insepunif} for valued fields of any (automatically finite) height. We do not prove the descent or simultaneous versions, but the only obstacle is that we do not have an appropriate version of Lemma \ref{valuniflem}. (It seems plausible that after developing basic tools of log analytic geometry, it will be easy to extend Lemma \ref{valuniflem} in that direction.)

Our proof runs by induction on the height $h$ of $\Kcirc$. Since the case of $h\le 1$ was established earlier, we should establish the step of the induction. So, we assume that the statement of the theorem holds true for $K$'s of smaller height. Let $\Fcirc$ be the localization of $\Kcirc$ whose height is $h-1$, then by $F$ we denote the valued field $(K,\Fcirc)$ (so $K=F$ as abstract fields). The image of $\Kcirc$ in $\tilF$ is a valuation ring. We denote it by $\tilFcirc$, and provide $\tilF$ with the corresponding valuation. Note that the valued field $\tilF$ is of height $1$, and the valuation on $K$ is composed from the valuations on $F$ and $\tilF$ in the sense that the preimage of $\tilFcirc$ in $\Fcirc$ coincides with $\Kcirc$.

Step 0. {\it Extending $K$ and refining $X$.} Obviously, it suffices to prove Theorem \ref{insepunif} for any model $X'$ of $\Kcirc$. In particular, we will freely replace $X$ with finer models of $\Kcirc$ throughout the argument. More generally, we can safely replace $k$, $K$ and $X$ with $l$, $L$ and $X'$, where $l/k$ and $L/lK$ are finite and purely inseparable and $X'$ is a model of $\Lcirc$ that refines $\Nr_L(X)$. This is shown exactly as in Step 0 from the proof of Theorem \ref{equivunif}.

Step 1. {\it Reduction to the case when $X$ is normal and there exists a morphism $g{\colon}X\to Y$ with an integral affine $k$-variety $Y$ such that $\Fcirc$ is centered on a simple smooth closed point $x$ of the generic fiber $X_\eta$.} Choose a subset $b=\{b_1\. b_d\}\subset\Fcirc$ such that $d=\trdeg_k(\tilF)$ and $\tilb$ is a transcendence basis of $\tilF$ over $k$. It then follows that $\Fcirc$ contains a subfield $\ok=k(b)$, and hence $F$ induces a trivial valuation on $\ok$. Provide $\ok$ with the valuation induced from $K$ and choose $Y$ to be any affine $k$-model of $\okcirc$. Then it is easy to see that there exists a refinement $X'\to X$ of affine $k$-models of $\Kcirc$ such that the embedding $i{\colon}\ok\into K$ induces a morphism $f{\colon}X'\to Y$. Thus, refining $X$ we can assume that $i$ induces a morphism $X\to Y$.

Let $x$ be the center of $\Fcirc$. Since $k(x)\subset\tilF$ and $\tilF$ is algebraic over $\ok$, we have that $x$ is a closed point of $X_\eta$ (where $\eta$ is the generic point of $Y$). Note that any refinement $X'_\eta\to X_\eta$ of affine $\ok$-models of $\Fcirc$ can be extended to a refinement $X'\to X$ of affine $k$-models of $\Kcirc$ and the induction assumption applies to the $\ok$-variety $X_\eta$ and the valued field $F$. In particular, there exists finite purely inseparable extensions $\oll/\ok$ and $L/\oll K$ such that the valuation ring $\Nr_L(\Fcirc)$ (which is the only extension of $\Fcirc$ to $L$) is centered on a closed simple $\oll$-smooth point $x_L\in \Nr_L(X_\eta)$. The latter variety is the generic fiber of the projection $\Nr_L(X)\to\Nr_\oll(Y)$ and by Step 0, it suffices to prove Theorem \ref{insepunif} for $L$ and $\Nr_L(X)$ instead of $X$ and $K$. So, we simply replace $\ok$, $K$, $X$ and $Y$ with $\oll$, $L$, $\Nr_L(X)$ and $\Nr_\oll(Y)$, and the conditions of Step 1 are now satisfied.

So, far we copied Step 0 and the fibration part of Step 1 from the proof of Theorem \ref{equivunif}. The remaining argument is also similar to \S\ref{honesec}, though a reference to Lemma \ref{valuniflem} will be used instead of the reference to Theorem \ref{dim1unif}.

Step 2. {\it The theorem holds true if the condition of Step 1 is satisfied.} The field $\om:=k(x)$ embeds into $\tilF$ because $F$ is centered on $x$, hence the valuation on $\tilF$ induces a height one valuation on $\om$, which agrees on $\ok\subset\om$ with the valuation induced by the embedding $\ok\into K$. In the sequel, we regard $\om$ and $\ok$ as valued fields. Note that $\okcirc$ is centered on $Y$ and its center is the image of the center of $K$ on $X$. Set $S=\Spec(\okcirc)$, $\eta=\Spec(\ok)\into S$ and $S_\om=\Spec(\omcirc)$. Then $X_S=\Nr_K(X\times_Y S)$ is an integral nft scheme over $S$ and its $\eta$-fiber is isomorphic to $X_\eta$ (we use that $X_\eta$ is normal because $X$ is so). Furthermore, the morphism $\Spec(\Kcirc)\to X$ obviously factors through $X_S$, and we obtain, in particular, a morphism from the closed subscheme $\Spec(\tilFcirc)\into\Spec(\Kcirc)$ to $X_S$. The image of the generic point of $\Spec(\tilFcirc)$ coincides with the image of the closed point of $\Spec(\Fcirc)$. Hence this point is $x$ and the morphism $\Spec(\tilFcirc)\to X_S$ factors through $S_\om$. In particular, $S$ and the induced $S$-morphism $i{\colon}S_\om\to X_S$ satisfy the condition of Lemma \ref{valuniflem}. Applying the lemma we find an affine morphism $f_S{\colon}X'_S\to X_S$ such that $f_S$ induces an isomorphism of the $\eta$-fibers, $i$ lifts to a morphism $i'{\colon}S_\om\to X'_S$ and the image $z_S$ of the closed point of $S_\om$ under $i'$ is smooth-equivalent to the closed point of $S_\om$. Note that $z_S$ is the center of $\Kcirc$ on $X'_S$ because $\Kcirc$ is composed from $\Fcirc$ and $\tilFcirc$, $\Fcirc$ is centered on $x$ and $\tilFcirc$ cuts off $\omcirc$ from $\om$.

Now, the argument from Step 2 in \S\ref{honesec} shows that there exists an affine refinement $X'\to X$ which induces $f_S$ in the sense that $X'_S\toisom\Nr_K(X'\times_Y S)$. So, refining $X$ we can achieve that $X'_S\toisom X_S$ (thus eliminating $X'_S$ and $f_S$ from the picture). Following the argument from Step 3 in \S\ref{honesec}, we deduce from Lemma \ref{projlem} that refining $Y$ via $Y'\to Y$ and updating $X$ as $\Nr_\eta(Y'\times_Y X)$ we can achieve that $\Kcirc$ is centered on a point $z\in X$ which is smooth-equivalent to the center $y_\om$ of $\omcirc$ on $Y_\om:=\Nr_\om(Y)$.

By Theorem \ref{equivunif} applied to $Y$, $\okcirc$ and $\omcirc$ (instead of $X$, $\Kcirc$ and $\Kcirc_1$ in the formulation of Theorem \ref{equivunif}), we find finite purely inseparable extensions $l/k$ and $\oll/l\ok$ and a refinement $Y'\to Y$ such that the valued field $\oll\om$ (which is the valued extension of $\om$) is centered on an $l$-smooth point of $\Nr_{\oll\om}(Y')$. Set $X'=\Nr_\eta(Y'\times_Y X)$ and perform the last update of our data by replacing $k$, $\ok$, $\om$, $K$, $Y$, $Y_\om$ and $X$ with $l$, $\oll$, $\oll\om$, $\oll K$, $\Nr_\oll(Y')$, $\Nr_{\oll\om}(Y')$ and $\Nr_{\oll K}(X')$, respectively. After this update, $\om$ is centered on $l$-smooth point $y_\om\in Y_\om$ and it also follows from Lemma \ref{easysmlem} that the center of $\Kcirc$ on $X$ is smooth-equivalent to $y_\om$. So, the center of $\Kcirc$ on $X$ is $l$-smooth, and enlarging $l$ we can even make it a simple $l$-smooth point. This establishes induction on height in the proof of Theorem \ref{insepunif}. (Clearly, the "last $K$" is a purely inseparable extension of the original $K$ accumulated during the proof, and similarly for the "last $X$", which accumulated refinements of the original $X$ and extensions of $K$).

\section{Simultaneous local log uniformization of Abhyankar valuations}\label{simulsec}

To finish the proof of Theorems \ref{insepunif} and \ref{equivunif} we have yet to prove Theorem \ref{equivunif} for Abhyankar valuations. We have been postponing that proof until this section because it involves techniques, including logarithmic geometry, that are not used in the rest of the paper. Although the proof is rather elementary, it involves a relatively heavy terminology, that may make it difficult to follow. So, let us outline the main idea before going into details.

\subsection{An outline of the method}
In order to uniformize an Abhyankar $K$ we choose an Abhyankar basis $B=B_E\coprod B_F$ and set $K_B=k(B)$. If $K=K_B$ then $K$ can be uniformized by toric geometry (i.e. essentially combinatorially). Namely, we will see that in this case $\Kcirc$ is the filtered union of regular local rings $\calO_{B,\oM}=\cup_\oM k(B_F)[\oM]_m$, where $\oM$ runs through free monoids in the valuation monoid $\Lam_B\cap\Kcirc$, $\Lam_B$ is the lattice in $K^\times$ generated by $B_E$ and $m$ is the ideal of $k(B_F)[\oM]$ generated by $\oM\setminus\{1\}$. We will construct an affine toric model $\bfA_{B,\oM}$ such that $\calO_{B,\oM}$ is the local ring of the center $\eta_{B,\oM}\in\bfA_{B,\oM}$ of $\Kcirc$, and for a large enough $\oM$ a neighborhood of $\eta_{B,\oM}$ will turn out to be finer than any fixed model of $\Kcirc$. In particular, this is enough to uniformize $K$ when $K=K_B$.

In general, we will consider the "toroidal" models $X_{B,\oM}=\Nr_K(\bfA_{B,\oM})$ with centers $x_{B,\oM}$ of $\Kcirc$ whose local rings will be denoted $A_{B,\oM}$. In principle, since the extension $K/K_B$ is defectless by the stability theorem, the extension $\Kcirc/\Kcirc_B$ admits a nice "toroidal description", and by approximation the latter is also valid for an extension $A_{B,\oM}/\calO_{B,\oM}$ with large enough $\oM$. However, we will mainly consider the especially simple case when the extension is unramified. Over a perfect ground field, this can always be achieved by an appropriate choice of $B$.

So far, we outlined a method to reprove the results of \cite{KK1}. This is not enough, however, because we should establish in Theorem \ref{equivunif} a descent form of inseparable local uniformization. Thus, we should uniformize $K$ by refining a model of $\Lcirc$ and normalizing it in $K$, where $K/L$ is finite. Since $K/L$ may be ramified (and even wildly ramified) we have to study the situation deeper. The above argument shows that $K$ and $L$ can be uniformized by the choice of appropriate bases $B$ and $B'$, respectively. Then we will show that for large enough $M$ the toroidal model $X_{B,\oM}$ is essentially independent of the Abhyankar basis. In particular, $A_{B,\oM}=A_{B',\oM}$ and we see that the refining work could be done already on the model of $\Lcirc$. At this stage it costs no extra-work to establish simultaneous local log uniformization for finitely many extensions of valued fields $K_i/L$, so the latter is the assertion of our main Theorem \ref{Abhth} on Abhyankar valuations.

\subsection{Some facts from log geometry}\label{logsec}
All our work can be done in the framework of toroidal geometry, whose basics can be found in \cite{KKMS}. We find it more convenient, however, to work within the framework of log geometry. Although for normal varieties they are rather close, the latter is better suited for the work with general schemes (e.g., this language may be applied to study local uniformization of Abhynakar valuations in mixed characteristic). We refer to \cite{K} or \cite{Ka} for basics of logarithmic geometry. Actually, we will work only with log structures induced from toroidal embeddings. We remark also that some basic notation and results concerning monoids are collected in \S\ref{toricsec}.

Let $X$ be a normal scheme of finite type over a field $k$ and let $D\subset X$ be a closed subset with complement $j{\colon}U\into X$. Consider the (\'etale) log structure $M(D):=j_*\calO^\times_U\cap\calO_X\into\calO_X$ induced by $D$ (where all sheaves are in the \'etale topology). Note that $D$ is a $\bfQ$-Cartier divisor if and only if $U$ is the locus of triviality of $M(D)$, hence $M(D)$ determines $D$ in this case. Thus, it is essentially equivalent to work with the pair $(X,D)$ or to work with the log scheme $(X,M(D))$ whenever $D$ is a $\bfQ$-Cartier divisor, and we will not consider the log structure $M(D)$ otherwise.

It is well known (see \cite[3.7]{K}) that $j$ is a toroidal embedding (i.e. \'etale-locally on $X$ it is isomorphic to the embedding of the open toric orbit into a toric variety) if and only if $D$ is $\bfQ$-Cartier (cf. Example \ref{torchart} below) and the log scheme $(X,M(D))$ is log smooth over the scheme $\Spec(k)$ provided with the trivial log structure. To simplify notation we will say that a pair $(X,D)$ is {\em log smooth} at a point $x\in X$ if $D$ is $\bfQ$-Cartier and $(X,M(D))$ is log smooth locally at $x$.

\begin{exam}\label{torchart}
Recall that a toric monoid $P$ is a finitely generated integral saturated monoid without torsion, see \S\ref{toricsec}. We associate to such $P$ a {\em toric chart} $\bfA_P:=\Spec(k[P])$ which is a toric variety (in particular, it is normal): the torus $\Spec(k[P^\gp])$ acts on $\bfA_P$ and the embedding $k[P]\into k[P^\gp]$ corresponds to the open immersion $j{\colon}\Spec(k[P^\gp])\into\bfA_P$. The image of $j$ is the only open orbit of the action and its complement is a toric divisor $D_P$. Note that $D_P$ is $\bfQ$-Cartier, in the obvious way. Note also that $I:=P\setminus P^\times$ is the maximal ideal of $P$ and $k[I]$ is a prime ideal of $k[P]$ giving rise to a closed subset $V_p\subset\bfA_P$ contained in $D_P$. Actually, $V_P$ is the only closed orbit of the torus action and we will call it the {\em center} of the chart. The pair $(\bfA_P,D_P)$ is log smooth at any point of $D_P$ and for the corresponding log structure $M=M(D_P)$ the monoids $\oM_x$ for $x\in\bfA_P$ are quotients of $\oP$, and $\oP\toisom\oM_x$ if and only if $x\in V_P$.
\end{exam}

\begin{lem}\label{torchartlem}
Let $X$ be a normal scheme of finite type over a field $k$, let $D\subset X$ be a $\bfQ$-Cartier divisor, and let $x\in X$ be a point. Then the pair $(X,D)$ is log smooth at $x$ if and only if \'etale-locally it is isomorphic to \'etale localization of a pair $(\bfA_P,D_P)$ at a point $x_P\in V_P$ for a toric monoid $P$. Any such $P$ is unique up to an isomorphism and $\oP$ is unique up to unique isomorphism.
\end{lem}
\begin{proof}
Everything except uniqueness of $P$ follows from \cite[3.7]{K}. To prove uniqueness we note that $\oP=P/P^\times$ is naturally isomorphic to $\oM_x=M_x/M^\times_x$ for $M=M(D)$, so $\oP$ does not depend on the choice of the chart. Recall that $P\toisom\oP\oplus L$ for a lattice $L$ (see \S\ref{toricsec}). Since $\rk(P^\gp)=\rk(L)+\rk(\oP^\gp)$ equals to the dimension of the irreducible component of $x$, we obtain that $\rk(L)$ is determined by $x$ and so $P$ is unique up to a non-canonical isomorphism.
\end{proof}

In the sequel, when we consider a log smooth pair we automatically assume that the ambient scheme is normal. By {\em monoidal chart} of a log smooth pair $(X,D)$ at a point $x$ we mean an embedding $P\into\calO^\sh_{X,x}$ (where $\calO^\sh$ denotes the strict henselization of a local ring $\calO$) which induces an \'etale morphism $(U,D\times_X U)\to(\bfA_P,D_P)$, where $U$ is a sufficiently small \'etale neighborhood of $x$. The above lemma implies that such charts exist and $P$ is unique (up to a non-unique isomorphism).

\begin{lem}\label{descentlem}
Assume that $f:Y\to X$ is a smooth morphism between normal $k$-varieties, and $D\subset X$ is a $\bfQ$-Cartier divisor with $E=f^{-1}(D)$. Let $(X,M_X=M(D))$ and $(Y,M_Y=M(E))$ be the associated (\'etale) log schemes then

(i) The morphism $(Y,M_Y)\to(X,M_X)$ is strict.

(ii) $(Y,M_Y)$ is log smooth at a point $y\in Y$ if and only if $(X,M_X)$ is log smooth at $x=f(y)$.
\end{lem}
\begin{proof}
Fix geometric points $\oy\to y\to Y$ and $\ox=f(\oy)$. We should check in (i) that $\oM_{X,\ox}\toisom\oM_{Y,\oy}$. Injectivity is clear, so let us check that an element $a\in\oM_{Y,\oy}$ is in the image of $\oM_{X,\ox}$. Our claim is \'etale-local at $\ox$ and $\oy$. Replacing $Y$ with an \'etale neighborhood of $\oy$ we can achieve that $a$ is defined as an element of $\Gamma(\calO_Y)$ which is invertible on $Y\setminus E$. Furthermore, replacing $X$ with an \'etale neighborhood $X'$ of $\ox$ and replacing $Y$ with a neighborhood of a lift of $\oy$ to $Y\times_XX'$ we can achieve that the fiber $Y_x=Y\times_X\Spec(k(x))$ is geometrically connected and $f$ admits a section $s:X\to Y$. Set $b=f^*s^*(a)$. We claim that $a=ub$ for $u\in\Gamma(\calO^\times_Y)$, and hence $s^*(a)$ gives rise to an element of $\oM_{X,\ox}$ mapping to $a$. Thus, (i) will follow when we prove this claim.

First, let us check the claim for $X=\Spec(R)$, where $R$ is a DVR. Then $D=x$ is the closed point of $X$ and $E=Y_x$ is integral. If $a$ is not a unit then it vanishes along $Y_x$ and hence is divisible by a uniformizer $\pi\in R$. Clearly, $b$ is also divisible by $\pi$ and so it suffices to prove the claim for $a'=a/\pi$ and $b'=b/\pi$ instead of $a$ and $b$. We can proceed inductively until $a$ is a unit, then $b$ is also a unit and we are done.

Assume, now, that $X$ is an arbitrary normal scheme. For any generic point $x\in D$ the local ring $\calO_{X,x}$ is a DVR and applying the above particular case to the base changes $Y\times_X\Spec(\calO_{X,x})\to\Spec(\calO_{X,x})$ we obtain that $a=ub$ at any generic point of $E$. Let $V(a)$ and $V(b)$ be the closed subschemes defined by the vanishing of $a$ and $b$. We have proved that $V(a)$ and $V(b)$ coincide at all points of $Y$ of codimension one. It follows that the closed immersion $V(a,b)\into V(a)$ is generically an isomorphism. Since $X$ is normal, it is $S_2$ and hence $V(a)$ is $S_1$, i.e. $V(a)$ has no embedded components. Therefore, $V(a,b)\toisom V(a)$ and we obtain that $V(a)=V(b)$. Thus, $a=ub$, as claimed.

Let us prove (ii). Choose $\oy$ and $\ox$ to be liftings of $y$ and $x$ and let $P=\oM_{X,\ox}\toisom\oM_{Y,\oy}$. Note that $P$ is fine by the log smoothness assumption and it is saturated by the normality assumption. Since $P$ is saturated, the epimorphism $M_{X,\ox}\to P$ splits (e.g., see an argument in \S\ref{toricsec}). Hence we obtain a homomorphism $\phi:P\to M_{X,\ox}\into\calO^\sh_{X,x}$ and, using that $P$ is fine, we can replace $X$ with an \'etale neighborhood of $\ox$ such that $\phi$ factors through $\calO_X$. The latter induces a $k$-chart $c_X:(X,M_X)\to\Spec(k[P])$ and since $(Y,M_Y)\to(X,M_X)$ is strict by (ii), the composition $c_Y:(Y,M_Y)\to\Spec(k[P])$ is also a $k$-chart. By \cite[Th. 3.5]{K}, $(X,M_X)$ is log smooth at $x$ if and only if $c_X$ is smooth at $x$, and the same is true for $y$. This reduces the question to its analog for smoothness of the usual schemes, which is classical (see \cite[$\rm IV_4$, 17.7.7]{ega}).
\end{proof}

\begin{defin}
A morphism $f:Y\to X$ is called {\em Kummer at} a point $y\in Y$ if the induced homomorphism $\oM_{X,x}\to\oM_{Y,y}$ is Kummer (see \S\ref{toricsec}). In this case, the rank of $f$ at $y$ is defined as the index of $\oM^\gp_{X,x}$ in $\oM^\gp_{Y,y}$. A morphism is {\em Kummer} if it is Kummer at all points of the source. Log schemes $X$ and $Y$ are {\em log isogenous} at points $x\in X$ and $y\in Y$ if there exist morphisms $g:Z\to X$ and $h:Z\to Y$ that are Kummer at a point $z\in g^{-1}(x)\cap h^{-1}(y)$.
\end{defin}

\begin{rem}
Kummer morphisms are typically non-flat. For example, consider the standard orbifold quotient $\Spec(k[x,y])\to\Spec(k[x^2,xy,y^2])$ with the obvious toric log structures given by monic monomials.
\end{rem}

\begin{rem}\label{logsmoothrem}
(i) Note that $X$ is smooth at $x$ and $D$ is normal crossings at $x$ if and only if $(X,D)$ is log smooth and the monoid $\oM_x$ is free.

(ii) If $(Y,E,y)$ is only log isogenous to such $(X,D,x)$ then we can only say that the monoid $\oM_y$ is of simplicial shape. In such case, we say that the log structure of $(Y,E)$ is {\em of simplicial shape} at the point $y$. One easily sees that the converse is also true, and so a log smooth $(Y,E)$ is of simplicial shape at $y$ if and only if there exists a Kummer morphism $(Y',E',y')\to(Y,E,y)$ such that $Y'$ is smooth and $E'$ is normal crossings at $y'$.
\end{rem}

\subsection{Toric charts}\label{toricblyasec}
Let $k$ be a trivially valued ground field and consider a finitely generated Abhyankar extension $K/k$. Note that $\oLam:=|K^\times|$ is a "multiplicative" lattice and $\oLamcirc:=|\Kcirc\setminus\{0\}|$ is a valuation monoid in $\oLam$, as defined in \S\ref{valmonsec}. It is well known (see Theorem \ref{simth}) that $\oLamcirc$ is a filtered union of its free submonoids; in particular, those are cofinal in the family of toric submonoids of $\oLamcirc$. In the sequel, the words "for sufficiently large toric monoid $\oM\subseteq\oLamcirc$ ..." will often be used instead of a more pedantic formulation "there exists a toric monoid $\oM_0\subseteq\oLamcirc$ such that for any toric monoid $\oM$ with $\oM_0\subseteq\oM\subseteq\oLamcirc$ ...".

Choose an Abhyankar transcendence basis $B=B_E\coprod B_F$ of $K$. We will associate to $B$ various objects, and this section is devoted to studying toric geometry related to $k(B)\subseteq K$. Let $K_B$ denote the field $k(B)$ provided with the valuation induced from $K$. Note that the valued subfield $k(B_F)\subseteq K_B$ is trivially valued because the set $\tilB_F$ is algebraically independent over $\tilk$. The value group $\oLam_B:=|K_B^\times|$ is a sublattice of $\oLam$ generated by $|B_E|$, and we also define $P_B^\times$ to be the free "multiplicative" abelian group generated by $B_F$ and set $\Lam_B=P_B^\times\oplus\oLam_B$. There is an obvious injection $i_B{\colon}\Lam_B\into K^\times$, and if $B$ is fixed usually we will simply identify $\Lam_B$ with a subgroup of $K^\times$.

Next portion of notation will be associated with a toric monoid $\oM\subset\oLam$ such that $\oM^\gp=\oLam$. Note that $\oM_B:=\oM\cap\oLam_B$ is a toric monoid and the embedding $\oM_B\into\oM$ is Kummer. We set $M_B=P_B^\times\oplus\oM_B$ and define a toric chart $\bfA_{B,\oM}=\Spec(k[M_B])$ with the toric divisor $D_{B,\oM}$ and the center $V_{B,\oM}$ as in Example \ref{torchart}. In addition, let $\eta_{B,\oM}$ denote the generic point of $V_{B,\oM}$ and let $\calO_{B,\oM}$ be its local ring. Note that though the chart depends only on the monoid $M_B$ we prefer to keep track in the notation for the initial dependency on $B$ and $\oM$.

\begin{lem}\label{raylem}
The local ring $\calO_{B,\oM}$ equals to the localization of the ring $k(B_F)[\oM_B]$ along the ideal generated by $\oM_B$. The following conditions on $B$ and $\oM$ are equivalent:

(i) $\Kcirc_B$ is centered on $\bfA_{B,\oM}$;

(ii) $\Kcirc_B$ is centered on $\eta_{B,\oM}$;

(iii) $\oM\subseteq\oLamcirc$.
\end{lem}
\begin{proof}
By its definition, $\eta_{B,\oM}$ corresponds to the ideal $I:=\oM_Bk[M_B]$. Hence the claim of the lemma about $\calO_{B,\oM}$ is obvious, and $\Kcirc_B$ is centered on $\eta_{B,\oM}$ if and only if $\Kcirccirc_B\cap k[M_B]=I$. If (iii) is violated, say $m\in\oM\setminus\oLamcirc$, then some positive power $m^n$ is in $\oM_B\setminus\oLamcirc$, hence $m^n\in k[M_B]$ and $m^n\notin\Kcirc$. So, $K_B$ is not centered on $\bfA_{B,\oM}$, and we proved that (i) implies (iii). The implication (ii)$\Rightarrow$(i) is obvious, so it remains to show that (iii) implies (ii). Assume that $\oM\subseteq\oLamcirc$. Then any $m\neq 1$ from $\oM_B$ belongs to $\Kcirccirc$, and since the valuation on $k[P_B^\times]\subset k(B_F)$ is trivial we obtain that $\Kcirccirc_B\cap k[M_B]=I$.
\end{proof}

\begin{lem}\label{cuplem}
For a fixed Abhyankar basis $B$ the equality $\Kcirc_B=\cup_{\oM\subset\oLamcirc}\calO_{B,\oM}$ holds, where $\oM$ runs through all toric monoids in $\oLamcirc$.
\end{lem}
\begin{proof}
We know that each $\calO_{B,\oM}$ is contained in $\Kcirc_B$ by Lemma \ref{raylem}(iii). On the other hand, each element of $K_B$ can be represented as $a/b$ for $a=a_1m_1+\dots+a_km_k$ and $b=b_1 n_1+\dots+b_ln_l$, where $a_i,b_j\in k(B_F)$ and the $m_i$'s (resp. $n_j$'s) are different elements of $\oLam_B$. Note that the valuations of $a$ and $b$ are equal to the maximum value of the valuation on the corresponding monomials, for example, $|a|=\max_i|m_i|$ because $|a_i|=1$ and the real numbers $|m_i|$ are all different. Multiply $a$ and $b$ by an appropriate $m\in\oLam_B$ such that $|b|=1$. Then after renumbering the indexes we achieve that $n_1=1$ and $b_1\neq 0$. Assuming now that $a/b$ is an arbitrary element of $\Kcirc_B$, we obtain that $|a|\le|b|=1$ and hence all $m_i$'s and $n_j$'s lie in $\oLamcirc$. Choosing a toric monoid $\oM\subset\oLamcirc$ which contains all $m_i$'s and $n_j$'s, we obtain that $a\in k(B_F)[\oM_B]$ and $b\in b_1+\oM_Bk(B_F)[\oM_B]$. So, $a/b\in\calO_{B,\oM}$ by the first part of Lemma \ref{raylem}, and we obtain that $\Kcirc_B$ is contained in the union of all $\calO_{B,\oM}$'s.
\end{proof}

\subsection{Normalization and independence of the basis}\label{normindsec}
In order to use toric charts to study the geometry of $K$ we set $X_{B,\oM}:=\Nr_K(\bfA_{B,\oM})$ and let $E_{B,\oM}\subset X_{B,\oM}$ denote the preimage of $D_{B,\oM}$. For any toric monoid $\oM\subset\oLamcirc$, $\Kcirc$ is centered on $X_{B,\oM}$ by Lemma \ref{raylem}, so let $x_{B,\oM}\in X_{B,\oM}$ denote the center of $\Kcirc$ and let $A_{B,\oM}$ be the local ring of $x_{B,\oM}$. Note also that $\eta_{B,\oM}$ is the image of $x_{B,\oM}$ in $\bfA_{B,\oM}$ by Lemma \ref{raylem}. We will see that the pair $(X_{B,\oM},E_{B,\oM})$ can be made log smooth at $x_{B,\oM}$ by an appropriate choice of $B$ and $\oM$. Actually, we will see in Proposition \ref{Bprop} that the local structure of $X_{B,\oM}$ at $x_{B,\oM}$ is essentially independent of $B$ (for sufficiently large $\oM$'s in $\oLamcirc$), so log uniformization is obtained by fixing $B$ and then choosing a sufficiently large $\oM$. To simplify notation we will often suppress
$B$ from the notation when it is clear from the context what $B$ is, e.g. we will simply write $X_\oM=X_{B,\oM}$, $x_\oM=x_{B,\oM}$, etc., though the dependency on $B$ and $\oM$ will be assumed. Later on we will have to consider simultaneously another Abhyankar basis $B'=B'_E\coprod B'_F$ and then we will use the notation $A'_\oM$, $\calO'_\oM$, etc., to denote the objects depending on $\oM$ and $B'$.

A final piece of notation is based on \S\ref{ebfibsec}. Consider the natural map of Riemann-Zariski spaces $\psi{\colon}\bfP_K\to\bfP_{k(B)}$. Let $\tilgtx\in\bfP_{k(B)}$ and $\gtx\in\bfP_K$ be the points corresponding to $\Kcirc_B$ and $\Kcirc$, respectively, and let $\gtx=\gtx_1\.\gtx_n$ be the whole fiber $\psi^{-1}(\tilgtx)$. By $Y_\oM\subset\bfP_K$ and $\tilY_\oM\subset\bfP_{k(B)}$ we denote the birational fibers of $A_\oM$ and $\calO_\oM$, respectively. Note that $Z_\oM:=\bfP_K\{k(B_F)\}\{\{i_B(\oM_B)\}\}$ is the preimage of $\tilY_\oM$ in $\bfP_K$ and $Y_\oM$ is one of its connected components by Corollary \ref{birconlem}. The following lemma will be used to separate $\gtx$ from other points of the fiber.

\begin{lem}\label{gtxlem}
(i) The fiber $\psi^{-1}(\tilgtx)$ is discrete. In particular, there exist closed, constructible, and pairwise disjoint sets $\gtX_i\subset\bfP_K$ such that $\gtx_i\in\gtX_i$.

(ii) Fix $\gtX_i$'s as in (i) and set $\gtX=\coprod_{i=1}^n \gtX_i$. Then $Z_\oM\subset\gtX$ and $Y_\oM\subset\gtX_1$ for large enough $\oM$.
\end{lem}
\begin{proof}
Assume that $\gtx,\gtx'\in\gtX$ are two points corresponding to valuation rings $\calO,\calO'$. Then $\gtx$ is a specialization of $\gtx'$ if and only if $\calO\subseteq\calO'$. Since all overrings of $\calO$ are localizations, it is easy to see that they form a totally ordered set (with respect to inclusion). In particular, the set of generalizations of $\gtx$ is totally ordered with respect to generalization and we obtain the following corollary: the set $\{\gtx,\gtx'\}$ is discrete if and only if the closures of $\gtx$ and $\gtx'$ are disjoint.

Now, let us check (i). The valuation ring $\calO_i$ corresponding to $\gtx_i$ is an extension of $\Kcirc_B$ to $K$. Since $\calO_i\nsubseteq\calO_j$ for $i\neq j$, the fiber is discrete. Furthermore, the closures of $\gtx_i$ are pairwise disjoint and each closure $\ogtx_i$ is the intersection of all closed constructible subsets containing $\gtx_i$. It follows that if $\gtX_i$ are sufficiently small closed constructible subsets containing $\gtx_i$ then they are pairwise disjoint.

Finally, let us prove (ii). By Lemma \ref{cuplem}, $\cup_{\oM}\calO_\oM=\Kcirc_B$, hence $\cap_\oM\tilY_\oM=\{\tilgtx\}$. It now follows from compactness of the constructible topologies on $\bfP_K$ and $\bfP_{k(B)}$ that the constructible neighborhood $\gtX$ of the fiber $\{\gtx_1\.\gtx_n\}$ contains the preimage of $\tilY_\oM$ for sufficiently large $\oM$. Returning back to the Zariski topology, in which $Y_\oM$ is connected, we obtain that $Y_\oM\subset\gtX_1$ for such $\oM$.
\end{proof}

\begin{cor}\label{cupcor}
For a fixed Abhyankar basis $B$ the equality $\Kcirc=\cup_{\oM\subset\oLamcirc}A_\oM$ holds, where $\oM$ runs through all toric monoids in $\oLamcirc$.
\end{cor}
\begin{proof}
If $\gtX_i$ are as in Lemma \ref{gtxlem} then for sufficiently large $\oM$ we have that $Y_\oM\subset\gtX_1$. In particular, $Y_\oM$ does not contain $\gtx_i$ with $i>1$, and we obtain that the intersection of all $Y_\oM$'s is just $\gtx$. So, $\Kcirc$ is the only valuation ring centered on all points $x_\oM$, and therefore $\Kcirc=\cup_\oM A_\oM$. (We use here that the $A_\oM$'s are normal local rings, so their union is a normal local ring and hence coincides with the intersection of all valuation rings dominating it by \cite[Ch.6, \S1, Th. 3]{Bou}.)
\end{proof}

\begin{prop}\label{Bprop}
Let $B$ and $B'$ be two Abhyankar bases. Then for any sufficiently large toric monoid $\oM\subset\oLamcirc$ the local rings $A_\oM$ and $A'_\oM$ in $K$ coincide, and for any $m\in \oM\cap\oLam_B\cap\oLam_{B'}$ one has that $i_B(m)=ui_{B'}(m)$ for a unit $u\in A_\oM^\times$. In particular, the divisors on $\Spec(A_\oM)$ induced from $E_{\oM}$ and $E'_{\oM}$ coincide.
\end{prop}
\begin{proof}
Recall that by Lemma \ref{birfiblem}, the normal local rings $A_\oM$ and $A'_\oM$ coincide if and only if their birational fibers $Y_\oM$ and $Y'_\oM$ coincide. Let $\gtX_i$ be as in Lemma \ref{gtxlem} and let $\oL\subset\oLamcirc$ be such that $Z_\oL\subset\gtX$. Obviously, $Z_\oM\subset\gtX$ for any larger toric monoid $\oM\subset\oLamcirc$, and in the sequel choosing $\oM$ we automatically assume that $\oL\subseteq\oM$. Any connected component of $Z_\oM$ contains a point $\gtx_i$ and hence is contained in $\gtX_i$. In particular, we obtain that $Z_\oM\cap\gtX_1=Y_\oM$. Similarly one can find a closed constructible set $\gtX'_1$ such that $Z'_\oM\cap\gtX'_1=Y'_\oM$. Then $S_0:=\gtX_1\cap\gtX'_1$ is a constructible set containing $\gtx$ and such that $Z_\oM\cap S_0\subset Y_\oM$ and $Z'_\oM\cap S_0\subset Y'_\oM$. A simple compactness argument shows that for any constructible set $S$ containing $\gtx$ we have that $Y_\oM\subset S$ for any sufficiently large $\oM$ (use that $\cap_\oM\oY_\oM=\{\gtx\}$ by Corollary \ref{cupcor} and that each $Y_\oM$ is compact in the constructible topology). So, for any constructible set $S\subseteq S_0$ we obtain that $Y_\oM=Z_\oM\cap S=S\{k(B_F)\}\{\{i_{B}(\oM_{B})\}\}$ for any sufficiently large $\oM$. Arguing similarly for $Y'_\oM$ we obtain that it is enough to find a constructible set $S\subseteq S_0$ containing $\gtx$ and such that $S\{k(B_F)\}\{\{i_{B}(\oM_{B})\}\}=S\{k(B'_F)\}\{\{i_{B'}(\oM_{B'})\}\}$ for any
sufficiently large $\oM\subset\oLamcirc$.

We will see that one can deal separately with strict and non-strict inequalities defining $Z_\oM$. First, we are going to find a constructible set $S_1\subset S_0$ such that $\gtx\in S_1$ and $S_1\{\{i_{B}(\oM_{B})\}\}=S_1\{\{i_{B'}(\oM_{B'})\}\}$. The monoid $\oN=\oM_B\cap\oM_{B'}$ coincides with $\oM\cap\oLam_B\cap\oLam_{B'}$ hence it is isogenous to both $\oM_B$ and $\oM_{B'}$. Since $S_1\{\{f\}\}=S_1\{\{f^n\}\}$, it suffices to find $S_1$ with $S_1\{\{i_{B}(\oN)\}\}=S_1\{\{i_{B'}(\oN)\}\}$. So, we can just pick up any basis $a_1\. a_E$ of $\oLam_B\cap\oLam_{B'}$ and define $S_1$ in $S_0$ by the conditions $|i_B(a_j)|=|i_{B'}(a_j)|$ for $1\le j\le E$. Then $|i_B(n)|=|i_{B'}(n)|$ on $S_1$ for any $n\in\oLam_B\cap\oLam_{B'}$, in particular, $i_B(n)=ui_{B'}(n)$ where $|u|=1$ on $S_1$. So, if $Y_\oM\subset S_1$ then $i_B(n)=ui_{B'}(n)$ for $u\in A_\oM^\times$.

Set $E=k(B_F)$ and $E'=k(B'_F)$, then we have to find a constructible set $S_2\subset S_0$ such that $\gtx\in S_2$ and $S_2\{E\}=S_2\{E'\}$. As soon as we establish existence of such $S_2$ we are done, since the set $S=S_1\cap S_2$ is then as required. The proposition now follows from the following lemma.
\end{proof}

\begin{lem}
Assume that $E$ and $E'$ are $k$-subfields of $\Kcirc$ of transcendence degree $F_{K/k}$.
Then there exists a constructible set $S\subset\bfP_K$ such that $\gtx\in S$ and $S\{E\}=S\{E'\}$.
\end{lem}
\begin{proof}
The reduction $\Kcirc\to\tilK$ induces an isomorphism of $E\subset\Kcirc$ onto the field $\tilE\subset\tilK$ and similarly for $E'$. We first consider a particular case when $\tilK$ is algebraic over the field $\tilL:=\tilE\cap\tilE'$; then the argument is similar to the argument on existence of $S_1$ from the above proposition. Let $L$ and $L'$ be the preimages of $\tilL$ in $E$ and $E'$, respectively, and let $\phi{\colon}\tilL\toisom L$ and $\phi'{\colon}\tilL\toisom L'$ be the isomorphisms that invert the reduction. Since $E$ is algebraic over $L$, a valuation ring in $K$ contains $L$ if and only if it contains $E$, in particular, $\bfP_K\{E\}=\bfP_K\{L\}$, and similarly $\bfP_K\{E'\}=\bfP_K\{L'\}$. Thus, our task reduces to finding a constructible set $S$ with $S\{L\}=S\{L'\}$ and $\gtx\in S$. Let $\tilL=k(c_1\. c_l)$ and take $S$ to be the set defined by the conditions $|\phi(c_i)|=|\phi'(c_i)|=1$ and $|\phi(c_i)-\phi'(c_i)|<1$ for $1\le i\le l$. Note that $\gtx\in S$ because $\wt{\phi(c_i)}=c_i=\wt{\phi'(c_i)}$. Furthermore, for any monomial $c^n=\prod_{i=1}^l c_i^{n_i}$ we have that $|\phi(c^n)-\phi'(c^n)|<1$. For any non-zero polynomial $x=\sum_n a_nc^n\in k[c_1\. c_l]$ we thus have that $|\phi(x)-\phi'(x)|<1$ on $S$. As $|\phi(x)|=1$ at any point of $S\{L\}$ it follows that $|\phi'(x)|=1$ on $S\{L\}$. Therefore, for any $y\in\tilL^\times$ we have that $|\phi'(y)|=1$ on $S\{L\}$, and so $S\{L\}=S\{L'\}$.

Now, let us drop the assumption on $\tilE\cap\tilE'$. Let $\tilL$ and $\tilL'$ be the separable closures of $\tilE$ and $\tilE'$ in $\tilK$. If the isomorphisms $\phi{\colon}\tilE\toisom E$ and $\phi'{\colon}\tilE'\toisom E'$ inverting the reductions extend to isomorphisms $\tilL\toisom L$ and $\tilL'\toisom L'$ onto $k$-subfields of $\Kcirc$ then $S$ exists by the previous paragraph. Indeed, $S\{L\}=S\{E\}$ and $S\{L'\}=S\{E'\}$ because $L$ (resp. $L'$) is algebraic over $E$ (resp. $E'$), but the previous paragraph implies that $S\{L\}=S\{L'\}$ because $\tilK$ is purely inseparable over $\tilL$ and $\tilL'$, hence $\tilL\cap\tilL'$ contains $\tilK^{p^n}$ for large $n$ and so $\tilK$ is algebraic over $\tilL\cap\tilL'$. A lifting $\tilL\to\Kcirc$ which extends the embedding $\tilE\toisom E\into\Kcirc$ is always possible after a strictly \'etale extension of $\Kcirc$. To show this fix an extension $L/E$ of trivially valued fields which is isomorphic to $\tilL/\tilE$ and consider the composite extension of valued fields $F=LK$. Then $\Fcirc$ is strictly \'etale over $\Kcirc$ and obviously $L\into\Fcirc$. Enlarging $F$ again, we can assume in addition that there is an embedding $L'\into\Fcirc$ which lifts $\tilL'\into\tilK=\tilF$. Let $\gty\in\bfP_F$ be the point corresponding to $\Fcirc$. We know that there exists a constructible set $S_1\subset\bfP_F$ such that $\gty\in S_1$ and $S_1\{E\}=S_1\{E'\}$, so we have only to "push down" this equality to $\bfP_K$.

The \'etale morphism $\Spec(\Fcirc)\to\Spec(\Kcirc)$ is induced from an \'etale morphism $f{\colon}Z\to Y$ of schemes of finite type over $\bfZ$ by \cite[$\rm IV_4$, 17.7.8]{ega}. Let $z\in Z$ and $y=f(z)$ be the centers of $\Fcirc$ and $\Kcirc$, respectively. Then the morphism $z\to y$ is an isomorphism because so is its pullback to the closed point of $\Spec(\Kcirc)$. Thus, $f$ is strictly \'etale at $z$ and replacing $Y$ and $Z$ with open subschemes we can also achieve that $f$ is strictly \'etale along the Zariski closure of $z$ and induces an isomorphism $\of{\colon}\oz\to\oy$ of the Zariski closures. The birational fibers $Z_\oz^\bir\subseteq\bfP_F$ and $Y_\oy^\bir\subseteq\bfP_K$ over $\oz$ and $\oy$ are constructible because $Y$ and $Z$ are of finite type over $\bfZ$. In particular, we can now replace $S_1$ with the smaller constructible set $S_1\cap Z_\oz^\bir$. Since $\of$ is strictly \'etale along $\oz$, the map $\bfP_F\to\bfP_K$  induces a bijection of constructible sets $\of^\bir{\colon}Z_\oz^\bir\toisom Y_\oy^\bir$ (see Theorem \ref{etaleprop}). In the constructible topology $\of^\bir$ is a continuous bijection between compact spaces and hence a homeomorphism. In particular, $S:=\of(S_1)$ is a constructible subset of $\bfP_K$ and to finish the proof we have now to check that $S\{E\}=S\{E'\}$. But the latter is obvious because the preimages of $S\{E\}$ and $S\{E'\}$ under the bijection $\of^\bir$ are $S_1\{E\}$ and $S_1\{E'\}$, and the latter sets coincide.
\end{proof}

\subsection{Main results on Abhyankar valuations}\label{abhsec}
\begin{theor}\label{toroidalth}
Assume that $K$ is an Abhyankar valued field finitely generated over a trivially valued field $k$, $\tilK$ is separable over $k$ and $B$ is an Abhyankar transcendence basis of $K$ over $k$, and keep other notation of \S\S\ref{toricblyasec}--\ref{normindsec}. Then there exists a toric monoid $\oM_0\subset\oLamcirc$ such that for any toric monoid $\oM$ with
$\oM_0\subseteq\oM\subseteq\oLamcirc$ the following conditions are satisfied:

(i) The pair $(X_{B,\oM},E_{B,\oM})$ is log smooth at $x_{B,\oM}$ and the projection $$f_{B,\oM}{\colon}(X_{B,\oM},E_{B,\oM})\to(\bfA_{B,\oM},D_{B,\oM})$$ is Kummer at $x_{B,\oM}$.

(ii) Let $(T,N_T)$ denote the log scheme associated with $(X_{B,\oM},E_{B,\oM})$ and let $t=x_{B,\oM}$. Then $\oM\toisom\oN_{T,t}$.

(iii) If $\tilB_F$ is a separating transcendence basis of $\tilK$ and $|B_E|$ is a basis of $|K^\times|$ then $f_{B,\oM}$ is \'etale at $x_{B,\oM}$.
\end{theor}
\begin{proof}
We start with (iii). In this case, the extension $K/K_B$ is unramified because $K_B$ is stable by Remark \ref{Abhrem2}, $|K_B^\times|=|K^\times|$ and $\tilK$ is separable over $k(\tilB_F)=\tilK_B$. Since $\Kcirc_B$ is the union of the rings $\calO_\oM$ by Lemma \ref{cuplem} and $\Kcirc$ is \'etale over $\Kcirc_B$, \cite[$\rm IV_4$, 17.7.8]{ega} implies that the \'etale morphism $\Spec(\Kcirc)\to\Spec(\Kcirc_B)$ is induced from an \'etale morphism $Y\to Z:=\Spec(\calO_\oM)$ for sufficiently large $\oM$. Clearly, we can assume that $Y$ is irreducible, and then it is $Z$-isomorphic to an open subscheme of $\Nr_K(Z)$ (we use that $Z$ and, hence, $Y$ is normal). Therefore, the localization of $Y$ at the center of $\Kcirc$ is $Z$-isomorphic to $\Spec(A_\oM)$; in particular, the morphism $\Spec(A_\oM)\to\Spec(\calO_\oM)$ is local-\'etale. Since $A_\oM$ and $\calO_\oM$ are the local rings of $x_\oM$ and its image $\eta_\oM$, we obtain that $f_\oM$ is \'etale at $x_\oM$. Since $A_\oM$ is log smooth at $\eta_\oM$ and the stalk of the sharp monoidal structure at $\eta_\oM$ is $\oM$, the same is true for $X_\oM$ and $x_\oM$.

To prove (i) and (ii) we choose a basis $B'$ as in (iii): this is possible because $\tilK$ is separable over $k$ and hence admits a separating transcendence basis. Then by Proposition \ref{Bprop} the pairs $(X_{B,\oM},E_{B,\oM})$ and $(X_{B',\oM},E_{B',\oM})$ are locally isomorphic at the points $t=x_{B,\oM}$ and $t'=x_{B',\oM}$ for sufficiently large $\oM$'s. The second pair is log smooth at $t'$ by (iii), hence the first pair is log smooth at $t$. Finally, if $(T',N_{T'})$ denotes the log scheme associated with $(X_{B',\oM},E_{B',\oM})$ then we have that $$\oN_{T,t}\toisom\oN_{T',t'}\toisom\oM_{B'}\toisom\oM$$
Thus, $f_{B,\oM}$ induces the map $\oM_B\into\oM$ on the stalks of sharp monoids and hence
it is Kummer at $x_{B,\oM}$.
\end{proof}

We are now in a position to prove simultaneous log uniformization for Abhyankar valuations. Assume that $k$ is a trivially valued field, $K$ is a finitely generated Abhyankar valued $k$-field, $X$ is an affine $k$-model of $\Kcirc$, $x\in X$ is the center of $\Kcirc$, and $K_1/K\. K_n/K$ are finite extensions of valued fields. For an affine refinement $f{\colon}X'\to X$ let $x'\in X'$ denote the center of $\Kcirc$. Furthermore, given a finite purely inseparable extension $l/k$ we provide each field $L_i=lK_i$ with the valuation extending that of $K_i$, set $X_i=\Nr_{L_i}(X')$, and define $x_i\in X_i$ as the center of $\Lcirc_i$ on $X_i$. Finally, $E'$ will denote a $\bfQ$-Cartier divisor on $X'$, and then $E_i$ will be the preimages of $E'$ under the finite morphisms $f_i{\colon}X_i\to X'$.

\begin{theor}\label{Abhth}
Let $k$, $K$, $X$ and $K_1\. K_n$ be as above.

(i) There exists a finite purely inseparable extension $l/k$, an affine refinement $X'\to X$ and a $\bfQ$-Cartier divisor $E'\subset X'$ such that the pairs $(X_i,E_i)$ and $(X',E')$ are log smooth at $x_i$ and $x'$, respectively, and each projection $f_i$ is Kummer at $x_i$. In addition, one can achieve that $x_1$ is a simple $l$-smooth point and all $x_i$'s are of simplicial shape.

(ii) If each $\tilK_i$ is separable over $k$ then the claim of (i) holds true for $l=k$.
\end{theor}
It is well known that one cannot expect all $x_i$'s to be smooth even when $X$ is a surface, $n=2$ and $K=K_1$. A counterexample was given by Abhyankar in \cite{Abh}.
\begin{proof}
We will need the following result which follows from Theorem \ref{insth} proved below: there exists a finite purely inseparable extension $l/k$ such that all fields $\wt{lK_i}$ are separable over $l$ (note that in the situation of (ii) we can just take $l=k$). Fix $l$ as above. Then it suffices to prove the theorem for $l$, $L=lK$, $X_L=\Nr_L(X)$ and $L_i=lK_i$ instead of the original $k$, $K$, $X$ and $K_1\. K_n$ (similarly to Step 1 from Theorem \ref{dim1unif}, we use that any refinement $X'_L\to X_L$ of normal affine $l$-models of $\Lcirc$ is the $L$-normalization of a refinement $X'\to X$ of affine $k$-models of $\Kcirc$). So, it suffices to establish (ii), and we assume in the sequel that each $\tilK_i$ is separable over $k$ and $l=k$.

Find an Abhyankar transcendence basis $B$ of $K$. Obviously, $B$ is also an Abhyankar transcendence basis of each $K_i$. Thus, we can associate to $B$ and each $K_i$ a sufficiently large toric monoid $\oM_i\subset\oLam^\circ_i=|\Kcirc_i\setminus\{0\}|$ that satisfies the assertion of Theorem \ref{toroidalth}(i). In the same way, we can associate to $B$ and $K$ a toric monoid $\oM\subset\oLam^\circ=|\Kcirc\setminus\{0\}|$ that satisfies \ref{toroidalth}(i) and contains $\cup_{i=1}^n(\oM_i\cap\oLam^\circ)$. Then we enlarge each $\oM_i$ by replacing it with the saturation of $\oM$ in $\oLam^\circ_i$.

Let us now explicate the assertion of \ref{toroidalth}(i) in our situation. By $(X_i,E_i)$ we denote the pair corresponding to $B$, $\oM_i$ and $K_i$ in \ref{toroidalth}(i). Similarly, the pair corresponding to $B$, $\oM$ and $K$ will be denoted $(X',E')$, temporarily allowing $X'$ to be not related to $X$. Recall that the affine chart $\bfA_{B,\oM_i}$ depends only on $B$ and the monoid $\oM_i\cap\oLam^\circ_B$, where $\oLam^\circ_B=|k(B)^\circ\setminus\{0\}|$. By our construction, the latter monoid equals to $\oM_B=\oM\cap\oLam^\circ_B$, and hence is independent of $i$. In particular, all affine charts are equal to $\bfA_{B,\oM}=\Spec(k(B)[\oM_B])$, and we obtain that $X'=\Nr_K(\bfA_{B,\oM})$ and $X_i=\Nr_{K_i}(\bfA_{B,\oM})$. Also, $E'$ and $E_i$'s are the preimages of the toric divisor $D_{B,\oM}$ of the chart $\bfA_{B,\oM}$. Thus, the condition of \ref{toroidalth}(i) tells that each pair $(X_i,E_i)$ is log smooth at $x_i$ and the projection $X_i\to\bfA_{B,\oM}$ is Kummer at $x_i$. The same is true for $(X',E')$ and $x'$, hence each projection $f_i$ is Kummer at $x_i$.

By Theorem \ref{simth}, there exists a free monoid $\oM'_1\subset\oLam^\circ_1$ which contains $\oM_1$. Replacing $\oM$ with the larger toric monoid $\oM'_1\cap\oLamcirc$ and replacing all $\oM_i$'s with the saturations of the new $\oM$ we keep all above conditions and, in addition, achieve that each $\oM_i$ is of simplicial shape and $\oM_1$ is even a free toric monoid. In particular, $x_1$ is a smooth point and, since $k(x_1)\subset\tilK_1$ is separable over $k$, $x_1$ is even a simple $k$-smooth point.

It remains to achieve, in addition to all the above properties, that $X'$ admits a morphism to $X$ compatible with the generic points. By Corollary \ref{cupcor} taking a sufficiently large $\oM$ we can also achieve that the local ring $A_{B,\oM}$ of $x'$ contains any finite subset of $\Kcirc$. Since $X=\Spec(A)$, where $A=k[a_1\. a_n]\subset\Kcirc$, we rechoose $\oM$ so that in addition to all the above properties we have that $A\subset A_{B,\oM}$. Since $A_{B,\oM}=\calO_{X',x'}$ contains $A$, a neighborhood of $x'$ admits a morphism to $X$. So, just shrinking $X'$ and updating the $X_i$'s accordingly, we achieve that $X'$ is affine, admits a morphism to $X$ and satisfies all the other properties listed in the theorem.
\end{proof}

To complete the proof of Theorem \ref{Abhth} it remains to establish the following result.

\begin{theor}\label{insth}
If $K$ is a finitely generated Abhyankar valued field over $k$ then there exists a finite purely inseparable extension $l/k$ such that for any finite purely inseparable extension $l'/l$ the field $\wt{l'K}$ is separable over $l'$.
\end{theor}
\begin{proof}
Note that if $\tilK$ is separable over $k$ and $l'/k$ is purely inseparable of degree $d$ then $[l'K:K]\le d$ and $[\wt{l'K}:\tilK]\ge [l'\tilK:\tilK]=d$. Since $[l'K:K]\ge [\wt{l'K}:\tilK]$, all inequalities are equalities and $\wt{l'K}=l'\tilK$. But $l'\tilK$ is separable over $l'$ by separability of $\tilK$ over $k$. This argument shows that it suffices only to find $l$ such that $\wt{lK}$ is separable over $l$ because then each $\wt{l'K}$ is separable over $l'$.

Next, we prove the theorem under the additional assumption that $K$ is of degree $p$ over a subfield $L$ such that $\tilL$ is separable over $k$. Let $\ok=k^{1/p^\infty}$ be the perfection of $k$ and set $\oK=\ok K$ and $\oL=\ok L$. If $\oK=\oL$ then already for a $k$-finite subfield $l\subset\ok$ we have that $lK=lL$, and we have shown above that $\wt{lK}=\!\wt{\ lL}=l\tilL$ is separable over $l$. So, we have only to consider the case when $[\oK:\oL]=p$. Since $\oL$ is an Abhyankar field over $\ok$, it is stable and the stability allows us to control the extension $\oK/\oL$ in terms of the value groups and the residue fields. In particular, we can find an element $x\in\oK$ such that either $|x|\notin|\oL^\times|$, or $|x|=1$ and $\tilx$ does not belong to the residue field of $\oL$. In the first case, we simply take a $k$-finite subfield $l\subset\ok$ so that $x\in lK$. Then $|x|$ belongs to $|(lK)^\times|$ but does not belong to $|(lL)^\times|$, hence $e_{lK/lL}=p$ and $f_{lK/lL}=1$. In particular, $\wt{lK}=\!\wt{\ lL}=l\tilL$ is separable over $l$. In the second case, we find a sufficiently large $k$-finite subfield $l\subset\ok$ so that $x\in lK$ and the composite field $l\tilL(\tilx)$ is separable over $l$ (use that $\ok\tilL(\tilx)$ is separable over $\ok$). Since $[\wt{lK}:\wt{lL}]\le p$,
$\wt{lK}$ must coincide with $l\tilL(\tilx)$ and we are done.

Finally, we drop our assumption on $K$. Anyway, $K$ is a finite extension of an Abhyankar field $L$ with separable $\tilL$ (for example, take $L=k(B)$ for an Abhyankar basis $B$ of $K$). It suffices to verify the assertion of the theorem for a finite valued extension of $K$. It follows from the Galois theory of valued fields that after enlarging $K$ we can split the extension $K/L$ as $K/L_s/L_t/L$ where the extension $L_t/L$ is tame, the extension $L_s/L_t$ is Galois and totally wildly ramified (and hence is of degree $p^n$) and the extension $K/L_s$ is purely inseparable. Since $\tilL_t$ is separable over $\tilL$, it is separable also over $k$ and we can safely replace $L$ with $L_t$. We have achieved the situation when the extension $K/L$ is normal of degree $p^n$, and it follows from the theories of $p$-groups and inseparable extensions that $K/L$ splits into a tower  $L=L_0\subset L_1\subset\dots\subset L_n=K$ of extensions of degree $p$. By the particular case proved above there exists a finite purely inseparable extension $l_1/k$ such that $\wt{l_1L_1}$ is separable over $l_1$. Applying the same argument once again, we find $l_2/l_1$ such that $\wt{l_2L_2}$ is separable over $l_2$, so we can proceed inductively until $l=l_n$ is found.
\end{proof}

\appendix
\section{Monoids}\label{monoidsec}

\subsection{Toric monoids}\label{toricsec}

By monoid we mean a set $P$ with a binary operation $\cdot$ or $+$ and a neutral element $1$ or $0$, respectively. All groups and monoids are automatically assumed to be commutative. Usually we will work with multiplicative notation $\cdot,1$, but a few times we will use additive notation $\bfN,\bfZ,\bfQ,\bfR$ or $(M,+)$. We prefer to work with multiplicative notation in order to be consistent with the language of valuations in the paper. The interested reader can easily translate everything to usual additive toric geometry.

Given a monoid $P$, we denote the set of its invertible elements as $P^\times$; it is the largest subgroup of $P$. Also, we use the notation $\oP=P/P^\times$. Any homomorphism from $P$ to a group factors through a universal group which will be denoted $P^\gp$ (the Grothendieck group of $P$). A monoid $P$ is {\em integral} if there is cancellation in $P$, and the latter happens if and only if the map $P\to P^\gp$ is injective. One says that $P$ is {\em fine} if it is finitely generated and integral. By {\em saturation} of a monoid $P$ in a larger monoid $Q\supset P$ we mean the set of all elements $x\in Q$ such that $x^n\in P$ for a positive $n$, and an integral monoid $P$ is called {\em saturated} if it coincides with its saturation in $P^\gp$.

By a {\em toric monoid} $P$ we mean a fine saturated monoid such that $P^\gp$ is a lattice (i.e. $P^\gp$ is torsion free), and by {\em dimension} of $P$ we mean the rank of $P^\gp$. Any such monoid can be described as a cone in $P^\gp$, in the sense that $P=P^\gp\cap P_\bfR$, where $P_\bfR$ is the {\em topological saturation} of $P$ in $P^\gp_\bfR:=P^\gp\otimes_\bfZ\bfR$, i.e. the closure of the saturation of $P$ in $P^\gp_\bfR$. (Note that elements of $P^\gp\otimes_\bfZ\bfR$ are products of real powers of elements of $P$.) Note that $P_\bfR$ is a rational polyhedral cone, i.e. it is the intersection of finitely many rational half spaces. Furthermore, $P$ is sharp (i.e. $\oP=P$, or $P^\times=1$) if and only if the cone is strictly convex. We say that $P$ is of {\em simplicial shape} if the saturation $P_\bfQ$ of $P$ in $P^\gp_\bfQ$ is isomorphic to $(\bfQ_{\ge 0}^n,+)$ (in particular, $P$ is sharp). Note that the latter happens if and only if $P_\bfR$ is a cone over a simplex. Any toric monoid splits non-canonically as $P^\times\oplus \oP\toisom P$. For example, to find a section $\oP\to P$ one can choose a splitting $P^\gp=P^\times\oplus L$ and then $L$ is isomorphic to $\oP^\gp$ and $L\cap P$ is a required copy of $\oP$ in $P$.

A {\em Kummer} homomorphism of toric monoids is an embedding $h\:P\into Q$ such that for any $q\in Q$ there exists $n\ge 1$ with $q^n\in h(P)$. In this case $Q$ is the saturation of $P$ in $Q^\gp$ and the index $[Q^\gp:P^\gp]$ is finite. We call this index the {\em rank} of $h$. Two toric monoids are called {\em isogenous} if they admit Kummer homomorphisms to a third monoid. Note that $M$ is of simplicial shape if and only if it is isogenous to a free monoid $P$ (i.e. $P\toisom\bfN^n$). Indeed, consider the submonoid $P$ in $M$ generated by the primitive elements on the edges of the cone $M_\bfR$.

\subsection{Valuation monoids}\label{valmonsec}
Let $\Lam$ be a "multiplicative" group. We say that a submonoid $\Lamcirc$ is a {\em valuation monoid} of $\Lam$ if $(\Lamcirc)^\gp=\Lam$ and for any element $m\in\Lam$ the monoid $\Lamcirc$ contains at least one element from the set $\{m,m^{-1}\}$. In particular, if $(\Lamcirc)^\times=1$ then $\Lamcirc$ contains exactly one element from any set $\{m,m^{-1}\}$. A valuation monoid is always saturated, in particular, it contains the torsion subgroup $\Lam_\tor$ and studying it reduces to studying the valuation monoid $\Lamcirc/\Lam_\tor$ of $\Lam/\Lam_\tor$. Even if $\Lam$ is a "multiplicative" lattice (i.e. it is finitely generated and torsion free), usually $\Lamcirc$ is not finitely generated, so the following theorem is very useful.

\begin{theor}\label{simth}
Assume that $\Lamcirc$ is a valuation monoid of a "multiplicative" lattice $\Lam$ and $(\Lamcirc)^\times=1$. Then $\Lamcirc$ is a filtered union of its free submonoids with $M^\gp\toisom\Lam$.
\end{theor}

Obviously, it is enough just to prove that any finite subset of $\Lamcirc$ is contained in a free submonoid. Surprisingly enough this is not so simple. We refer to \cite[6.1.30]{GR} for an elementary proof of the theorem. The remaining part of the appendix is not used in the paper. We will make two remarks about the geometry of dual monoids, monoidal desingularization and local uniformization, and the monoidal Riemann-Zariski space. All these objects describe some combinatorial features of their classical analogs. We will treat $\bfN$ as a multiplicative monoid, so we choose a "uniformizer" $\pi\in(0,1)$ and embed $\bfN$ into $\bfR_+^\times$ as
$\pi^\bfN$.

\begin{rem}\label{lasttorrem}
(i) Elements of toric and valuation monoids can be considered as functions on geometric objects corresponding to dual monoids. For example, as a geometric object corresponding to a toric monoid $M$ one can take the dual monoid $M^*=\Hom(M,\pi^\bfN)$ or the dual real cone $M^*_\bfR=\Hom(M,(0,1]^\times)$ or the monoidal spectrum $\Spec(M)$ as defined by Deitmar in \cite{Dei}, i.e. the set of facets of $M^*_\bfR$.

(ii) One can glue global monoidal schemes from such monoidal spectra. If $\Lam$ is a lattice then to any complete rational fan $\Sigma=\{X_\sigma\}_{\sigma\in\Sigma}$ in $\Lam^*_\bfR$ there corresponds a monoidal scheme $X_\Sigma$ glued from $\Spec(M_i)$ with $M_i^\gp=\Lam$. On the level of topological spaces, $X_\Sigma$ is the set of facets $\sigma\in\Sigma$ provided with the quotient topology with respect to the projection $\Lam^*_\bfR\to X_\Sigma$. The stalk $\calO_\sigma=\calO_{X_\Sigma,\sigma}$ consists of the elements $\lam\in\Lam$ with $\lam(\sigma)\le 1$.

(iii) We say that $X_\Sigma$ as above is {\em regular} if all stalks $\calO_\sigma$ are of the form $\bfZ^l\times\bfN^m$. An equivalent condition is that all monoids $\ocalO_\sigma$ are free. This happens if and only if the associated toric variety over a field $k$ is regular. By \cite[Ch. 1, Th. 11]{KKMS} and its proof, any fan $\Sigma$ has a refinement by a regular fan $\Xi$. This claim can be considered as a combinatorial (or monoidal) global desingularization $X_\Xi\to X_\Sigma$, and it implies toric (and toroidal) desingularization. In a sense, this is the "combinatorial part" of the desingularization of varieties. Passing to the dual monoids (the monoids of functions) one easily deduces Theorem \ref{simth}, which is a monoidal analog of local uniformization along a valuation but is formulated in the dual language.
\end{rem}

\begin{rem}\label{lastrzrem}
(i) One can also define a monoidal Riemann-Zariski space $\RZ_\Lam$ to be the set of all valuation monoids of $\Lam$ provided with the natural quasi-compact Zariski (and compact constructible) topology and a sheaf of monoids. We do not give all details but note that on the level of sets it can be described as follows: there is one generic point of height zero; the set of points $M$ of height one can be naturally identified with the unit sphere $S(\Lam^*_\bfR)$ in $\Lam^*_\bfR:=\Hom(\Lam,\bfR_+^\times)$ by normalizing an order preserving functional $\lam_M{\colon}M\to\bfR_+^\times$; if a point $M$ is of height one and the projective coordinates of $\lam_M$ are not linearly independent over $\bfQ$ then $M$ possesses specializations of height two corresponding to rational directions through $M$ in $S(\Lam^*_\bfR)$, and so on for higher heights.

(ii) Alternatively, $\RZ_\Lam$ can be described as the projective limit of all $X_\Sigma$'s, where $\Sigma$ runs through the set of all complete rational fans. Note also that $\RZ_\Lam$ is the set of points of the site (or topos) of $S(\Lam^*_\bfR)$ provided with the $G$-topology of rational polyhedra.

(iii) The monoidal Riemann-Zariski space $\gtX=\RZ_\Lam$ is tightly connected to the graded Riemann-Zariski spaces $\gtY=\bfP_{K/k}$ with $K=k[\Lam]$; see the example after Corollary 2.7 in \cite[\S2]{temred2}. In particular, these spaces are homeomorphic and their sheaves of monoids and graded rings are related by $\calO_\gtY=k[\calO_\gtX]$.
\end{rem}

\section{Relations between local-\'etaleness and \'etaleness}\label{etalesec}
This appendix is due to B. Conrad and the referee. We say that a morphism $f{\colon}Y\to X$ is {\em local-\'etale} at a point $y\in Y$ if the induced morphism $\Spec(\calO_{Y,y})\to\Spec(\calO_{X,f(y)})$ is a localization of an \'etale morphism. Note that the same notion is called essentially \'etale in \cite[$\rm IV_4$, \S18.6.1]{ega}, but our terminology is also common. If, in addition, $k(f(y))\toisom k(y)$ then $f$ is {\em strictly local-\'etale} at $y$.

If $f$ is locally of finite presentation then $f$ is (strictly) local-\'etale at $y$ if and only if it is (strictly) \'etale at $y$. Somewhat surprisingly, one should be very careful with attempts to replace finite presentation with a finite type assumption. We start with a result on the positive side. Although it is not used directly in the paper, it is proved by the same argument that plays the main role in the proof of Theorem \ref{etaleth}.

\begin{propsect}\label{localetale}
Let $X$ be a scheme that locally has a finite number of associated points and let $f{\colon}Y\to X$ be of finite type. Then $f$ is local-\'etale at a point $y\in Y$ if and only if $f$ is \'etale at $y$.
\end{propsect}
\begin{proof}
Only direct implication needs a proof. Shrinking $Y$ around $y$ we can assume that $f$ is quasi-finite. Furthermore, if $X'\to X$ is an \'etale morphism and $y'\in Y\times_XX'$ is a point over $y$ then it suffices to show that $f'=f\times_XX'$ is \'etale at $y'$. Obviously, $f'$ is local-\'etale at $y'$, hence we can replace the initial $X$, $y$, and $Y$ with $X'$, $y'$, and a neighborhood $V'$ of $y'$ in $Y\times_XX'$. We will use this reduction a couple of times until $f$ becomes an isomorphism.

Clearly, we can achieve in this way that $k(x)\toisom k(y)$, where $x=f(y)$. Furthermore, by \cite[$\rm IV_4$, 18.12.1]{ega} we can choose $X'\to X$, $y'$, and $V'$ in such a way that $V'\to X'$ is finite, so we can assume in addition that $f$ is finite. At this stage, $\calO_{X,x}\to\calO_{Y,y}$ becomes a finite strictly-\'etale homomorphism, hence an isomorphism. Therefore, after shrinking $X$ around $x$ (and replacing $Y$ with the preimage) we can also achieve that $f$ is a closed immersion. By our assumption on the finiteness of associated points, we can further shrink $X$ so that $x$ is contained in the closure of any associated point of $X$ (in other words, $x$ lies in all irreducible and embedded components of $X$). At this stage, the closed immersion $f$ becomes an isomorphism because any associated point $\eta$ of $X$ is the image of a point $\varepsilon\in Y$ and $f$ induces an isomorphism $\calO_{X,\eta}\toisom\calO_{Y,\varepsilon}$. This finishes the proof.
\end{proof}

The following examples show that the local finiteness assumption is necessary.

\begin{examsect}
(i) We start with a reduced example with infinitely many irreducible components. Let $A=\prod_{i\in I}k_i$ be an infinite product of fields. It is well known that $X=\Spec(A)$ is a totally disconnected compact space that contains non-discrete points (that correspond to the non-principal ultrafilters in $I$). If $x$ is such a point then $\calO_{X,x}=k(x)$ and the closed immersion $x\into X$ is a local-\'etale morphism of finite type. On the other hand it is not \'etale because the point is not discrete and hence its ideal $m_x\subset A$ is not finitely generated.

(ii) Now, let us construct an irreducible example with infinitely many embedded components. Fix a field $k$ with an infinite subset $I$ (e.g. $\bfQ$ and $\bfN$), and consider the closed immersion of irreducible schemes
$$Y=\Spec(k[y])\to X=\Spec(k[x_i,y_i]_{i\in I}/(x_i^2,x_i(y_i-i)))$$
which is local-\'etale at every point of $\bfA^1_k\setminus I$ but not \'etale anywhere.
\end{examsect}


\begin{thebibliography}{EGA I}

\bibitem[Abh]{Abh}
Abhyankar, S.: {\it Simultaneous resolution for algebraic surfaces}, Amer. J. Math. {\bf 78} (1956), 761--790.

\bibitem[AO]{AO}
Abramovich, D.; Oort, F.: {\it Alterations and resolution of singularities}, Resolution of singularities (Obergurgl, 1997), 39--108, Progr. Math., 181, Birkh\"auser, Basel, 2000.

\bibitem[Ber1]{berbook}
Berkovich, V.: {\em Spectral theory and analytic geometry over non-Archimedean fields}, Mathematical Surveys and Monographs, vol. 33, American Mathematical Society, 1990.

\bibitem[Ber2]{berihes}
Berkovich, V.: {\em \'Etale cohomology for non-Archimedean analytic spaces}, Publ. Math. IHES, {\bf 78} (1993), pp. 7--161.

\bibitem[Ber3]{berform}
Berkovich, V.: {\em Vanishing cycles for formal schemes}, Invent. Math. {\bf 115} (1994), 539--571.

\bibitem[Ber4]{bercontr}
Berkovich, V.: {\em Smooth p-adic analytic spaces are locally contractible}, Invent. Math. {\bf 137} (1999), 1--84.

\bibitem[BGR]{BGR}
Bosch, S.; G\"untzer, U.; Remmert, R.: {\it Non-Archimedean analysis. A systematic approach to rigid analytic geometry}, Springer, Berlin-Heidelberg-New York, 1984.

\bibitem[BL1]{BL0}
Bosch, S.; L\"utkebohmert, W.: {\it Stable reduction and uniformization of abelian varieties. I.}, Math. Ann. {\bf 270} (1985), 349--379.

\bibitem[BL2]{BL}
Bosch, S.; L\"utkebohmert, W.: {\it Formal and rigid geometry} I., Math. Ann., {\bf 295} (1993), 291--317.

\bibitem[BM]{bm}
Bierstone, E.; Milman, P.: {\it Functoriality in resolution of singularities}, Publ. Res. Inst. Math. Sci. {\bf
44} (2008), 609--639.

\bibitem[Bo]{Bo}
Bosch, S.: {\it Eine bemerkenswerte Eigenschaft der formellen Fasern affinoider R\"aume}, Math. Ann. {\bf 229}
(1977), 25--45.

\bibitem[Bou]{Bou}
Bourbaki, N.: {\it Alg\`ebre commutative}, Hermann, Paris, 1961.

\bibitem[Con]{Con}
Conrad, B.: {\it Deligne's notes on Nagata compactification}, J. Ramanujan Math. Soc. {\bf 22} (2007), 205--257.

\bibitem[CP1]{CP1}
Cossart, V.; Piltant, O.: {\it Resolution of singularities of threefolds in positive characteristic. I. Reduction to local uniformization on Artin-Schreier and purely inseparable coverings}, J. Algebra {\bf 320} (2008), 1051--1082.

\bibitem[CP2]{CP2}
Cossart, V.; Piltant, O.: {\it Resolution of singularities of threefolds in positive characteristic II}, J. Algebra {\bf 321} (2009), 1836--1976.

\bibitem[CT1]{ct}
Conrad, B.; Temkin, M.: {\it Non-archimedean analytification of algebraic spaces}, J. Algebraic Geom. {\bf 18} (2009), 731-788.

\bibitem[CT2]{ctdescent}
Conrad, B.; Temkin, M.: {\it Descent for non-archimedean analytic spaces}, in preparation, preliminary version at http://www.math.huji.ac.il/$\sim$temkin/papers/Descent.pdf.

\bibitem[dJ1]{dJ1}
de Jong, J.: {\it Smoothness, semi-stability and alterations}, Publ. Math. IHES, {\bf 83} (1996), 51--93.

\bibitem[dJ2]{dJ2}
de Jong, J: {\it Families of curves and alterations.}, Ann. Inst. Fourier (Grenoble) {\bf 47} (1997), 599--621.

\bibitem[Dei]{Dei}
Deitmar, A: {\it Schemes over $\Bbb F\sb 1$.}  Number fields and function fields---two parallel worlds, 87--100, Progr. Math., 239, Birkh\"auser Boston, Boston, MA, 2005.

\bibitem[Duc]{Duc}
Ducros, A: {\it Toute forme mod\'er\'ement rami\'ee d'un polydisque ouvert est triviale}, preprint, arXiv:[1106.0135].

\bibitem[EGA]{ega}
Dieudonn\'e, J.; Grothendieck, A.: {\it \'El\'ements de g\'eom\'etrie alg\'ebrique}, Publ. Math. IHES, {\bf 4, 8, 11, 17, 20, 24, 28, 32}, (1960-7).

\bibitem[EGA I]{egaI}
Dieudonn\'e, J.; Grothendieck, A.: {\it \'El\'ements de g\'eom\'etrie alg\'ebrique, I: Le langage des schemas}, second edition, Springer, Berlin, 1971.

\bibitem[FK]{FK}
Fujiwara, K.; Kato, F. {\it Rigid geometry and applications}, Moduli spaces and arithmetic geometry, 327--386, Adv. Stud. Pure Math., 45, Math. Soc. Japan, Tokyo, 2006.

\bibitem[GR]{GR}
Gabber, O.; Ramero, L.: {\it Almost ring theory}, Lecture Notes in Mathematics, 1800. Springer-Verlag, Berlin,
2003, vi+307 pp.

\bibitem[Har]{Har}
Hartshorne, R.: {\it Algebraic geometry}, Graduate Texts in Mathematics, No. 52. Springer-Verlag, New York-Heidelberg, 1977. xvi+496 pp.

\bibitem[Hir]{Hir}
Hironaka, H.: {\it Resolution of singularities of an algebraic variety over a field of characteristic zero. I, II}, Ann. of Math. {\bf 79} (1964), 109--203.

\bibitem[Hub]{Hub}
Huber, R.: {\it \'Etale Cohomology of Rigid Analytic Varieties and Adic Spaces}, Aspects of Mathematics, Vol. 30, Vieweg, 1996.

\bibitem[Ill]{Ill}
Illusie L.: {\it On Gabber's refined uniformization}, Talks at the Univ. Tokyo, Jan. 17, 22, 31, Feb. 7, 2008, http://www.math.u-psud.fr/$\sim$illusie/refined\_uniformization3.pdf.

\bibitem[K]{K}
Kato, K.: {\it Logarithmic structures of Fontaine-Illusie}, Algebraic analysis, geometry, and number theory (Baltimore, MD, 1988), 191--224, Johns Hopkins Univ. Press, Baltimore, MD, 1989.

\bibitem[Ka]{Ka}
Kato, F.: {\it Log smooth deformation theory}, Tohoku Math. J. (2) {\bf 48} (1996), 317--354.

\bibitem[Ked1]{Ked1}
Kedlaya, K.: {\it Semistable reduction for overconvergent F-isocrystals, IV: Local semistable reduction at nonmonomial valuations}, Compositio Math. {\bf 147} (2011), 467--523.

\bibitem[Ked2]{Ked2}
Kedlaya, K.: {\it Good formal structures for flat meromorphic connections, II: Excellent schemes}, Journal of the American Mathematical Society {\bf 24} (2011), 183--229.

\bibitem[KK1]{KK1}
Knaf, H.; Kuhlmann, F.-V.: {\it Abhyankar places admit local uniformization in any characteristic}, Ann. Sci. Ecole Norm. Sup. (4) {\bf 38} (2005), 833--846.

\bibitem[KK2]{KK2}
Knaf, H.; Kuhlmann, F.-V.: {\it Every place admits local uniformization in a finite extension of the function field}, Advances in Math. {\bf 221} (2009), 428--453.

\bibitem[KKMS]{KKMS}
Kempf, G.; Knudsen, F.; Mumford, D.; Saint-Donat, B.: {\it Toroidal embeddings}, Lecture Notes in Mathematics, Vol. 339. Springer-Verlag, Berlin-New York, 1973. viii+209 pp.

\bibitem[Kuh]{Kuh}
Kuhlmann, F.-V.: {\it Elimination of Ramification I: The Generalized Stability Theorem}, Trans. Amer. Math. Soc. {\bf 362} (2010), 5697--5727.

\bibitem[Laz]{Laz}
Lazarus, M.: {\it Fermeture int\'egrale et changement de base},  Ann. Fac. Sci. Toulouse Math. (5) {\bf 6} (1984), 103--120.

\bibitem[LMB]{LMB}
Laumon, G.; Moret-Bailly, L.: {\it Champs alg\'ebriques} A Series of Modern Surveys in Mathematics 3rd Series, 39. Springer-Verlag, Berlin, 2000. xii+208 pp.

\bibitem[RG]{RG}
Raynaud, M.; Gruson, L.: {\it Crit\`eres de platitude et de projectivit\'e}, Inv. Math. {\bf 13} (1971), 1--89.

\bibitem[Rydh]{Rydh}
Rydh, D.: {\it Families of zero cycles and divided powers: I. Representability}, preprint, arXiv:[0803.0618].

\bibitem[SGA4]{sga4}
A. Grothendieck, {\it Th\'eorie des topos et cohomologie \'etale des sch\'emas}, Lecture Notes in Math. 269, 270, 305, Springer-Verlag, New York, 1972-73.

\bibitem[Tem1]{temred1}
Temkin, M.: {\em On local properties of non-Archimedean analytic spaces}, Math. Ann., {\bf 318} (2000), 585--607.

\bibitem[Tem2]{temred2}
Temkin, M.: {\it On local properties of non-Archimedean analytic spaces II}, Isr. J. of Math. {\bf 140} (2004), 1--27.

\bibitem[Tem3]{temst}
Temkin, M.: {\it Stable modification of relative curves}, J. Alg. Geom. {\bf 19} (2010), 603--677.

\bibitem[Tem4]{temdes}
Temkin, M.: {\it Desingularization of quasi-excellent schemes in characteristic zero}, Adv. Math., {\bf 219} (2008), 488--522.

\bibitem[Tem5]{temrz}
Temkin, M.: {\it Relative Riemann-Zariski spaces}, Isr. J. of Math., {\bf 185} (2011), 1--42.

\bibitem[Tem6]{nonemb}
Temkin, M.: {\it Functorial desingularization of quasi-excellent schemes in characteristic zero: the non-embedded case}, preprint, arXiv:[0904.1592], to appear in Duke Math. J.

\bibitem[Ura]{Ura}
Urabe, T.: {\it New Ideas for Resolution of Singularities in Arbitrary Characteristic}, preprint, arXiv:[1011.1083].

\bibitem[Va]{Va}
Vakil, R.: {\it Murphy's Law in algebraic geometry: Badly-behaved deformation spaces}, Invent. Math. {\bf 164} (2006), 569--590

\bibitem[Vo]{Vo}
Voevodsky, V.: {\it Homology of schemes I}, Selecta Mathematica {\bf 2} (1996), 111--153.

\bibitem[Zar1]{Zar1}
Zariski, O.: {\it Local uniformization on algebraic varieties}, Ann. of Math., {\bf 41,} (1940), 852--896.

\bibitem[Zar2]{Zar2}
Zariski, O.: {\it Reduction of the singularities of algebraic three dimensional varieties}, Ann. of Math. {\bf 45} (1944), 472--542.

\end{thebibliography}
\end{document}